\documentclass[10pt]{article}

\title{P=W conjectures for character varieties with symplectic resolution}
\author{Camilla Felisetti, Mirko Mauri}
\usepackage[utf8]{inputenc}
\usepackage[T1]{fontenc}
\usepackage{fixltx2e}
\usepackage{amsthm}
\usepackage{geometry}
\usepackage{mathtools}
\usepackage{graphicx}
\usepackage{longtable}
\usepackage{float}
\usepackage{wrapfig}
\usepackage{rotating}
\usepackage{amsmath}
\usepackage{textcomp}
\usepackage{marvosym}
\usepackage{wasysym}
\usepackage{amssymb}
\usepackage[hidelinks]{hyperref}
\usepackage{tikz}
\usepackage{float}
\usepackage{stmaryrd}
\usepackage{enumerate}
\usepackage{caption}
\usepackage{multirow}
\usepackage{fancyhdr}

\usetikzlibrary{matrix}
\usetikzlibrary{shapes.multipart}
\usetikzlibrary{arrows}
\usepackage{amsthm,amsmath,amsfonts,amscd}
\usepackage{tikz-3dplot}
\usetikzlibrary{decorations.markings}
\usetikzlibrary{arrows}
\usepackage{parskip}
\newcommand{\Z}{\mathbb{Z}}
\newcommand{\Aff}{\mathbb{A}}
\newcommand{\Gm}{\mathbb{G}_m}

\newcommand{\CC}{\mathbb{C}}
\newcommand{\QQ}{\mathbb{Q}}

\newcommand{\ZZ}{\mathbb{Z}}
\newcommand{\PP}{\mathbb{P}}

\newcommand{\Attr}{\operatorname{Attr}}
\newcommand{\Repell}{\operatorname{Repell}}
\newcommand{\Gr}{\operatorname{Gr}}

\newcommand{\Gl}{\operatorname{GL}}
\newcommand{\Sl}{\operatorname{SL}}
\newcommand{\PGl}{\operatorname{PGL}}
\newcommand{\Spec}{\operatorname{Spec}}

\newcommand{\alb}{\operatorname{alb}}
\newcommand{\mdoltwisted}{M_{\mathrm{Dol}}^{\mathrm{tw}}(X,G)}
\newcommand{\mbtwisted}{M_{\mathrm{B}}^{\mathrm{tw}}(X,G)}
\newcommand{\MDolG}{{M}_{\mathrm{Dol}}(X,G)}
\newcommand{\MBG}{{M}_{\mathrm{B}}(X,G)}
\newcommand{\MDRG}{{M}_{\mathrm{DR}}(X,G)}
\newcommand{\MHodG}{{M}_{\mathrm{Hod}}(X,G)}
\newcommand{\MHod}{{M}_{\mathrm{Hod}}}
\newcommand{\bMHodG}{\overline{M}_{\mathrm{Hod}}(X,G)}
\newcommand{\bMDolG}{\overline{M}_{\mathrm{Dol}}(X,G)}
\newcommand{\bMDRG}{\overline{M}_{\mathrm{DR}}(X,G)}
\newcommand{\tMHodG}{\widetilde{M}_{\mathrm{Hod}}(X,G)}
\newcommand{\tMBG}{\widetilde{M}_{\mathrm{B}}(X,G)}
\newcommand{\tMDolG}{\widetilde{M}_{\mathrm{Dol}}(X,G)}
\newcommand{\tMDRG}{\widetilde{M}_{\mathrm{DR}}(X,G)}
\newcommand{\pMHodG}{{M}^+_{\mathrm{Hod}}(X,G)}

\newcommand{\MDolGl}{{M}_{\mathrm{Dol}}(X, \Gl_n)}
\newcommand{\MHodGl}{{M}_{\mathrm{Hod}}(X, \Gl_n)}

\newcommand{\MHodGm}{{M}_{\mathrm{Hod}}(X, \Gm)}
\newcommand{\MDolSl}{{M}_{\mathrm{Dol}}(X, \Sl_n)}
\newcommand{\MDolSlone}{{M}_{\mathrm{Dol}}(A, \Sl_n)}
\newcommand{\MDolPGl}{{M}_{\mathrm{Dol}}(X, \PGl_n)}

\newcommand{\MDolSlsm}{{M}^{\mathrm{sm}}_{\mathrm{Dol}}(X, \Sl_n)}
\newcommand{\tMDolGl}{\widetilde{M}_{\mathrm{Dol}}(X, \Gl_n)}
\newcommand{\tMDolSl}{\widetilde{M}_{\mathrm{Dol}}(X, \Sl_n)}
\newcommand{\spmap}{\mathrm{sp}^{!}}

\newcommand{\MBGl}{{M}_{\mathrm{B}}(X,\Gl_n)}
\newcommand{\MBSl}{{M}_{\mathrm{B}}(X,\Sl_n)}
\newcommand{\MBSlone}{{M}_{\mathrm{B}}(A, \Sl_n)}

\newcommand{\tMBGl}{\widetilde{M}_{\mathrm{B}}(X,\Gl_n)}
\newcommand{\tMBSl}{\widetilde{M}_{\mathrm{B}}(X,\Sl_n)}

\newcommand{\tm}{\widetilde{M}}
\newcommand{\tomega}{\widetilde{\Omega}}

\newcommand{\NSl}{\operatorname{Bun}^s}
\newcommand{\NSli}{\operatorname{Bun}^s(C/\iota)}
\newcommand{\NSliss}{\operatorname{Bun}^{ss}(C/\iota)}

\newcommand{\Hom}{\operatorname{Hom}}
\newcommand{\codim}{\operatorname{codim}}
\newcommand{\id}{\operatorname{id}}
\newcommand{\rank}{\operatorname{rank}}
\newcommand{\Tr}{\operatorname{tr}}
\newcommand{\Aut}{\operatorname{Aut}}
\newcommand{\Pic}{\operatorname{Pic}}
\newcommand{\Sing}{\operatorname{Sing}}

\newcommand{\Fix}{\operatorname{Fix}}
\newcommand{\NBun}{N}

\newcommand{\Rcomp}{R}
\newcommand{\Runi}{R}
\newcommand{\Mukai}{M(S, v)}
\newcommand{\Kumm}{K(S, v)}
\newcommand{\KummA}{K(A, v)}

\usepackage{verbatim}

\begingroup
\makeatletter
\@for\theoremstyle:=definition,remark,plain\do{%
\expandafter\g@addto@macro\csname th@\theoremstyle\endcsname{%
\addtolength\thm@preskip\parskip
}%
}
\endgroup
\usepackage{graphicx}
\usepackage[capitalise]{cleveref}
\newtheorem{thm}{Theorem}[section]
\newtheorem{lem}[thm]{Lemma}
\newtheorem{cor}[thm]{Corollary}
\newtheorem{prop}[thm]{Proposition}
\newtheorem{conj}[thm]{Conjecture}
\newtheorem*{mainthm}{Main theorem}
\theoremstyle{definition}
\newtheorem{defn}[thm]{Definition}
\newtheorem{notation}[thm]{Notation}
\newtheorem{exa}[thm]{Example}

\newtheorem{rmk}[thm]{Remark}
\crefname{thm}{Theorem}{Theorems}
\Crefname{thm}{Theorem}{Theorems}
\Crefname{thm}{Theorem}{Theorems}
\Crefname{thm}{Theorem}{Theorems}
\crefname{lem}{Lemma}{Lemmas}
\Crefname{lem}{Lemma}{Lemmas}
\crefname{Conjecture}{Conjecture}{Conjectures}
\Crefname{Conjecture}{Conjecture}{Conjectures}
\crefname{Corollary}{Corollary}{Corollaries}
\Crefname{Corollary}{Corollary}{Corollaries}
\crefname{Claim}{Claim}{Claims}
\Crefname{Claim}{Claim}{Claims}
\crefname{Proposition}{Proposition}{Propositions}
\Crefname{Proposition}{Proposition}{Propositions}
\crefname{Remark}{Remark}{Remarks}
\Crefname{Remark}{Remark}{Remarks}
\crefname{Definition}{Definition}{Definitions}
\Crefname{Definition}{Definition}{Definitions}
\crefname{Example}{Example}{Examples}
\Crefname{Example}{Example}{Examples}
\crefname{Exercise}{Exercise}{Exercises}
\Crefname{Exercise}{Exercise}{Exercises}

\newtheoremstyle{plain2}    
   {}            
   {}            
   {\itshape}    
   {}            
   {\bfseries}   
   {.}           
   {5pt plus 1pt minus 1pt}  
   {{\thmnumber{#1} \thmname{#2}{\thmnote{ (#3)}}}}          

\begin{document}

\maketitle


\abstract{We establish P=W and PI=WI conjectures for character varieties with structural group $\mathrm{GL}_n$ and $\mathrm{SL}_n$ which admit a symplectic resolution, i.e.\ for genus 1 and arbitrary rank, and genus 2 and rank 2. We formulate the P=W conjecture for resolution, and prove it for symplectic resolutions. We exploit the topology of birational and quasi-\'{e}tale modifications of Dolbeault moduli spaces of Higgs bundles. To this end, we prove auxiliary results of independent interest, like the construction of a relative compactification of the Hodge moduli space for reductive algebraic groups, and the projectivity of the compactification of the de Rham moduli space. In particular, we study in detail a Dolbeault moduli space which is specialization of the singular irreducible holomorphic symplectic variety of type O'Grady 6.}


\tableofcontents
\section{Introduction}

Let $X$ be a compact Riemann surface of genus $g$, and let $G$ be a complex reductive algebraic group. The Betti and Dolbeault moduli spaces $\MBG$ and $\MDolG$ are central objects in non-abelian Hodge theory. The \textbf{Betti} moduli space, or $G$-character variety of $X$, is the affine GIT quotient
\begin{align}\label{eq:charactervar} \MBG \coloneqq & \Hom(\pi_1(X), G)\sslash G \nonumber \\
 = & \big\{ (A_1, B_1, \ldots, A_{g}, B_{g}) \in G^{2g} \, \big| \, \prod^{g}_{j=1}[A_j, B_j]=1_{G} \big\}\sslash G.
\end{align}
It parametrises isomorphism classes of semistable representations of the fundamental group of $X$ with value in $G$.

The \textbf{Dolbeault} moduli space $\MDolG$ instead parametrises semistable principal $G$-Higgs bundles with vanishing Chern classes; see \cite{Simpson1994I}. For example, we have that:
\begin{itemize}
    \item a \textbf{$\Gl_n$-Higgs bundle} is a pair $(E, \phi)$ with $E$ vector bundle of rank $n$ and degree $0$, and $\phi \in \Hom(E, E \otimes K_X)$;
    \item a $\Gl_n$-Higgs bundle is an \textbf{$\Sl_n$-Higgs bundle} if in addition the determinant of $E$ is trivial and the trace of $\phi$ vanishes;
    \item a \textbf{$\PGl_n$-Higgs bundle} is an equivalence class of $\Sl_n$-Higgs bundles under tensorization by an $n$-torsion line bundle on $C$.
\end{itemize} 

Despite the different origin of these moduli spaces, there exists a real analytic isomorphism
\[\Psi\colon \MDolG \to \MBG\]
called \textbf{non-abelian Hodge correspondence}; see \cite{Simpson1994} or Section \ref{sec:p=wforresolution}. However, the map $\Psi$ is not an algebraic isomorphism. Indeed, note that the Betti moduli space is an affine variety, while the Dolbeault moduli space admits a projective morphism with connected fibres
\[\chi: \MDolG \to \Aff^{\dim\MDolG /2},\]
 called \textbf{Hitchin fibration}.  
 The purpose of this paper is to study the behaviour in cohomology of the non-abelian Hodge correspondence in view of the P=W conjecture \cite{deCataldoHauselMigliorini2012}. In the rest of the paper we will only consider reductive groups of type $A$, i.e.\ $G=\Gl_n, \Sl_n, \PGl_n$, unless stated otherwise, e.g. in the formulation of the P=W conjectures or in \cref{sec:compactificationHodge}.
 
 One of the main difficulties while studying the cohomology of these moduli spaces is that they are generally singular. To circumvent this issue, it is customary to slightly change the moduli problem as follows. Given an integer\footnote{We omitted the dependence of $\mbtwisted$ and $\mdoltwisted$ on the degree $d$ not to burden the notation too much.} $d$ coprime with the rank $n$ of the group, the twisted Betti moduli space is the GIT quotient 
\[ \mbtwisted \coloneqq
 \big\{ (A_1, B_1, \ldots, A_{g}, B_{g}) \in G^{2g} \, \big| \, \prod^{g}_{j=1}[A_j, B_j]=e^{2\pi id/n}1_{G} \big\}\sslash G.
\]
On the other hand, the twisted version of Dolbeault moduli, denoted $\mdoltwisted$, parametrises semistable pairs $(E, \phi)$, with $E$ vector bundle of rank $n$ and degree\footnote{Note that we recover the untwisted Dobeault moduli space for $d=0$.} $d$, and $\phi \in \Hom(E, E \otimes K_X)$. The technical advantage of working with these twisted moduli spaces is that they are smooth varieties and satisfy a non-abelian Hodge theorem as in the untwisted case; see \cite{HauselThaddeus04}. 
  
  While studying the weight filtration on $H^*(\mbtwisted, \QQ)$, Hausel and Ro\-driguez-Villegas discovered a surprising symmetry, that they called \textbf{curious hard Lefschetz} theorem: there exists a class $\alpha\in H^2(\mbtwisted, \QQ)$ which induces the isomorphisms
\begin{equation}\label{cHL}
\cup {\alpha^k}\colon \Gr^W_{n-2k}H^*(\mbtwisted, \QQ)\xrightarrow{\simeq} \Gr^W_{n+2k}H^{*+2k}(\mbtwisted, \QQ).
\end{equation}
The theorem holds for $G= \Gl_2, \Sl_2$ and $\PGl_2$ by \cite{HauselRodriguez-Villegas2008}, and for $G=\Gl_n$ by \cite{Mellit2019}. 
To explain this phenomenon, de Cataldo, Hausel and Migliorini conjectured that the non-abelian Hodge correspondence should exchange the weight filtration on $H^*(\mbtwisted, \QQ)$ with the perverse (Leray) filtration associated to $\chi$ on $H^*(\mdoltwisted, \QQ)$; see \cref{Def:perverseLeray}. In this way, the curious hard Lefschetz theorem would correspond to the classical relative hard Lefschetz theorem for $\chi$; see \cref{thm:rhl}.

\begin{conj}[P=W conjecture for twisted moduli spaces]
\[P_k H^*(\mdoltwisted, \QQ) = \Psi^* W_{2k}H^*(\mbtwisted, \QQ).\]
\end{conj}

The conjecture holds for $g\geq 2$ and $G=\Gl_2$, $\Sl_2$ and $\PGl_2$ by \cite{deCataldoHauselMigliorini2012}, and for $g=2$ and $G=\Gl_n, \Sl_p$ with $p$ prime by \cite{deCataldoMaulikShen2019, deCataldoMaulikShen2020}. An enumerative approach has been proposed in \cite{CHS2020}, and other P=W phenomena have been studied in
\cite{ShenZhang2018, ShenYin2018, FelisettiShenYin2021, ZhangZili2019, Szabo2018, Szabo2019, NemethiSzabo2020, HKP2019, Harder2019, HLSY2019, Mauri2021}. 
However, P=W phenomena for the original moduli spaces $\MBG$ and $\MDolG$ have not been explored yet. This is then the goal of our paper. 

In the singular case, relative and curious hard Lefschetz theorems fail in general for singular cohomology; see \cref{failurecurioushardLef}. 
Nonetheless, it is known that the relative hard Lefschetz theorem for $\chi$ holds for \textbf{intersection cohomology} $IH^*(\MDolG)$; see Sections \ref{sec:intercohom} and \ref{sec:perverseleray}.
Moreover, de Cataldo and Maulik proved in \cite{deCataldoMaulik2018} that the perverse filtration on intersection cohomology is independent of the complex structure of the curve $X$, exactly as it happens for the weight filtration. Therefore, they conjectured \cite[Question 4.1.7]{deCataldoMaulik2018}.

\begin{conj}[PI=WI conjecture]\label{PI=WIconj} Let $G$ be a complex reductive group. Then  
\[P_k IH^*(\MDolG, \QQ)= \Psi^* W_{2k}IH^*(\MBG, \QQ).\]
\end{conj}

It is also conceivable that one could obtain the P=W conjecture for the singular moduli spaces $\MDolG$ from the previous conjectures.
\begin{conj}[P=W conjecture for singular moduli spaces] \label{P=Wsing} Let $G$ be a complex reductive group. Then  
\[P_kH^*(\MDolG, \QQ) = {\Psi}^* W_{2k}H^*(\MBG, \QQ).\]
\end{conj}

Alternatively, we may also opt for a desingularization of $\MDolG$, and continue to work with singular cohomology. We show that a P=W conjecture for \textbf{symplectic resolution} does hold, i.e.\ for resolutions where a holomorphic symplectic form on the smooth locus of each moduli space extends to a symplectic form on the whole of the resolution. 

To this end, we first show how to lift the non-abelian Hodge correspondence to resolutions of $\MDolG$ and $\MBG$, up to isotopy, according to \cref{LiftNAHC}.
\begin{thm}[\cref{LiftNAHC}] \label{thmintro:liftnonabelian}
Let $G$ be a complex reductive group. Then there exist resolutions of singularities $f_{\mathrm{Dol}}\colon \tMDolG \to \MDolG$ and $f_{\mathrm{B}}\colon \tMBG \to \MBG$, and a diffeomorphism $\widetilde{\Psi}\colon \tMDolG \to \tMBG$, such that the following square commutes:
\begin{center}
\begin{tikzpicture}[baseline=(current  bounding  box.east)]
\node (A) at (0,1.2) {$H^*(\tMDolG, \QQ)$};
\node (B) at (5,1.2) {$H^*(\tMBG, \QQ)$};
\node (C) at (0,0) {$H^*(\MDolG,\QQ)$};
\node (D) at (5,0) {$H^*(\MBG, \QQ)$.};
\draw[<-,thick] (A) -- (C) node [midway,left] {$f_{\mathrm{Dol}}^*$};
\draw[<-,thick] (B) -- (D) node [midway,right] {$f_{\mathrm{B}}^*$};
\draw[<-,thick] (A) -- (B)node [midway,above] {$\widetilde{\Psi}^*$};
\draw[<-,thick] (C) -- (D) node [midway,above] {$\Psi^*$};
\end{tikzpicture}
\end{center}
The resolutions $f_{\mathrm{Dol}}$ and $f_{\mathrm{B}}$ can be taken functorial with respect to smooth algebraic or analytic morphisms, and symplectic if $G=\Gl_n$ or $\Sl_n$ with $(g,n)=(1,n)$ or $(2,2)$.
\end{thm}

\begin{conj}[P=W conjecture for symplectic resolution]\label{P=Wres} Let $G$ be a complex reductive group. Let $\widetilde{\Psi}, f_{\mathrm{Dol}}$ and $f_{\mathrm{B}}$ be the diffeomorphism appearing in 
\cref{thmintro:liftnonabelian}. If $f_{\mathrm{Dol}}$ is a symplectic resolution (if it exists!), or equivalently $f_{\mathrm{B}}$ is so, then
\[P_kH^*(\tMDolG, \QQ) = \widetilde{\Psi}^* W_{2k}H^*(\tMBG, \QQ).\]
\end{conj}

In an earlier version of this paper, we stated the P=W conjecture for resolution without the assumption of the existence of a symplectic resolution, but later the second author proved that 
the hypothesis is indeed essential at least for $G=\Gl_n$ and $\Sl_n$, see \cite[\S 5.6]{Mauri2021}. Recent results suggest that the existence of a holomorphic symplectic form should be a key ingredient for P=W phenomena, see \cite[\S 4.4]{deCataldoHauselMigliorini2013}, \cite{Mellit2019} and \cite[Theorem 1.7]{Harder2019}. 



In this paper, we provide the first evidence for the P=W conjectures in the singular context. 

\begin{mainthm}
Let $G=\Gl_n$ or $\Sl_n$. Suppose that $(g,n)=(1,n)$ or $(2,2)$. Then the following conjectures hold:
\begin{enumerate}
    \item the P=W conjecture;
    \item the PI=WI conjecture;
    \item the P=W conjecture for a symplectic resolution.
\end{enumerate}
\end{mainthm}
Observe that $\MDolGl$ and $\MDolSl$ admit a (unique) symplectic resolution if and only if $(g,n)=(1,n)$ or $(2,2)$; see \cite{BellamySchedler2019} and \cite[Theorem 2.2]{FuNamikawa04}. Under this assumption, the P=W and PI=WI conjectures for $\MDolPGl$ hold too; see \cref{rmk:PGLn}.

The expectation is that the PI=WI conjecture holds even in the absence of a symplectic resolution. The second author has provided first evidence of this fact in \cite[\S 5]{Mauri2021}.
\begin{proof}[Proof of Main Theorem]
We first reduce to $G=\Sl_n$; see \cref{P=WSLimpliesGL}. 

For $g=1$, the P=W and PI=WI conjectures follow from \cref{thm:PI=IWgenusone} and  
\cref{PI=WIisP=Wgenusone}. 
Although not presented in these terms, the proof of the P=W conjecture for the symplectic resolution in $g=1$ is due to \cite{deCataldoHauselMigliorini2013}.

The proof of the conjectures for $M \coloneqq M_{\mathrm{Dol}}(C, \Sl_2)$, with $C$ a curve of genus $2$, takes up most of the paper. We first reduce the P=W conjecture for $M$ and $\widetilde{M}$ to the PI=WI conjecture; see Theorems \ref{P=W and PI=WI}, \ref{cor:splitting} and \ref{P=Wforsingularloci}. Finally, the PI=WI conjecture follows from Theorems \ref{Thm:PWinv and var}, \ref{PI=WIvariant} and \ref{thm:P=Winvar}.
\end{proof}



\textbf{Symplectic resolutions.} The Dolbeault moduli spaces which admit a symplectic resolution appear as specialisation of (a crepant contraction of) of compact hyperk\"{a}hler manifolds as shown in the table. 
\begin{table}[ht]
\centering
\small
\renewcommand{\arraystretch}{1.3}
   \begin{tabular}{ll}
   \multicolumn{1}{}{} \\
Special fibre & symplectic resolution of the general fibre \\ \hline
\multirow{2}{*}{$M_{\mathrm{Dol}}(A, \Gl_n)$} &  Hilbert scheme of $n$ points on a K3 surface \\ & containing the elliptic curve $A$\\
\multirow{2}{*}{$M_{\mathrm{Dol}}(A, \Sl_n)$} &  generalised Kummer variety of dimension $2(n-1)$\\ & associated to the abelian surface $A \times A$\\
$M_{\mathrm{Dol}}(C, \Gl_2)$ & O'Grady 10-dimensional moduli space $\mathrm{OG10}$\\
$M \coloneqq M_{\mathrm{Dol}}(C, \Sl_2)$ & O'Grady 6-dimensional moduli space $\mathrm{OG6}$
    \end{tabular}
    \vspace{0.2 cm}
   \caption{Degenerations of compact hyperk\"{a}hler manifolds to $\MDolG$; see \cref{appendix:deg}. We denote by $A$ and $C$ a compact Riemann surface of genus $1$ and $2$ respectively.}\label{table}
    \end{table}
Even if these degenerations are not strictly used in the proof of Main Theorem, they have been our sources of inspiration. For instance, the proof of the P=W conjecture for $g=1$ is inspired by the description of the cohomology of generalised Kummer varieties in \cite{GottscheSoergel93}. Whilst the alterations in \cref{sec:preliminary} are specializations of those exploited by \cite{MRS18} to determine the Hodge numbers of $\mathrm{OG6}$. We included details about the construction of the degenerations in \cref{appendix:deg} for the interested reader. In the twisted case these degenerations have been exploited in \cite[\S 4]{deCataldoMaulikShen2019} and \cite[\S 4]{deCataldoMaulikShen2020}; see \cref{prop:surjspeci} and \cref{rmk:smoothvssingular} for a bizarre  difference between the behaviour of the degenerations in the smooth and singular cases. Analogous degenerations on the Betti side for $g=1$ have been considered in \cite[\S 5.2 and \S 5.3]{MMS} for the proof of the geometric P=W conjecture. 
\vspace{0.1 cm}



\textbf{Twisted vs untwisted moduli spaces.} Let $G=\Gl_n$ or $\Sl_n$. The known proofs \cite{deCataldoHauselMigliorini2012} and \cite{deCataldoMaulikShen2019} of the P=W conjecture for twisted moduli spaces crucially rely on the fact that $H^*(\mdoltwisted)$ is generated in degree not greater than 4. Further, the generators are K\"{u}nneth components of the second Chern class of a universal Higgs bundle on $\mdoltwisted \times X$, called \textbf{tautological classes}.

In the untwisted case this can fail.
\begin{itemize}
    \item The cohomology ring of $\MDolG$ may not be generated in degree $\leq 4$. For instance, \cref{thm:PoincarePoly} and the second paragraph of the proof of  \cref{prop:Epolynomialinvvar} imply that
    \[H^*(M, \QQ) \simeq \QQ[\alpha, \gamma_j]/(\alpha^3, \gamma_j^2, \alpha\cup \gamma_j),\]
    with $\deg \alpha =2$, $\deg \gamma_j = 6$, and $j=1, \ldots, 16$.
    \item A universal Higgs bundle $\mathbb{E}$ on $\MDolG^{\mathrm{sm}} \times X$ may not exist. Indeed, if $\mathbb{E}$ exists on $M^{\mathrm{sm}} \times C$, then its restriction to the moduli space of semistable vector bundles of rank 2 and degree 0 would be a universal vector bundle, which does not exist by \cite{NR69}.
\end{itemize}

In $g=2$, we fix this problem by constructing a tautological class $\beta$ on a quasi-\'{e}tale cover of $M$, i.e.\ \'{e}tale in codimension one; see \cref{sec:a tautological class}. However, $\beta$ does not descend in cohomology, but as an intersection cohomology class. More precisely, $IH^*(M, \QQ)$ is the $H^*(M,\QQ)$-module  
\[IH^*(M, \QQ) \simeq H^*(M,\QQ) [ 1, \beta ]/\Big(\alpha \cup \beta - \sum^{16}_{j=1}\gamma_j, \alpha^2 \cup \beta, \gamma_j \cup \beta \Big).\]

One may avoid to construct a universal bundle on a quasi-\'{e}tale cover by appealing to the Dolbeault moduli stack, and the class $\beta$ may be interpreted as a Chern class of an orbibundle. 
However, the construction of the universal bundle on the quasi-\'{e}tale cover is interesting in itself; cf \cite[\S 6.1]{HeuLoray19}. 
Note also that the existence of this cover is a special feature of $M$: we show that when $g \geq 2$, $M$ is the only Dolbeault moduli space which admits a non-trivial quasi-\'{e}tale cover; see \cref{sec:quasi-etalecover}.

\begin{figure}[ht]
    \centering
  \begin{tikzpicture}
  \matrix (m) [matrix of math nodes,
    nodes in empty cells,nodes={minimum width=5ex,
    minimum height=1ex,outer sep=-5pt},
    column sep=1ex,row sep=1ex]{ 
                &   &   &    &    &\\               
          6     & 1 &   &    &    & \\
          5     & 0 & 1 &    &    & \\
          4     & 1 & 0 & 16 &    & \\
          3     & 0 & 6 & 0  & 16 & \\
          2     & 1 & 0 & 16 &    & \\
          1     & 0 & 1 &    &    & \\
          0     & 1 &   &    &    &\\
    \quad\strut & 0 & 1 & 2  &  3 &\strut \\};
\draw (m-1-1.east) -- (m-9-1.east) ;
\draw (-2.5,-2) -- (2,-2);
\node[below] at (2.5,-2) {$\Delta$};
\node[above] at (-2.5,2.5) {$p$};
\draw[<->] (2.5,-0.1) arc(-90:90:8pt and 10pt);
\draw (2,0.25)--(4,0.25);
\node[above] at (3.5,0.25) {RHL};
\node[above] at (-0.5, -3.5) {};
\end{tikzpicture}
\captionsetup{singlelinecheck=off}
   \caption[.]{\small The $(p, \Delta)$-entry of the table is the dimension of the graded piece \[\protect \Gr^P_{p}H^{\Delta-p}(\widetilde{M}_{\mathrm{Dol}}(C,\Sl_2), \QQ),\] of the perverse Leray filtration on $H^{*}(\widetilde{M}_{\mathrm{Dol}}(C,\Sl_2), \QQ)$, where $C$ is a compact Riemann surface of genus 2. The sums along the northwest-southeast diagonals give the Betti numbers of $\widetilde{M}_{\mathrm{Dol}}(C,\Sl_2)$. Relative hard Lefschetz accounts for a symmetry of this perverse diamond, namely a reflection about the horizontal axis placed at middle perversity. The P=W conjecture for resolution implies that the sums along the rows are the coefficients of the E-polynomial of $\widetilde{M}_{\mathrm{B}}(C,\Sl_2)$, which is computed in \eqref{EpoltildeM}.}
  \label{fig:my_label}
\end{figure}



\subsection{Outline of the paper} 
\begin{itemize}
    \item In Sections \ref{sec:preliminaries} and \ref{sec:modulispacesforGlSl} we recall basic notions and theorems used throughout the paper.
    \item In \cref{sec:compactificationHodge} we lift the non-abelian Hodge correspondence $\Psi$ to a diffeomorphism $\widetilde{\Psi}$ between the resolutions of the Betti and Dolbeault moduli spaces; see \cref{LiftNAHC}. To this end, we describe an explicit compactification of the Hodge moduli space in \cref{thm:compactification}. Note that $\widetilde{\Psi}$ is the diffeomorphism which appears in the statement of the P=W conjecture for symplectic resolution. As a by-product, we answered a question by Simpson about the projectivity of the compactification of the de Rham moduli space, see \cref{cor:compactificationdeRhamproj}.
    \item In \cref{sec: P=W for SL implies P=W for GL} we show that the P=W conjecture for $\Sl_n$ implies the P=W conjecture for $\Gl_n$.
    \item In \cref{sec:P=W1} we prove the P=W conjectures for $g=1$.
    \item The rest of the paper is devoted to the proof of the P=W conjectures for $M \coloneqq M_{\mathrm{Dol}}(C, \Sl_2)$, with $C$ a curve of genus $2$. We describe the geometry of $M$ in great detail in \cref{sec:preliminary}: its singularities and its symplectic resolution $\widetilde{M}$ in \cref{sec:singularityresolution}; the fixed loci of the $\Gm$-action on $M$ and $\widetilde{M}$ in \cref{sec:torusaction};  the (universal) quasi-\'{e}tale cover $q\colon M_{\iota} \to M$ in \cref{sec:orbibundlemoduli}; a universal Higgs bundle on the smooth locus $M^{\mathrm{sm}}_{\iota}$ of $M_{\iota}$ in \cref{sec:Universalbundle}; the zero fibre of the Hitchin fibration in \cref{sec:nilpotent}.
    \item In \cref{sec:P=W2} we explain the strategy of the proof of the P=W conjecture for $M$. Ultimately, we reduce the proof of the P=W conjectures for $M$ and $\widetilde{M}$ to the PI=WI conjecture for $M$. 
    \item In \cref{sec:The action of the 2-torsion of the Jacobian} we compute the necessary intersection Poincar\'{e} and E-polynomials. 
    \item In \cref{sec:a tautological class} we build a tautological class of perversity $2$ and weight $4$, out of the universal bundle on $M^{\text{sm}}_{\iota}$. This allows to conclude the proof of the PI=WI conjecture for $M$ in \cref{sec:invIH}.
    \item In \cref{appendix:deg} we collected some information about degenerations of compact hyperk\"{a}hler varieties to Dolbeault moduli spaces.
\end{itemize} 


\subsection{Acknowledgements} 
We would like to acknowledge useful conversations with Thorsten Beckmann, Simone Chiarello, Mark de Cataldo, Peter Gothen, Isabell Hellmann, Daniel Huybrechts, Luca Migliorini, Giovanni Mongardi, Mircea Mustaţă, André Oliveira, Johannes Schmitt, Junliang Shen, Andras Szenes, Michael Temkin, Marco Trozzo. In particular, we are grateful to Mark de Cataldo for useful discussions on the projectivity of the compactifications of de Rham and Dolbeault moduli spaces in \cref{thm:compactification}, and to Michael Temkin for suggesting \cref{lem:analalg}.
Finally, we are grateful to the anonymous referees for many useful suggestions and comments. 

Camilla Felisetti is supported by the University of Geneva, the Swiss National Science Foundation grants 17599 and 156645, as well as the NCCR SwissMAP. Mirko Mauri is supported by Max Planck Institute for Mathematics in Bonn, and University of Michigan. We would like to thank the University of Geneva and Max Planck Institute for supporting our reciprocal visits during the preparation of this paper.


\section{Preliminaries}\label{sec:preliminaries}
In this section we introduces preliminary notions and results which will be useful throughout the paper. For further details we refer to 
\cite{BeilinsonBernsteinDeligne1981, GM80, GM81, deCataldoMigliorini05}.

When omitted, the coefficients of  (intersection) cohomology are intended to be rational.

\subsection{Perverse sheaves}
An algebraic variety $X$ is an irreducible separated scheme of finite type over $\CC$. Denote by $D^b_c(X)$ the bounded derived category of $\QQ$-constructible complexes on $X$. Let $\mathbb{D}\colon D^b_c(X)\rightarrow D^b_c(X)$ be the Verdier duality functor. 
The full subcategories 
\begin{align*}
  {^{\mathfrak{p}}}D^b_{\leq 0}(X) &\coloneqq \left \lbrace K^*\in D^b_c(X) \mid \dim \mathrm{Supp}(\mathcal{H}^j(K^*))\leq -j \right \rbrace\\
{^{\mathfrak{p}}}D^b_{\geq 0}(X) &\coloneqq\left \lbrace K^* \in D^b_c(X) \mid \dim \mathrm{Supp}(\mathcal{H}^j(\mathbb{D}K^*))\leq -j \right \rbrace
\end{align*}
determine a $t$-structure on $D^b_c(X)$, called perverse $t$-structure. The heart \[ \mathrm{Perv}(X) \coloneqq {^{\mathfrak{p}}}D^b_{\leq 0}(X)\cap {^{\mathfrak{p}}}D^b_{\geq 0}(X)\] of the $t$-structure is the abelian category of \textbf{perverse sheaves}.  The truncation functors are denoted ${^{\mathfrak{p}}}\tau_{\leq k}\colon D^b_c(X)\to {^{\mathfrak{p}}}D^b_{\leq k}(X)$, ${^{\mathfrak{p}}}\tau_{\geq k}\colon D^b_c(X)\to {^{\mathfrak{p}}}D^b_{\geq k}(X)$, and the perverse cohomology functors are 
\[{^{\mathfrak{p}}} \mathcal{H}^{k} \coloneqq {^{\mathfrak{p}}}\tau_{\leq k} {^{\mathfrak{p}}}\tau_{\geq k}\colon D^b_c(X) \to \mathrm{Perv}(X).\]

\begin{defn}
Let $K^*$ be a complex in $D^b_c(X)$. The cohomology $H^d(X, K^*)$ is endowed with the \textbf{perverse filtration} defined by
\[
P_k H^d(X, K^*) = \mathrm{Im} \{H^d(X,{^{\mathfrak{p}}}{\tau}_{\leq k}K^*)\rightarrow H^d(X,K^*)\}.
\]
\end{defn}

\subsection{Intersection cohomology}\label{sec:intercohom}
The category $\mathrm{Perv}(X)$ is abelian, artinian, and its simple objects are the intersection cohomology complexes.

\begin{defn}[Intersection cohomology complex]
 Let $L$ be a local system on a smooth Zariski-dense open subset $U\subseteq X$. 
The \textbf{intersection cohomology complex} $IC_X(L)$ is a complex of sheaves in $D^b_c(X)$ which is uniquely determined up to isomorphism by the following conditions:
\begin{itemize}
	\item $IC_{X}(L)|_{U}\simeq L[\dim X]$;
	\item $\dim \mathrm{Supp}\mathcal{H}^j(IC_{X}(L))<-j$, for all $j>-\dim X$;
	\item $\dim\mathrm{Supp}\mathcal{H}^j(\mathbb{D}IC_{X}(L))<-j$, for all $j>-\dim X$.
\end{itemize}
\end{defn}
When $L=\underline{\QQ}_{X^{\mathrm{sm}}}$, i.e.\ the constant sheaf on the smooth locus of $X$, we just write $IC_X$ for $IC_X(\underline{\QQ}_{X^{\mathrm{sm}}})$. Further, if $X$ has at worst quotient singularities, then $IC_X\simeq \underline{\QQ}_{X}[\dim X]$. 

\begin{defn}[Intersection cohomology]
The \textbf{intersection cohomology} of $X$ with coefficient in $L$ is its (shifted) cohomology  
	\[IH^*(X,L)=H^{*-\dim X}(X,IC_X(L)).\]
Analogously, the intersection cohomology of $X$ with compact support and coefficients in $L$ is
	$IH_c^*(X,L)=H^{*-\dim X}(X,\mathbb{D}IC_X(L))$. For further details, we refer the interested reader to \cite{KW2006}.
	\end{defn}
	
There is a natural morphism \[H^*(X)\rightarrow IH^*(X),\] which is an isomorphism when $X$ has at worst quotient singularities. This morphism endows $IH^*(X)$ with the structure of $H^*(X)$-module, but in general intersection cohomology has no ring structure or cup product. 

Moreover, the groups $IH^*(X)$ are finite dimensional, satisfy Mayer-Vietoris theorem and K\"{u}nneth formula. Although they are not homotopy invariant, they satisfy analogues of Poincaré duality, i.e.\ $IH^*(X) \simeq IH^{2\dim X -*}_c(X)^{\vee}$, and hard Lefschetz theorem. 
They also carry a mixed Hodge structures. 

\begin{defn}[Mixed Hodge structure] The mixed Hodge structure $(V, F^{*}, W_*)$ is the datum of
\begin{itemize}
        \item a $\QQ$-vector space $V$,
		\item an increasing filtration $W_{*}$ on $V$, called \textbf{weight filtration},
		\item a decreasing filtration $F^{*}$ on $V \otimes \CC$, called Hodge filtration,
	\end{itemize}
such that the graded pieces $\Gr_k^W V \coloneqq W_kV/W_{k-1}V$ admit a pure Hodge structure of weight $k$, induced by $F^{*}$ on $\Gr_k^W V \otimes \CC$. 

An element $v \in V$ has \textbf{weight} k if $v \in W_kV$ but $v \not\in W_{k-1}V$.
\end{defn}

\begin{defn}[E-polynomial] The \textbf{E-polynomial} of $X$ is an additive function on the category of separated $\CC$-schemes of finite type given by
\[E(X) = \sum_{p,q,d}(-1)^d \dim ( \Gr^W_{p+q} H^d_{c}(X, \CC))^{p,q} u^pv^q.\]
The additivity means that if $Z \subset X$ is a closed subscheme, then $E(X)=E(X^{\mathrm{red}})=E(X \setminus Z)+ E(Z)$.

Analogously, we define the \textbf{intersection E-polynomial} as
\[IE(X) = \sum_{p,q,d}(-1)^d \dim ( \Gr^W_{p+q} IH^d_{c}(X, \CC))^{p,q} u^pv^q.\]
\end{defn}

Note however that the intersection $E$-polynomial is not an additive function, due to the fact that in general the restriction to a closed subscheme $Z\subset X$ of $IC_X$ is not $IC_Z$.

\subsection{Decomposition theorem}
In this section we recall in brief the statement of the decomposition theorem for semismall maps.
\begin{defn}
A morphism of algebraic varieties $f\colon X \to Y$ is \textbf{semismall} if $\dim X \times_Y X \leq \dim X$.
\end{defn}
A stratification of $f$ is a collection of finitely many locally closed subsets $Y_k$ such that $f^{-1}(Y_k) \to Y_k$ are topologically locally trivial fibrations. A stratum $Y_k$ is \textbf{relevant} if $2 \dim f^{-1}(Y_k) - \dim(Y_k)=\dim X$.

\begin{thm}[Decomposition theorem for semismall maps]\label{DecThm}
Let $f\colon X \to Y$ be a proper algebraic semismall map from a smooth variety $X$. Then there exists a canonical isomorphism
\[\mathrm{R}f_* \underline{\QQ}_X [\dim X] \simeq \bigoplus_{Y_k} IC_{\overline{Y}_k}\big(\mathrm{R}^{\dim X - \dim Y_k}f_*\underline{\QQ}_{f^{-1}(Y_k)}\big),\]
where the summation index runs over all the relevant strata of a stratification of $f$.
\end{thm}
\subsection{Perverse Leray filtration}\label{sec:perverseleray}
Let $\chi\colon X\rightarrow Y$ be a projective morphism of algebraic varieties of relative dimension $r$. Set $r(\chi) \coloneqq \dim X \times_Y X - \dim X$.
\begin{defn}\label{Def:perverseLeray}
The \textbf{perverse Leray filtration associated to $\chi$} is the (shifted) perverse filtration on the cohomology of the complex $\mathrm{R}\chi_*IC_X$
\[P_k IH^*(X)=P_k H^{*-(\dim X -r(\chi))}(Y,\mathrm{R}\chi_*IC_X[\dim X -r(\chi)]).\]
\end{defn}

When $Y$ is affine, de Cataldo and Migliorini provided an equivalent geometric description of the perverse Leray filtration.
Assume for simplicity that 
$\dim X = 2\dim Y = 2 r(\chi)$. Let $\Lambda^k\subset Y$ be a general $k$-dimensional linear section of $Y\subset \Aff^N$.
\begin{thm}[Flag filtration]\emph{\cite[Theorem 4.1.1]{deCataldoMigliorini2010}}\label{thm:kercharacterizationp} 
$$ P_k IH^d(X)=\mathrm{Ker}\left\lbrace IH^d(X)\rightarrow IH^d(\chi^{-1}(\Lambda^{d-k-1}))
\right \rbrace.$$
\end{thm}
This means that the class $\eta \in IH^d(X)$ belongs to $P_kIH^d(X)$ if and only if its restriction to $\chi^{-1}(\Lambda^{d-p-1})$ vanishes, i.e.\ $\eta|_{ \chi^{-1}(\Lambda^{d-p-1})}=0$. 


Most remarkably, the perverse Leray filtration satisfies the relative hard Lefschetz theorem. 
\begin{thm}[Relative hard Lefschetz]\label{thm:rhl}
Let $\chi\colon X\rightarrow Y$ be a proper map of algebraic varieties, and let $\alpha\in H^2(X)$ be the first Chern class of a relatively ample line bundle. Then there exists an isomorphism 
$$\alpha^i\colon  \Gr^P_{r-k}IH^*(X)\rightarrow \Gr^P_{r+k}IH^{*+2k}(X).$$
\end{thm}


\section{Lifting the non-abelian Hodge correspondence}\label{sec:p=wforresolution}
Let $X$ be a compact Riemann surface, and fix a complex reductive algebraic group $G$. 
The first cohomology group $H^1(X, G)$ comes in various incarnations (cf \cite{Simpson1994I} and \cite{Simpson1994}):
\begin{itemize}
    \item the \textbf{Betti} moduli space $\MBG$, also named \textbf{character variety}, parametrising semistable representations of the fundamental group of $X$ with value in $G$;
    \item the \textbf{Dolbeault} moduli space $\MDolG$ of semistable principal $G$-Higgs bundles with vanishing Chern classes;
    \item the \textbf{De Rham} moduli space $\MDRG$ of semistable principal $G$-bundles with an integrable connection.
\end{itemize}
All these moduli spaces are homeomorphic to each other. The Riemann-Hilbert correspondence yields a complex analytic isomorphism
\begin{equation}\label{RiemannHilbert}
    \MDRG^{\mathrm{an}} \simeq \MBG^{\mathrm{an}}.
\end{equation}
There exists an algebraic fibration (real analytically trivialisable)
\begin{equation}\label{Hodge}
    \lambda\colon  \MHodG \to \Aff^1,
\end{equation}
whose fibers are moduli spaces of semi-simple principal $G$-bundles with $\lambda$-con\-nections; see \cite{Simpson1997}. Hence, the fibre over $0$ is $\MDolG$, and the fibres over $\lambda \neq 0$ are isomorphic to $\MDRG$. The space $\MHodG$ is called \textbf{Hodge} moduli space. In particular, a continuous trivialization $\MHodG^{\mathrm{top}}\simeq \MDolG \times \Aff^1$ gives the homeomorphism 
\begin{equation}\label{trivialization}
    \MDolG^{\mathrm{top}} \simeq \MDRG^{\mathrm{top}}.
\end{equation}
The \textbf{non-abelian Hodge correspondence}
\[\Psi\colon  \MDolG^{\mathrm{top}} \to \MBG^{\mathrm{top}}\]
is the composition of the maps (\ref{RiemannHilbert}) and (\ref{trivialization}) for a choice of a preferred real analytic trivialization; see \cite{Simpson1997} for details.

\subsection{Compactification of Hodge moduli spaces}\label{sec:compactificationHodge}
The Hodge moduli space $\MHodG$ admits a partial compactification, relative to the morphism $\lambda\colon  \MHodG \to \Aff^1$. We obtain it as $\Gm$-quotient of the total space of the degeneration of $\MHodG$ to the normal cone of $\lambda^{-1}(0) \simeq \MDolG$. The construction is an extension to the singular case of \cite[Lemma 6.1]{HauselThaddeus03} or \cite[Theorem 7.2.1]{HauselLetellierRodriguez-Villegas2011}. 

To this end, we shall use the following results by Simpson.
\begin{prop}\emph{\cite[Theorem 11.2]{Simpson1997} 
}\label{geometricquotient}
Let $Z$ be a variety over the variety
$S$, endowed with a $\Gm$-action covering the trivial $\Gm$-action on $S$. Assume that $Z/S$ carries a relatively ample line
bundle admitting a $\Gm$-linearisation. Assume that the fixed point set $\Fix(Z) \subseteq Z$ is proper over
$S$, and that for any $z \in Z$ the limit $\lim_{t \to 0} t \cdot z$ exist in $Z$. Let $U \subset Z$ be the subset of points $z$ such that the limit $\lim_{t \to \infty} t \cdot z$ does
not exist. Then $U$ is open in $Z$ and there is
a universal geometric quotient $Z/\Gm$. This quotient is  is separated and proper over $S$.
\end{prop}
\begin{thm}[Partial compactification of the Hodge moduli space]\label{thm:compactification}
There exists a projective morphism 
\[\bar{\lambda}\colon  \bMHodG \to \Aff^1\]
which is a relative compactification of the morphism $\lambda$.
\end{thm}
\begin{proof}
$\MHodG$ is endowed with the $\Gm$-action 
\[t \cdot (E, \nabla_{\lambda}) \to (E, t \nabla_{\lambda})\]
covering the standard $\Gm$-action on $\Aff^1$, namely $t \cdot \lambda =t\lambda$.
Endow $\Aff^2$ with the $\Gm$-action given by $t \cdot (x,y)=(x, ty)$. The morphism $\Aff^2 \to \Aff^1$, given by $(x, y) \mapsto xy$, is $\Gm$-equivariant. Therefore the fibre product $\MHodG \times_{\Aff^1}\Aff^2$ is endowed with a $\Gm$-action. We summarize the maps constructed in a diagram: note that the subscripts indicates the coordinatization chosen for the affine spaces.
\begin{center}
\begin{tikzpicture}
\node (A) at (0,1.2) {$\MHodG \times_{\Aff^1}\Aff^2$};
\node (B) at (3,1.2) {$\MHodG$};
\node (C) at (0,0) {$\Aff^2_{x,y}$};
\node (D) at (3,0) {$\Aff^1_{\lambda}$};
\node (E) at (0,-1.2) {$\Aff^1_{x}$};
\node (F) at (5,0) {$(x,y)$};
\node (G) at (7,0) {$xy$};
\node (H) at (5,-1.2) {$x$};
\draw[->,thick] (A) -- (C);
\draw[->,thick] (B) -- (D) node [midway,right] {$\lambda$};
\draw[->,thick] (A) -- (B) node [midway,above] {};
\draw[->,thick] (C) -- (D) node [midway,above] {};
\draw[->,thick] (C) -- (E);
\draw[|->,thick] (F) -- (G);
\draw[|->,thick] (F) -- (H);
\node at (-1.8,0) {$\lambda'$};
\draw[<-, thick] (-0.3,-1.2) arc(270:135:35pt and 35pt);
\end{tikzpicture}
\end{center}

Choose a $\lambda'$-ample line bundle $\mathcal{L}'$ on $\MHodG \times_{\Aff^1}\Aff^2$ admitting a $\Gm$-linearization (which exists since $\MHodG \times_{\Aff^1}\Aff^2$ is normal and because of \cite[Corollary 1.6]{MumfordFogartyKirwan1994}).
Let $\chi(X, G)\colon  \MDolG \to \Aff^{\dim \MDolG/2}$ be the Hitchin's proper map for $\MDolG$; see \cite[p.22]{Simpson1994}. The fixed locus is contained in 
\[\chi(X, G)^{-1}(0) \times \{ y = 0\} \subset \MDolG \times \{ y = 0\} \subset \MHodG \times_{\Aff^1} \Aff^2,\]
so it is proper over $\Aff^1_x$. By \cref{geometricquotient}, there exists a universal geometric quotient
\[
\bMHodG \coloneqq (\MHodG \times_{\Aff^1}\Aff^2 \setminus (\chi(X, G)^{-1}(0) \times \Aff^1_x))/\Gm
\]
and a proper morphism $\bar{\lambda}\colon  \bMHodG \to \Aff^1_x$.

$\bMHodG$ contains an open subset isomorphic to $\MHodG$, given by the $\Gm$-quotient of 
\[
(\MHodG \times_{\Aff^1}\Aff^2)\times_{\Aff^1_y} (\Aff^1_y \setminus \{0\}) \simeq \MHodG \times \Gm.
\]

We show now that the morphism $\bar{\lambda}$ is projective. 
Let $\partial\MHod \coloneqq \bMHodG \setminus \MHodG$ be the Cartier boundary divisor. By \cite[Theorem 2.3]{DN89} (or \cite[Proof of Proposition 3.2.2]{deCataldo2020}), a power of the line bundle $\mathcal{L}'$ descends to a line bundle $\mathcal{L}$ on $\bMHodG$. We claim that the line bundle $\mathcal{L}\otimes \mathcal{O}(m\cdot \partial\MHod)$ is ample for $m \gg 0$.

To this end, observe that $\bar{\lambda}^{-1}(0) \eqqcolon \bMDolG$ coincides with the projective compactification of $\MDolG$ constructed in \cite[Theorem 3.1.1.(1)]{deCataldo2020}. Let $\bar{\chi}: \bMDolG \to \overline{A}$ be also the projective compactification of the Hitchin morphism constructed in \cite[Theorem 3.1.1.(2)]{deCataldo2020}. The restriction of $\mathcal{L}$ to $\bar{\lambda}^{-1}(0)$ is $\bar{\chi}$-ample by \cite[Proposition 3.2.2]{deCataldo2020}, while the restriction of $\partial M_{\mathrm{Hod}}$ is the pullback of an ample divisor on $\overline{A}$. Therefore, the line bundle $\mathcal{L}\otimes \mathcal{O}(m \cdot \partial\MHod)$ is ample for $m \gg 0$ when restricted to $\bMDolG$. By the openness of ampleness \cite[Theorem 1.2.17]{Lazarsfeld2004}, it is $\bar{\lambda}$-ample in a neighbourhood of $\bar{\lambda}^{-1}(0)$. But $\bMHodG \setminus \bar{\lambda}^{-1}(0)$ is isomorphic to the trivial product $\bMDRG \times (\mathbb{A}^1 \setminus \{0\})$, where the first factor is Simpson's compactification of $\MDRG$; see \cite[\S 11]{Simpson1997}. Therefore, we conclude that $\mathcal{L}\otimes \mathcal{O}(m\cdot \partial\MHod)$ is $\bar{\lambda}$-ample for $m\gg0$.
\end{proof}
Incidentally note that \cref{thm:compactification} answers the question about the projectivity of the compactification of the de Rham moduli space risen in \cite[p.268]{Simpson1997} and \cite[Remark 3.1.2]{deCataldo2020}.
\begin{cor}\label{cor:compactificationdeRhamproj}
Simpson's compactification $\bMDRG$ is projective.
\end{cor}
We now study the local geometry of the morphism of $\bar{\lambda}$.

\begin{prop}\label{blambdaloctriv}
The morphism $\bar{\lambda}$ is locally analytically trivial, i.e.\ for any $p \in \bMHodG$ over $\lambda_p \in \Aff^1$ there exist analytic neighbourhoods $p \in U_p \subseteq \bMHodG$ and $p \in V_p \subseteq \bMHodG_{\lambda}$ such that $U_p\simeq V_p \times \mathbb{D}$, with $\mathbb{D}$ a disk in $\Aff^1$, and $\lambda$ corresponds to the second projection $V_p \times \mathbb{D} \to \mathbb{D}$. 
\end{prop}
\begin{proof}
The $\Gm$-action on $\MHodG$ extends to $\bMHodG$, and so
\[\bMHodG|_{\Aff^1_\lambda \setminus \{0\}} \simeq \bMDRG \times \Gm;\]
see also \cite[p.232]{Simpson1997}. By \cite[Theorem 9.1]{Simpson1997}, $\lambda$ is locally analytically trivial. Therefore, it is enough to show that $\bar{\lambda}$ is locally analytically trivial at $p \in \bar{\lambda}^{-1}(0) \setminus {\lambda}^{-1}(0)$. 

 Let $p' \in \MHodG \times_{\Aff^1}\Aff^2$ be a lift of $p$. Since $\lambda$ is locally analytically trivial, so $\lambda'$ is. Following the proof of \cite[Lemma 3.5.1]{deCataldo2020}, we can choose a transverse slice to the $\Gm$-orbit though $p'$, locally isomorphic to an affine variety $N_{p'} \times \Aff^1_x$, such that $\bMHodG$ is locally isomorphic at $p$ to $N_{p'}/\operatorname{Stab}(p') \times \Aff^1_x$, and $\bar{\lambda}$ is the projection onto the second factor. As a result, we obtain that $\bar{\lambda}$ is locally analytically trivial.
\end{proof}

In \cref{LiftNAHC} we show that there exists a diffeomorphism $\widetilde{\Psi}$ which lifts the isomorphism 
\[\Psi^*\colon  H^*(\MBG) \to H^*(\MDolG)\] 
to an isomorphism between the cohomology of the resolution spaces. 

To this purpose we recall that for any noetherian quasi-excellent generically reduced scheme $X$ over $\Spec(\QQ)$ there exists a resolution of singularities $\mathcal{R}(X) \to X$ functorial with respect to regular morphism $X' \to X$, in the sense that $\mathcal{R}(X')$ is isomorphic to $\mathcal{R}(X) \times_{X} X'$. See \cite{Temkin2012} for further details and the definition of quasi-excellent schemes and regular morphisms. Here we just mention that by definition, if $X$ is excellent, then the completion morphism $\widehat{X}_x \coloneqq \Spec \widehat{\mathcal{O}}_{X,x} \to X$ is regular for any closed point $x \in X$. In \cite[Theorem 5.2.2]{Temkin2012}, Temkin showed also that quasi-compact analytic spaces admit functorial resolutions compatible with smooth analytic morphism. The following lemma is implicit in \cite{Temkin2012}, and it has been kindly communicated to us by Temkin. For clarity, we distinguish the complex algebraic variety $X$ from its complex analytification $X^{\mathrm{an}}$, but omit the difference elsewhere in the paper.
\begin{lem}\label{lem:analalg}
If $X$ is a complex algebraic variety, then the analytification of the algebraic functorial resolution is biholomorphic to the analytic functorial resolutions of $X^{\mathrm{an}}$, i.e. 
$\mathcal{R}(X)^{\mathrm{an}} \simeq \mathcal{R}(X^{\mathrm{an}})$.
\end{lem}
\begin{proof} Without loss of generality suppose that $X=\Spec(B)$ is affine.
We briefly recall Temkin's construction of the analytic functorial resolution; see \cite[Theorem 5.2.2]{Temkin2012}. Take a covering of $X^{\mathrm{an}} = \bigcup_{i} X_i$ by Stein compact domains (e.g. embed locally $X^{\mathrm{an}}$ in a complex affine space and take intersections of $X^{\mathrm{an}}$ with closed polydiscs). The ring of functions $A_i \coloneqq \mathcal{O}^{\mathrm{an}}_{X}(X_i)$ is excellent, and the functorial resolution of $\Spec A_i$ glue to the analytic functorial resolution $\mathcal{R}(X^{\mathrm{an}})$. Since $B$ and $A_i$ are excellent, the completion morphism $B \to \widehat{\mathcal{O}}_{X, x}$ and $A_i \to \widehat{\mathcal{O}}^{\mathrm{an}}_{X, x}$ are regular, so the algebraic and the functorial resolutions $\mathcal{R}(X)$ and $\mathcal{R}(X)^{\mathrm{an}}$ are compatible with completions. Now, since $\widehat{\mathcal{O}}_{X, x}\simeq \widehat{\mathcal{O}}^{\mathrm{an}}_{X, x}$, we have $\mathcal{R}(\widehat{X}_x)\simeq \mathcal{R}(\widehat{X}^{\mathrm{an}}_x)$. By functoriality, we obtain that
\[\mathcal{R}(X) \times_X \widehat{X}_x \simeq \mathcal{R}(\widehat{X}_x)\simeq \mathcal{R}(\widehat{X}^{\mathrm{an}}_x) \simeq \mathcal{R}(X^{\mathrm{an}}) \times_X \widehat{X}_x\]
for any closed point $x \in X$. Hence, $\mathcal{R}(X)^{\mathrm{an}} \simeq \mathcal{R}(X^{\mathrm{an}})$.
\end{proof}

\begin{cor} A biholomorphism $f\colon X' \to X$ between complex algebraic varieties (not necessarily algebraizable) lifts to a biholomorphism $\mathcal{R}(f) \colon \mathcal{R}(X') \to \mathcal{R}(X)$ between their functorial resolutions, which gives a fibre product square.
\begin{equation*}
\begin{tikzpicture}[baseline=(current  bounding  box.east)]
\node (A) at (0,1.2) {$\mathcal{R}(X')$};
\node (B) at (5,1.2) {$\mathcal{R}(X)$};
\node (C) at (0,0) {$X'$};
\node (D) at (5,0) {$X$};
\draw[->,thick] (A) -- (C) node [midway,left] {};
\draw[->,thick] (B) -- (D) node [midway,right] {};
\draw[->,thick] (A) -- (B)node [midway,above] {$\mathcal{R}(f)$};
\draw[->,thick] (C) -- (D) node [midway,above] {$f$};
\end{tikzpicture}
\end{equation*}
\end{cor}
\begin{proof}
By functoriality in the complex analytic category, $\mathcal{R}(X)\times_{X} X'$ is an analytic functorial resolution, so biholomorphic to $\mathcal{R}(X)^{\mathrm{an}}$ by \cref{lem:analalg}.
\end{proof}

\begin{lem}\label{lem:symplfunct}
Let X be a normal locally $\QQ$-factorial\footnote{This means that for any closed point $x \in X$ the analytic local ring $\mathcal{O}^{\text{an}}_{X,x}$ are $\QQ$-factorial, that is some multiple of every Weil divisor is Cartier.} complex variety. Suppose that $X$ admits a symplectic resolution $f \colon Y \to X$ with an irreducible exceptional divisor, obtained by blowing-up the singular locus. Then $f$ is functorial.
\end{lem} 
\begin{proof}
By \cite[Theorem 2.2]{FuNamikawa04}, any symplectic resolution of $X$ is isomorphic to $f$. Let $h \colon X' \to X$ be any smooth morphism. The blow-up $Y' \to X'$ of the singular locus of $X'$ is smooth and symplectic since $h$ is smooth.  Then $X'$ satisfies all the hypotheses of \cref{lem:symplfunct} with symplectic resolution $Y'' \coloneqq Y \times_{X} X'$, so  $Y'=Y''$ by \cite[Theorem 2.2]{FuNamikawa04}, i.e. the resolution is functorial for smooth morphisms, and also for regular morphism following \cite[Theorem 1.2, Corollary 4.6]{BMT2011}.
\end{proof}

\begin{thm}[Lift of the non-abelian Hodge correspondence $\Psi$]\label{LiftNAHC}
There exist resolutions of singularities $f_{\mathrm{Dol}}\colon  \tMDolG \to \MDolG$ and $f_{\mathrm{B}}\colon  \tMBG \to \MBG$, and a diffeomorphism \[\widetilde{\Psi}\colon  \tMDolG \to \tMBG,\]
such that the square 
\begin{equation}\label{1square:Psicommsquarecoh}
\begin{tikzpicture}[baseline=(current  bounding  box.east)]
\node (A) at (0,1.2) {$\tMDolG$};
\node (B) at (5,1.2) {$\tMBG$};
\node (C) at (0,0) {$\MDolG$};
\node (D) at (5,0) {$\MBG$};
\draw[->,thick] (A) -- (C) node [midway,left] {$f_{\mathrm{Dol}}$};
\draw[->,thick] (B) -- (D) node [midway,right] {$f_{\mathrm{B}}$};
\draw[->,thick] (A) -- (B)node [midway,above] {$\widetilde{\Psi}$};
\draw[->,thick] (C) -- (D) node [midway,above] {$\Psi$};
\end{tikzpicture}
\end{equation}
commutes up to an isotopy of $\MBG$. In particular, the following square in cohomology commutes
\begin{equation}\label{Psicommsquarecoh}
\begin{tikzpicture}[baseline=(current  bounding  box.east)]
\node (A) at (0,1.2) {$H^*(\tMDolG)$};
\node (B) at (5,1.2) {$H^*(\tMBG)$};
\node (C) at (0,0) {$H^*(\MDolG)$};
\node (D) at (5,0) {$H^*(\MBG)$.};
\draw[<-,thick] (A) -- (C) node [midway,left] {$f_{\mathrm{Dol}}^*$};
\draw[<-,thick] (B) -- (D) node [midway,right] {$f_{\mathrm{B}}^*$};
\draw[<-,thick] (A) -- (B)node [midway,above] {$\widetilde{\Psi}^*$};
\draw[<-,thick] (C) -- (D) node [midway,above] {$\Psi^*$};
\end{tikzpicture}
\end{equation}
The resolutions $f_{\mathrm{Dol}}$ and $f_{\mathrm{B}}$ can be taken functorial with respect to smooth algebraic or analytic morphisms, and symplectic if $G=\Gl_n$ or $\Sl_n$ with $(g,n)=(1,n)$ or $(2,2)$.
\end{thm}
\begin{proof}
Let $f_{\mathrm{Hod}} \colon \mathcal{R}(\bMHodG) \to \bMHodG$ be the functorial resolution of $\bMHodG$, equivalently in the analytic or algebraic category by \cref{lem:analalg}. 
Since $\bar{\lambda}$ is locally analytically trivial by \cref{blambdaloctriv}, $f_{\mathrm{Hod}}$ is a simultaneous resolution of $\bMHodG_{\lambda}$; see for instance \cite[Lemma 4.2]{Graf}. Note also that any
vector field on the smooth locus of $\bMHodG$ can be lifted to a vector field on $\mathcal{R}(\bMHodG)$ by \cite[Corollary 4.7]{GKK10}.

For such a resolution, Proposition 5.2 in \cite{AV21} holds: the family $ \overline{\lambda} \circ f_{\mathrm{Hod}}$
admits a real analytic Ehresmann connection such that the corresponding flow of
diffeomorphisms preserves the exceptional locus of $f_{\mathrm{Hod}}$, and moreover it does so
fibrewise over its image in $\MHodG$. The same proof of \cite[Proposition 5.2]{AV21} shows that we can further suppose that the flow preserves $\partial\MHodG \simeq \partial \MDolG \times \Aff^1$ and its inverse image in $\mathcal{R}(\bMHodG)$. Hence, there exists a resolution of singularities \[\tMHodG \coloneqq f_{\mathrm{Hod}}^{-1}(\MHodG)\] of $\MHodG$ such that the following square commutes 
\begin{center}
\begin{tikzpicture}
\node (A) at (0,1.2) {$\tMDolG \coloneqq \tMHodG_0$};
\node (B) at (5,1.2) {$ \tMHodG_{\epsilon} \eqqcolon \tMDRG$};
\node (C) at (0,0) {$\MDolG \coloneqq \MHodG_0$};
\node (D) at (5,0) {$\MHodG_{\epsilon} \eqqcolon \MDRG$,};
\draw[->,thick] (A) -- (C) node [midway,left] {$f_{\mathrm{Dol}}\coloneqq f_{\mathrm{Hod},0}$};
\draw[->,thick] (B) -- (D) node [midway,right] {$f_{\mathrm{DR}}\eqqcolon f_{\mathrm{Hod},\epsilon}$};
\draw[->,thick] (A) -- (B) node [midway,above]{};
\draw[->,thick] (C) -- (D) node [midway,above]{};
\end{tikzpicture}
\end{center}
where the horizontal arrows are stratified diffeomorphisms, and $\epsilon \neq 0$.

Since the Riemann-Hilbert correspondence is a smooth analytic map, the map $f_{\mathrm{DR}}$ is obtained via base change from the functorial resolution $f_{\mathrm{B}}\colon  \tMBG \coloneqq \mathcal{R}(\MBG)\to \MBG$ by functoriality. Therefore, we obtain the commutative square 
\begin{center}
\begin{tikzpicture}
\node (A) at (0,1.2) {$\tMDolG$};
\node (B) at (5,1.2) {$\tMBG$};
\node (C) at (0,0) {$\MDolG$};
\node (D) at (5,0) {$\MBG$.};
\draw[->,thick] (A) -- (C) node [midway,left] {$f_{\mathrm{Dol}}$};
\draw[->,thick] (B) -- (D) node [midway,right] {$f_{\mathrm{B}}$};
\draw[->,thick] (A) -- (B)node [midway,above] {$\widetilde{\Psi}$};
\draw[->,thick] (C) -- (D) node [midway,above] {$\Psi'$};
\end{tikzpicture}
\end{center}
Since $\Psi'$ and the non-abelian Hodge correspondence $\Psi$ are induced by trivialisation of $\MHodG$, the square \eqref{1square:Psicommsquarecoh} commutes up to a stratified isotopy of $M_{\mathrm{Hod}}(X,G)$. Since stratified isotopy are trivial in cohomology, the square \eqref{Psicommsquarecoh} commutes too.

We now show that the functorial resolutions $f_{\mathrm{Dol}}$ and $f_{\mathrm{B}}$ are symplectic if $G=\Gl_n$ or $\Sl_n$, with $(g,n)=(1,n)$ or $(2,2)$. Indeed, in this case $\MDolG$ and $\MBG$ are normal complex varieties which admit a symplectic resolution obtained by blowing-up the singular locus; see Sections \ref{sec:VarietiesKummertype} and \ref{sec:singularityresolution}, or \cite[Theorem 1.8]{BellamySchedler2019}. Note that the results of \cite{BellamySchedler2019} are stated for $\MBG$, but they extends to $\MDolG$ by the isosingularity principle, see \cite[Theorem 10.6]{Simpson1994} or \cite[\S 2.4 and first paragraph of \S 3.2]{Mauri2021}. Further, analytic neighbourhoods of the singularities of these varieties are $\QQ$-factorial. Indeed, the singularities of $\MDolG$ and $\MBG$ are either quotient singularities or the nilpotent cone in $\mathfrak{sp}(4)$, which is a cone over a projective variety with quotient singularities and Picard number one; see the last paragraph of the proof of \cite[Theorem 1.3]{BellamySchedler2019} and references therein, and \cite[Lemma 1.3]{MRS18} or \cite[\S 3.4]{Mauri2021}. By \cite[Proposition 5.15]{KollarMori1998} and \cite[Proposition 7.4]{K13} these singularities are analytically $\QQ$-factorial. Hence, the last statement of \cref{LiftNAHC} follows from \cref{lem:symplfunct}.
\end{proof}

 \begin{rmk}
 In this paper, functorial resolutions are used only for the following purposes: to lift vector fields and group actions to resolutions, and for the compatibility with respect to the Riemann--Hilbert correspondence; see proof of \cref{LiftNAHC} and Section \ref{sec: P=W for SL implies P=W for GL}. If $G=\Gl_n$ or $\Sl_n$ with $(g,n)=(1,n)$ or $(2,2)$, the symplectic resolutions of $\MDolG$ and $\MBG$ are indeed functorial by \cref{lem:symplfunct} but these properties can be shown more directly. The resolutions are obtained by blowing up the singular locus, which is invariant with respect to any group action on the varieties and preserved by the Riemann--Hilbert correspondence. Further, the liftability of vector fields follows easily for instance from \cite[Lemma 5.3]{AV21}. 
 \end{rmk}

\section{Moduli spaces for $\Gl_n$ vs $ \Sl_n$}\label{sec:modulispacesforGlSl}
Let $\Gamma \coloneqq \Pic^0(X)[n] \simeq (\ZZ/n\ZZ)^{2g}$ be the group of $n$-torsion line bundles on the Riemann surface $X$ of genus $g$ and canonical line bundle $K_X$. We review the relation between the moduli spaces $\MDolG$ and $\MBG$ for $G=\Gl_n$ and $\Sl_n$; see also \cite{Hitchin1987, Simpson1994, Simpson1994I}.

Recall that $\MDolGl$ parametrises semistable Higgs bundles $(E, \phi)$, where $E$ is a vector bundle on $X$ of rank $n$ and degree $0$, and $\phi \in \Hom(E, E \otimes K_X)$. 

The fibre of the isotrivial morphism
\begin{align}
    \alb\colon  \MDolGl & \to \Pic^0(X) \times H^0(X, K_X) \label{eq:alb}\\
    (E, \phi) & \mapsto (\det E, \Tr \phi) \nonumber
\end{align}
is isomorphic to $\MDolSl$. In particular, the monodromy of $\alb$ is the group $\Gamma$. Indeed, the \'{e}tale cover 
\begin{align}
    \MDolSl \times \Pic^0(X) \times H^0(X, K_X) & \to \MDolGl\label{eq:etalecoverMDolSLGL}\\ 
    ((E, \phi), L, s) & \mapsto (E \otimes L, \phi + \frac{s}{n}\mathrm{id}_{E}) \nonumber
\end{align}
has Galois group $\Gamma$, which acts on the domain diagonally by tensorisation
\begin{align*}
    \Gamma \times \MDolSl \times \Pic^0(X) \times H^0(K) & \to \MDolSl \times \Pic^0(X) \times H^0(K)\\
    (L_{\gamma}, (E, \phi), L, s) & \mapsto (L_{\gamma}, (E \otimes L_{\gamma}, \phi), L \otimes L^{-1}_{\gamma}, s)
\end{align*}

Therefore, when we take cohomology, we obtain
\begin{align}
H^*(\MDolGl)& \simeq H^*(\MDolSl \times \Pic^0(X) \times H^0(X, K_X))^{\Gamma} \nonumber \\
& \simeq H^*(\MDolSl)^{\Gamma} \otimes H^*( \Pic^0(X)), \label{eq:cohomMGLDol}
\end{align}
where the former equality follows from an observation of Grothendieck in \cite{Grothendieck_1957}, and the latter from the fact that $\Gamma$ acts trivially on  $H^*(\Pic^0(X))$, since it is a restriction to a subgroup of the action of the connected group $\Pic^0(X)$.

The \textbf{Hitchin map}
\[\chi(X,\Gl_n)\colon  \MDolGl \to \bigoplus^{n}_{i=1} H^0(X, K_X^{\otimes i})\]
is a projective fibration sending $(E, \phi)$ to the characteristic polynomial of $\phi$. It is Lagrangian with respect to $\omega$, i.e.\ the holomorphic symplectic form of the canonical hyper-K\"{a}hler metric on the smooth locus of $\MDolGl$; see \cite[\S 6]{Hitchin1987}. The map $\chi(X,\Gl_n)$ restricts on $\MDolSl$ to
\[\chi(X,\Sl_n)\colon  \MDolSl \to \bigoplus^{n}_{i=2} H^0(X, K_X^{\otimes i}).\]
The map $\chi(X,\Sl_n)$ is $\Gamma$-equivariant, covering the trivial $\Gamma$-action of the codomain. In particular, there exists a commutative diagram
\begin{equation}\label{diagramHitchinmaps}
\begin{tikzpicture}[baseline=(current  bounding  box.east)]
\node (A) at (0,1.2) {$\MDolSl \times \Pic^0(X) \times H^0(X, K_X)$};
\node (B) at (6,1.2) {$\MDolGl$};
\node (C) at (0,0) {$\bigoplus^{n}_{i=2} H^0(X, K_X^{\otimes i}) \times H^0(X, K_X)$};
\node (D) at (6,0) {$\bigoplus^{n}_{i=1} H^0(X, K_X^{\otimes i})$};
\draw[->,thick] (A) -- (C) node [midway,right] {$(\chi(X,\Sl_n), S_{\Pic^0(X)}, \mathrm{id}_{H^0(X, K_X)})$};
\draw[->,thick] (B) -- (D) node [midway,right] {$\chi(X,\Gl_n)$};
\draw[->,thick] (A) -- (B)node [midway,above] {};
\draw[->,thick] (C) -- (D) node [midway,below] {$=$};
\end{tikzpicture}
\end{equation}
with $S_{\Pic^0(X)}\colon  \Pic^0(X) \to \mathrm{pt}$.

Via the non-abelian Hodge correspondence $\Psi$, the action of $\Gamma$ on $\MDolSl$ corresponds to the algebraic action of the characters $\Hom(\pi_1(C), \ZZ/n\ZZ)$ which acts on $\MBSl$ by multiplication (changing the signs of the matrices $A_j$, $B_j$'s as in \eqref{eq:charactervar}). 

The multiplication map $\Sl_n \times \Gm \to \Gl_n$ induces the \'{e}tale cover 
$\MBSl \times (\CC^{*})^{2g} \to \MDolGl$ with Galois group $\Gamma$. Therefore, the analogue of (\ref{eq:cohomMGLDol}) holds
\begin{align}
    H^*(\MBGl)& \simeq H^*(\MBSl \times (\CC^{*})^{2g})^{\Gamma} \nonumber \\
& \simeq H^*(\MDolSl)^{\Gamma} \otimes H^*( (\CC^{*})^{2g}). \label{eq:cohomMGLB}
\end{align}

\subsection{P=W for $\Sl_n$ implies P=W for $\Gl_n$}\label{sec: P=W for SL implies P=W for GL}
In this section we show that the P=W conjectures for $\Sl_n$ imply the corresponding statements for $\Gl_n$. In the twisted case, this is proved in \cite[\S 2.4]{deCataldoHauselMigliorini2012}; see also \cite[\S 1]{deCataldoMaulikShen2020}.  In view of \cref{P=WSLimpliesGL}, starting from  \cref{sec:P=W1}, we will focus our attention on the $\Sl_n$ case exclusively. 

Fix $\Gamma$-equivariant resolution of singularities 
\begin{align*}
    f_{\mathrm{Dol}}(X,\Sl_n) & \colon  \tMDolSl \to \MDolSl\\ 
    f_{\mathrm{B}}(X,\Sl_n)& \colon  \tMBSl \to \MBSl,
\end{align*}
which satisfies \cref{LiftNAHC}. Note that the functorial resolutions in the proof of \cref{LiftNAHC} are actually $(\Gamma \times \mathbb{G}_m)$-equivariant; see \cite[Proposition 3.9.1]{Kollar2007}.
By the isotriviality of 
\begin{align*}
    \alb_{\mathrm{Hod}}\colon  \MHodGl & \to \MHodGm\\
    (E, \nabla_{\lambda})& \mapsto (\det E, \Tr \, \nabla_{\lambda})
\end{align*}
(which extends the morphism $\alb$ defined in (\ref{eq:alb})), the resolutions $f_{\mathrm{Dol}}(X, \Sl_n)$ and $f_{\mathrm{B}}( X, \Sl_n)$ extends to resolutions
\begin{align*}
    f_{\mathrm{Dol}}(X, \Gl_n) & \colon  \tMDolGl \to \MDolGl\\ 
    f_{\mathrm{B}}(X, \Gl_n)& \colon  \tMBGl \to \MBGl
\end{align*}
such that the square 
\begin{equation}\label{diagrametalecover}
\begin{tikzpicture}
\node (A) at (0,1.2) {$\tMDolSl \times T^*\Pic^0(X)$};
\node (B) at (7,1.2) {$\tMBSl \times (\CC^*)^{2g}$};
\node (C) at (0,0) {$\tMDolGl$};
\node (D) at (7,0) {$\tMBGl$};
\draw[->,thick] (A) -- (C) node [midway,left] {$/ \Gamma$};
\draw[->,thick] (B) -- (D) node [midway,right] {$/ \Gamma$};
\draw[->,thick] (A) -- (B)node [midway,above] {$\widetilde{\Psi}(X, \Sl_n) \times \widetilde{\Psi}(X, \Gm)$};
\draw[->,thick] (C) -- (D) node [midway,above] {$\widetilde{\Psi}(X, \Gl_n)$};
\end{tikzpicture}
\end{equation}
and the diagrams in \cref{LiftNAHC} commute. 
\begin{thm}\label{P=WSLimpliesGL}
In the notation above, if the $P=W$ conjecture for the resolution $f_{\mathrm{Dol}}(X, \Sl_n)$ holds, then it holds for $f_{\mathrm{Dol}}(X, \Gl_n)$.
\end{thm}

\begin{proof}
Cohomologically the Hitchin fibration 
\[\chi(X,\Gl_n) \circ f_{\mathrm{Dol}}(X, \Gl_n)\colon  \tMDolGl \to \bigoplus_{i=1}^nH^0(X,K_X^{\otimes i})\]
behaves like the product of the fibration
$\chi(X,\Sl_n) \circ f_{\mathrm{Dol}}(X, \Sl_n)$ and $S_{\Pic^0(X)}\colon$ $\Pic^0(X) \to \mathrm{pt}$, by lifting (\ref{diagramHitchinmaps}) to the resolution. Hence, the perverse filtration associated to
$\chi(X, \Gl_n) \circ f_{\mathrm{Dol}}(X, \Gl_n)$ is the convolution of the $\Gamma$-invariant part of the perverse filtrations associated to $\chi(X, \Sl_n) \circ f_{\mathrm{Dol}}(X, \Sl_n)$ and $S_{\Pic^0(X)}$ (the latter being trivial); compare with \cite[\S 2.4]{deCataldoHauselMigliorini2012}. In symbols, we write 
\begin{equation}\label{eq:PMGL}
  P_kH^d( \tMDolGl) \simeq \bigoplus_{j\geq 0} P_{k-j}H^{d-j}(\tMDolSl)^{\Gamma}\otimes H^j(\Pic^0(C)).   
\end{equation} 

By the $\Gamma$-equivariance of $f_{B}(X, \Sl_n)$, the map $\MBSl \times (\CC^{*})^{2g} \to \MDolGl$ lifts to the resolutions, and so there exists an isomorphism of mixed Hodge structures
\begin{align}
    H^*(\tMBGl) \simeq H^*(\tMBSl)^{\Gamma} \otimes H^*( (\CC^{*})^{2g}), \label{eq:cohomtMGLB}
\end{align}
as in (\ref{eq:cohomMGLB}). Explicitly, we write 
\begin{equation}\label{eq:WMGL}
W_{k}H^d(\tMBGl) \simeq \bigoplus_{j\geq 0}  W_{k-2j}H^{d-j}(\tMBSl)^{\Gamma}\otimes H^j((\CC^*)^{2g}),
\end{equation}
since $H^j((\CC^*)^{2g})$ has weight $2j$. 

Assume now that 
\[P_{k}H^*(\tMDolSl) = \widetilde{\Psi}(X, \Sl_n)^* W_{2k}H^*(\tMBSl).\]
Then by the commutativity of (\ref{diagrametalecover}), together with (\ref{eq:PMGL}) and (\ref{eq:WMGL}), we conclude that
\[P_{k}H^*(\tMDolGl) = \widetilde{\Psi}(X, \Gl_n)^* W_{2k}H^*(\tMBGl).\]
\end{proof}

\begin{rmk}
With obvious change, the analogues of \cref{P=WSLimpliesGL} for the PI=WI and P=W conjectures hold.
\end{rmk}

\begin{rmk}\label{rmk:PGLn}
Since $\MDolPGl$ is the quotient of $\MDolSl$ by the $\Gamma$-action, the PI=WI conjecture for $\MDolSl$ (or $\MDolGl$) implies the PI=WI conjecture for $\MDolPGl$.
\end{rmk}



\section{P=W conjectures for genus 1}\label{sec:P=W1}
Let $A$ be a compact Riemann surface of genus 1.
The construction of the moduli spaces $\MDolSlone$ and $\MBSlone$ agrees formally with that of a generalised Kummer variety in \cite[\S 7]{Beauville84}. It is possible to make this analogy more precise by showing that $\MDolSlone$ and $\MBSlone$ are specializations of generalized Kummer varieties; see Example \ref{genusonekummer} and also \cite[\S 5.3]{MMS}.

Following \cite{GottscheSoergel93}, we describe a stratification of these Kummer-like varieties in Section \ref{sec:VarietiesKummertype}, from which we deduce the P=W conjecture in genus 1 (\cref{thm:PI=IWgenusone}).

\subsection{Kummer-like varieties}\label{sec:VarietiesKummertype}
Let $X$ be a complex algebraic group of dimension 2. We denote by $X^{(n)}$ and $X^{[n]}$ the n-fold symmetric product of $X$ and the Hilbert scheme of n-points on $X$; see \cite[\S 6]{Beauville84} for an overview of their construction. Recall that the Hilbert-Chow morphism $f\colon  X^{[n]} \to X^{(n)}$ is a desingularization of $X^{(n)}$.

Consider  the addition map $a_n\colon  X^{(n)} \to X$, given by $a_n(x_1, \ldots, x_n)= \sum_{i=1}^{n}x_i$. For any $g \in \ZZ_{\geq 0}$, denote by $X(g)$ the set of $g$-torsion points in $X$. Let $P(n)$ be the set of partitions of $n$. We write $\alpha \in P(n)$ as $n=\alpha_1 \cdot 1 + \ldots + \alpha_l \cdot l$, and put $|\alpha|= \sum \alpha_i$ and $g(\alpha)\coloneqq \gcd\{\nu | \alpha_{\nu}\neq 0 \}$. 

Following \cite{GottscheSoergel93}, we describe a stratification of the fibre of $a_n$.

\begin{itemize}
\item The variety $K^{[n]}$ is the fibre $f^{-1} \circ a^{-1}_n(0)$ of the composition
$
X^{[n]} \xrightarrow{f} X^{(n)} \xrightarrow{a_n} X.
$
When necessary, we emphasize the dependence on $X$ by writing $K^{[n]}(X)$.
\item The fibre $K^{(n)} \coloneqq a_n^{-1}(0)$ can be described as the set of maps from $X$ to $\ZZ_{\geq 0}$ of total sum $n$
\[
K^{(n)} = \big\{ h \in \mathrm{Hom}_{\underline{\mathrm{Sets}}}(X, \ZZ_{\geq 0}) \big| \, \sum_{x \in X} h(x)=n \big\}.
\]
We say that $K^{(n)}$ is Kummer-like.
\item There exists a stratification $K^{(n)} = \bigsqcup_{\alpha \in P(n)} K^{(n)}_{\alpha}$ with 
$
K^{(n)}_{\alpha} = \big\{ h \in K^{(n)} \big| \, \# h^{-1}(x)=  \alpha_{\nu} \quad \forall \nu \big\}.
$
\item The normalization of the closure of the stratum $K^{(n)}_{\alpha}$ in $K^{(n)}$, denoted $K^{(\alpha)}$, is the disjoint union
\[
K^{(\alpha)}= \bigsqcup_{y \in X(g(\alpha))} K^{(\alpha)}_y,
\]
where 
\[
K^{(\alpha)}_y = \big\{ h =(h_1, \ldots, h_l) \in K^{(\alpha)} \big|\, \sum_{\nu, x} \frac{\nu}{g(\alpha)}h_{\nu}(x) \cdot x=y\big\}.
\]
\item Let $\tau_{z}\colon  X \to X$ be the translation by $z \in X$. The finite map $q^{(\alpha)}_y\colon  X \times K^{(\alpha)}_y \to X^{(\alpha)}$, given by $q^{(\alpha)}_y(z, h_1,\ldots, h_l)=(h_1 \circ \tau_z, \ldots, h_l \circ \tau_z)$, induces the isomorphism of mixed Hodge structures $H^*(X \times K^{(\alpha)}_y) \simeq H^*(X^{(\alpha)})$; see \cite[p.243]{GottscheSoergel93}.
\end{itemize}
All these facts implies the following theorem due to G\"{o}ttsche and Soergel, that we enunciate without proof. 
\begin{thm}\label{thm:weightgenusone} \emph{\cite[Theorem 7]{GottscheSoergel93}}
Denote by $f_0$ the birational map $f_0 \coloneqq (\mathrm{id}_X, f|_{K^{[n]}})\colon $ $ X \times K^{[n]}\to X \times K^{(n)}$. Let $\kappa^{(\alpha)}_y\colon  K^{(\alpha)}_y \to K^{(n)}$ be the composition $K^{(\alpha)}_y \hookrightarrow K^{(\alpha)} \to \overline{K^{(n)}_{\alpha}}\hookrightarrow K^{(n)}$. 

Then there exists a distinguished splitting isomorphism 
\begin{equation}\label{eq:splittingKummerinderivedcategory}
(f_0)_* (\underline{\QQ}_{X \times K^{(n)}}[n]) \simeq \bigoplus_{\alpha \in P(n)} \bigoplus_{y \in X(g(\alpha))} (\mathrm{id}_X \times \kappa^{(\alpha)}_y)_* (\underline{\QQ}_{X \times K^{(\alpha)}_y}[|\alpha|]).
\end{equation}
The splitting induces a canonical isomorphism of mixed Hodge structures (recall that a
Tate twist $(-k)$ increases the weights by $2k$):
\begin{equation}\label{eq:splittingKummerincohomology}
H^{d+2n}(X \times K^{[n]})(n) \simeq \bigoplus_{\alpha \in P(n)} \bigoplus_{y \in X(g(\alpha))} H^{d + 2|\alpha|}(X^{(\alpha)})(|\alpha|).
\end{equation}
\end{thm}
A morphism $\chi\colon  X \to \CC$ yields the commutative diagram
\begin{center}
\begin{tikzpicture}
\node (A) at (0,1.2) {$X \times K^{(n)}$};
\node (B) at (4,1.2) {$X \times K^{(\alpha)}_y$};
\node (C) at (8,1.2) {$X^{(\alpha)}$};
\node (F) at (0,0) {$\CC \times \CC^{(n-1)}$};
\node (G) at (4,0) {$\CC \times \CC^{(|\alpha|-1)}$};
\node (H) at (8,0) {$\CC^{(\alpha)}.$};
\draw[<-,thick] (A)--(B) node [midway,above] {$\mathrm{id}_X \times \kappa^{(\alpha)}_y$};
\draw[->,thick] (B)--(C) node [midway,above] {$q^{(\alpha)}_y$};
\draw[<-,thick] (F)--(G) node [midway,above] {$\mathrm{id}_{\CC} \times \kappa^{(\alpha)}_{\chi(y)}$};
\draw[->,thick] (G)--(H) node [midway,above] {$q^{(\alpha)}_{\chi(y)}$};
\draw[->,thick] (A)--(F) node [midway,left] {$\chi_0$}; 
\draw[->,thick] (B)--(G) node [midway,right] {$\chi^{(\alpha)}_y$};
\draw[->,thick] (C) -- (H) node [midway,right] {$\chi^{(\alpha)}$};
\end{tikzpicture}
\end{center}
The perverse filtration associated with $\chi_0 \coloneqq \mathrm{id}_X \times \chi^{(n)}|_{K^{(n)}}$ can be written in terms of the perverse filtration associated with $\chi^{(\alpha)}$.
\begin{thm}\label{thm:perverseingenusone}
The perverse filtration associated with $\chi_0$ can be expressed as 
\[P_k H^{d+2n}(X \times K^{[n]})(n) \simeq \bigoplus_{\alpha \in P(n)} \bigoplus_{y \in X(g(\alpha))} P_k H^{d + 2|\alpha|}(X^{(\alpha)})(|\alpha|).\]
where $P_k H^{*}(X^{(\alpha)})$ is the perverse filtration associated with $\chi^{(\alpha)}$.
\end{thm}
\begin{proof} By \cref{thm:weightgenusone} and the $t$-exactness of finite morphisms, we obtain
\begin{align*}
{^{\mathfrak{p}}}\mathcal{H}^{k}((\chi_0 \circ f_0)_*& (\underline{\QQ}_{X \times  K^{(n)}}[n])) \simeq \\
& \bigoplus_{\alpha \in P(n)} \bigoplus_{y \in X(g(\alpha))} (\mathrm{id}_X \times \kappa^{(\alpha)}_y)_* {^{\mathfrak{p}}}\mathcal{H}^{k} (\chi^{(\alpha)}_{y,*}\underline{\QQ}_{X \times K^{(\alpha)}_y}[|\alpha|])
\end{align*}
\[
q^{(\alpha)}_{\chi(y), *}{^{\mathfrak{p}}}\mathcal{H}^{k}(\chi^{(\alpha)}_{y,*}\underline{\QQ}_{X \times K^{(\alpha)}_y}) 
= {^{\mathfrak{p}}}\mathcal{H}^{k}(\chi^{(\alpha)}_{ *} \, q^{(\alpha)}_{y, *}\underline{\QQ}_{X \times K^{(\alpha)}_y}) \supset {^{\mathfrak{p}}}\mathcal{H}^{k}(\chi^{(\alpha)}_{ *} \underline{\QQ}_{X^{(\alpha)}}).
\]
This means that the isomorphism (\ref{eq:splittingKummerincohomology}) is filtered strict with respect to the perverse filtration associated with $\chi_0$ and $\chi^{(\alpha)}$.
\end{proof}
\subsection{The proof of the conjecture} 
\begin{thm}\label{thm:PI=IWgenusone}
The PI=WI conjectures for $\MDolSlone$ and the P=W conjectures for its symplectic resolutions hold.
\end{thm}
\begin{proof}
The moduli space $\MDolSlone$ parametrises semistable Higgs bundles on the elliptic curve $A$, and it is isomorphic to $K^{(n)}(A \times \CC)$; see for instance \cite[Theorem 4.27.(v), which actually holds for any $n$, not only for $n \geq 4$]{FGPN14} or \cite{Groechenig14}. The character variety $\MBSlone$ instead is isomorphic to $K^{(n)}(\CC^* \times \CC^*)$ (cf \cite[Proof of Theorem 5.3.2]{MMS}), and in suitable coordinates the non-abelian Hodge correspondence is induced by the symmetric product of the exponential map
\begin{align*}
A \times \CC & \to \CC^* \times \CC^*\\
(\theta_1, \theta_2, r_1, r_2) & \mapsto (\exp(-2r_1+ i \theta_1), \exp(2r_2+ i \theta_2));
\end{align*}
see \cite[Example after Proposition 1.5]{Simpson92}.

By \cref{thm:weightgenusone} and \ref{thm:perverseingenusone}, the P=W conjecture for the symplectic resolution $K^{[n]}(A \times \CC)$ is equivalent to
\begin{equation}\label{eq:P=WX^alpha}
P_k H^*((A \times \CC)^{(\alpha)}) = W_{2k} H^*((\CC^* \times \CC^*)^{(\alpha)})
\end{equation}
for any partition $\alpha \in P(n)$. The identity (\ref{eq:P=WX^alpha}) have already been proved in \cite[Lemma 3.1.1 and 3.2.2]{deCataldoHauselMigliorini2013}.
\end{proof}

\begin{rmk}\label{PI=WIisP=Wgenusone}
Since $\MDolSlone$ has at worst quotient singularities, the P=W conjecture for $\MDolSlone$ is equivalent to the PI=WI conjecture for $\MDolSlone$.
\end{rmk}
\section{The moduli space of Higgs bundles $M$ and its alterations}\label{sec:preliminary}
Here and in the following $C$ is a compact Riemann surface of genus $2$. We denote by $\iota\colon  C \to C$ the hyperelliptic involution, and by $K_C$ the canonical bundle of $C$. 

For the sake of notational simplicity, we denote 
\begin{itemize}
    \item the Dolbeault moduli space ${M}_{\mathrm{Dol}}(C, \Sl_2)$ simply by $M$;
    \item the desingularization $\widetilde{M}_{\mathrm{Dol}}(C, \Sl_2)$ in \cref{prop:fibresf} by $\widetilde{M}$;
    \item the character variety ${M}_{\mathrm{B}}(C, \Sl_2)$ by $M_B$;
    \item the resolution $f_{\mathrm{Dol}}(C, \Sl_2)$ by $f: \widetilde{M} \to M$;
    \item the Hitchin map $\chi(C, \Sl_2)$ by $\chi\colon  M \to H^0(C, K^{\otimes 2})$. 
\end{itemize}

\subsection{Symplectic resolution of $M$}
\subsubsection{Singularities of $M$ and its resolution}\label{sec:singularityresolution}
We briefly recall the description of the singular locus of $M$ and the construction of the resolution. A key aspect is the local isomorphism between the singularities of $M$ and those of the celebrated O'Grady six dimensional example of irreducible holomorphic symplectic variety. We refer to \cite{Felisetti2018} for more details.
Via the non-abelian Hodge correspondence, we obtain an analogous description of the singularities of $M_B$.

There exists a Whitney stratification of $M$
\begin{equation}\label{WhitneyM}
 \Omega = \Sing(\Sigma) \subset\Sigma = \Sing(M) \subset M   
\end{equation}
where 
\begin{align*}
 \Sigma& \simeq \left\lbrace(E,\phi)\simeq (L,\varphi)\oplus (L^{-1},-\varphi) \text{ with }L\in \Pic^0(C),\text{ and }\varphi\in H^{0}(C, K_C)\right \rbrace;\\
 \Omega& \simeq \left\lbrace(E,\phi)\simeq (L,0)\oplus (L,0) \text{ with }L\in \Pic^0(C)\text{ s.t. }L^2\simeq\mathcal{O}_C \right \rbrace.
\end{align*}
 Note that $\Sigma$ is isomorphic to the quotient of $\Pic^0(C)\times H^0(C, K_C)$ by the involution $(L,\varphi)\mapsto (L^{-1},-\varphi)$, hence it has dimension 4. The locus $\Omega$ instead is the branch locus of the quotient map $\Pic^0(C)\times H^0(C, K_C) \to \Sigma$, and consists of 16 points $\Omega_j$, with $j=1, \ldots, 16$. 

A transverse slice to $\Sigma$ at a point in $\Sigma \setminus \Omega$ has a quotient surface singularity of type $A_1$. An analytic neighbourhood of a point of $\Omega$ is more complicated, and it was described in detail in \cite{LehnSorger2006}. The singularities are symplectic, and a symplectic resolution can be constructed simply by blowing-up $M$ along $\Sigma$.

\begin{prop}
\emph{\cite[Proposition 4.2]{Felisetti2018}}\label{prop:fibresf}
Let $f\colon \tm \rightarrow M$ be the blow-up of $M$ along $\Sigma$. Then $f$ is a symplectic resolution, and we have that:
\begin{itemize}
    \item $f$ is an isomorphism over $M\setminus \Sigma$;
    \item $f^{-1}(p)\simeq \mathbb{P}^1$ for all $p\in \Sigma\setminus \Omega$;
    \item $f^{-1}(\Omega_j)\simeq \tomega_j$, where $\tomega_j$ is the Grassmanian of Lagrangian planes in a symplectic 4-dimensional vector space, which is isomorphic to a smooth quadric in $\PP^4$.
\end{itemize}
\end{prop}
Via the non-abelian Hodge correspondence $\Psi$, the stratification of $M$ in (\ref{WhitneyM}) induces the stratification of $M_B$ given by
\begin{equation}\label{WhitneyMB}
 \Omega_B = \Sing(\Sigma_B)=\Psi(\Omega) \subset\Sigma_B = \Sing(M_B) = \Psi(\Sigma) \subset M_B,
\end{equation}
where 
\begin{align}
 \Sigma_B &\coloneqq\left\lbrace (A_1, A_2, B_1, B_2) \in (\CC^*)^{4} \subset \Sl^{4}_2\right \rbrace\sslash \Sl_2 \simeq (\CC^*)^{ 4}/(\ZZ/2\ZZ); \label{descriptionSigmaB}\\
 \Omega_B&\coloneqq\left\lbrace (A_1, A_2, B_1, B_2) \in (\pm \id)^{4} \subset \Sl^{ 4}_2\right \rbrace = \bigcup^{16}_{j=1}\Omega_{B, j}. \nonumber
\end{align}
By \cref{thmintro:liftnonabelian}, $M_B$ admits a symplectic resolution, and its fibres can be described as in \cref{prop:fibresf}. 

\subsubsection{Attracting and repelling sets}\label{sec:attrrepell}
\begin{defn}
Let $X$ be a complex variety with a $\Gm$-action, and $F$ be a subset of its fixed locus. We denote by
\[
\Attr(F)= \{ x \in X | \lim_{\lambda \to 0} \lambda \cdot x \in F \}
\] 
the \textbf{attracting set} of $F$, and by 
\[
\Repell(F)= \{ x \in X | \lim_{\lambda \to \infty} \lambda \cdot x \in F \}
\] 
the \textbf{repelling set} of $F$.
\end{defn}

The tangent space of any fixed point $p \in \Fix(X)$ decomposes into the direct sum of weights spaces
\[
T_p X = \bigoplus_{m \in \mathbb{Z}} T_p X_{m}
\]
where $T_p X_m = \{ v \in T_p X | \, \lambda \cdot v = \lambda^m v \text{ for all }\lambda \in \Gm \}$. 
 
\begin{defn}
The sequences of integers $m_1, m_2, \ldots$ such that $\lambda^{m_1}, \lambda^{m_2}, \ldots$ are eigenvalues of the linear operator induced by the $\Gm$-action on $T_p X$ 
are called \textbf{weights} of the $\Gm$-action at the fixed point $p$.  
\end{defn}
Let $X^{\mathrm{sm}}$ be the smooth locus of $X$, and denote a connected component of the fixed locus $\Fix(X^{\mathrm{sm}})$ simply by $F$. Note that the function of weights 
\begin{align*}
\Fix(X^{\mathrm{sm}}) & \to \mathbb{Z}^{(\dim X)}\\
p & \mapsto (m_1(p), m_2(p), \ldots )
\end{align*} is locally constant. 

In particular, the following identities hold:
\begin{equation}\label{eq:dimTangattractor}
T_p \Attr(p)= \bigoplus_{m>0} T_p X_{m} \qquad T_p \Repell(p)= \bigoplus_{m<0} T_p X_{m}
\end{equation} 
\begin{equation}\label{eq:dimensiontangentspace}
T_p \Attr(F)= \bigoplus_{m\geq 0} T_p X_{m} \qquad T_p \Repell(F)= \bigoplus_{m\leq 0} T_p X_{m}.
\end{equation} 

\subsubsection{Białynicki--Birula decomposition}
We briefly recall the celebrated Białynicki--Birula decomposition.
\begin{defn} \cite[Definition 1.1.1]{HauselVillegas15}
A \textbf{semiprojective} variety is a complex quasi-projective algebraic variety $X$ with a $\Gm$-action such that:
\begin{itemize}
\item the fixed point set $\Fix(X)$ is proper;
\item for every $x \in X$ the limit $\lim_{\lambda \to 0} \lambda \cdot x$ exists.
\end{itemize}
\end{defn}

\begin{thm}[Białynicki--Birula decomposition]\label{BBdecomposition}
Let $X$ be a normal semiprojective variety. Then the following facts hold:
\begin{enumerate}
\item $X$ admits a decomposition into $\Gm$-invariant locally closed subsets
\[
X = \bigsqcup_{F \in \pi_0(\Fix(X))} \Attr(F);
\]
\item \label{limitmap} the limit map 
\[ \Attr(F) \to F \:\colon\: x \mapsto \lim_{x \to 0}\lambda \cdot x \]
is an algebraic map, and it is an affine bundle if $F \subset X^{\mathrm{sm}}$;
\item the connected components of the fixed locus $\Fix(X^{\mathrm{sm}})$ are smooth.
\end{enumerate}
\end{thm}
\begin{proof}
See \cite[Theorem 4.3]{BB73} in the smooth projective case; \cite[\S 1.2]{HauselVillegas15} and \cite[Lemma 3.2.4]{MaulikOkunkov19} in the smooth semiprojective case; \cite[Corollary 4]{Weber17} in the normal complete case. 
\end{proof}

The cohomology of a semiprojective variety can be expressed in terms of the cohomology of the components of the fixed locus.
\begin{thm}[Local-to-global spectral sequence]\label{local-to-globalspseq} \emph{\cite[\S 4.4]{Williamson20}}
Let $X$ be a normal semiprojective variety. Fix an ordering $F_0, F_1, \ldots$ of the connected components of $\Fix(X)$ such that if $F_i < F_j$ then $\dim \Attr F_i \geq \dim \Attr F_j$. Then the following facts hold.
\begin{itemize}
\item The Białynicki--Birula decomposition yields the spectral sequence
\begin{equation}\label{eq:spseq1}
E^{i,j}_1= H^{i+j}(\Attr(F_i), u^{!}_i \underline{\QQ}_X) \Rightarrow H^{i+j}(X, \QQ),
\end{equation}
where $u_i \colon  \Attr(F_i) \hookrightarrow X$ is the inclusion.
\item If $X$ is smooth and $\Attr(F_i)$ are smooth subvarieties of codimension $c_j$, 
then we can rewrite the spectral sequence (\ref{eq:spseq1}) as
\begin{equation}\label{eq:spseq2}
E^{i,j}_1= H^{i+j-2c_j}(F_i,  \QQ) \Rightarrow H^{i+j}(X, \QQ)
\end{equation}
\item The spectral sequence (\ref{eq:spseq2}) degenerates at the first page, and the Poincar\'{e} polynomial $P_t(X)\coloneqq \sum^{2\dim X}_{k=0} (-1)^n \dim H^k(X, \QQ)$ can be written
\[P_t(X) = \sum P_t(F_i)t^{2c_i}.\]
\end{itemize}
\end{thm}
\subsubsection{Torus action on $M$ and $\widetilde{M}$}\label{sec:torusaction}
The multiplicative group $\Gm$ acts on $M$ by rescaling the Higgs field
\[
\lambda \cdot (E, \phi)=(E, \lambda \phi).
\]
The Hitchin map $\chi\colon  M \to H^0(C, K_C^{\otimes 2})$ 
is $\Gm$-equivariant, where $\Gm$ acts linearly on $H^0(C, K_C^{\otimes 2})$ with weight $(2,2,2)$. In particular, the fixed locus of $M$ is contained in the nilpotent cone $\chi^{-1}(0)$. Therefore, $M$ is semiprojective. Since the singular locus $\Sigma$ of $M$ is $\Gm$-invariant, the action lifts to $\widetilde{M}$, and $\widetilde{M}$ is semiprojective as well. 

The goal of this section is to describe the fixed locus of the $\Gm$-action on $M$, $\Omega_j$ and $\widetilde{M}$, and to compute the weights of the action.

\begin{prop}\label{Hitchin:classificationvb}
\emph{\cite[Example 3.13]{Hitchin1987}} A vector bundle $E$ underlying a semistable Higgs bundle $(E, \phi) \in M$ satisfies one of the following property:
\begin{enumerate}
    \item $E$ is a stable vector bundle;
    \item $E \simeq L \oplus L^{-1}$ with $L\in \Pic^0(C)$ and $L^2 \not\simeq \mathcal{O}_C$, i.e.\ $(E, \phi) \in \Sigma \setminus \Omega$;
    \item $E \simeq L \oplus L^{-1}$ with $L^2 \simeq \mathcal{O}_C$, i.e.\ $(E, \phi) \in \Omega$;
    \item\label{item:nontrivialextension} $E$ is a non-trivial extension of $L$ by $L^{-1}$ with $L^2 \simeq \mathcal{O}_C$;
    \item $E$ is an unstable vector bundle isomorphic to $\theta^{-1}_j \oplus \theta_j$, where $\theta_j$ is a theta-characteristic, i.e.\ a line bundle such that $\theta_j^2=K_C$.
\end{enumerate}
\end{prop}

\begin{prop}[Fixed locus of $M$]\label{prop:fixedlocusM}
The fixed locus of the $\Gm$-action on $M$ is
\[
\Fix(M)= \NBun \sqcup \Theta = \NBun \sqcup \bigsqcup^{16}_{j \in 1} \Theta_j, 
\]
where 
\begin{enumerate}
\item $\NBun$ is the moduli space of semistable Higgs bundles $(E, \phi)$ with $\phi=0$, equivalently the moduli space of semistable vector bundles of rank 2 and
degree 0, which is isomorphic to $\PP^3$;
\item $\Theta$ is the set of 16 points in $M$ corresponding to the Higgs bundles
\[
\Theta_j \coloneqq \bigg(\theta^{-1}_j \oplus \theta_j, \begin{pmatrix}
0 & 1\\
0 & 0\\
\end{pmatrix}\bigg).
\]
\end{enumerate}
\end{prop}
\begin{proof}

It is clear that $\NBun$ and $\Theta$ are fixed by the $\Gm$-action. Hence, we just need to show that they are the only components of $\Fix(M)$.

To this end, recall that by \cref{Hitchin:classificationvb} the vector bundle $E$ underlying a semistable Higgs bundle $(E, \phi) \in M$ is:
\begin{enumerate}
    \item either a semistable vector bundle,
    \item or an unstable vector bundle, isomorphic to $\theta^{-1}_j \oplus \theta_j$ for some $\theta_j$.
\end{enumerate}
In the former case, the limit of the one-parameter subgroup $(E, \lambda \cdot \phi)$ is $(E, 0)$ (or $(L \oplus L^{-1}, 0)$ in case \eqref{item:nontrivialextension} of \cref{Hitchin:classificationvb}), and so it lies in $\NBun$, which is isomorphic to $\mathbb{P}^3$ by \cite{NarasimhanRamanan69}.
In the latter case, $(E, \lambda \cdot \phi)$ is isomorphic to 
\begin{equation}\label{isotropicsliceTheta}
 \bigg(\theta^{-1}_j \oplus \theta_j, \, \lambda \cdot \begin{pmatrix}
0 & 1\\
u & 0\\
\end{pmatrix}\bigg) \simeq \bigg(\theta^{-1}_j \oplus \theta_j, \, \begin{pmatrix}
0 & 1\\
\lambda^2 u & 0\\
\end{pmatrix}\bigg)
\end{equation}
for some $u\in \Hom(\theta_j^{-1},\theta_j \otimes K_C)$, after normalizing with the group of diagonal automorphisms of $E$; see \cite[\S 11]{Hitchin1987}. Therefore, the locus of $\Gm$-fixed Higgs bundles with underlying unstable vector bundles is given by $\Theta$ (which corresponds to $u=0$).
\color{black}
\end{proof}


In \cref{prop:fibresf} we mentioned that the $\Gm$-invariant fibre $f^{-1}(\Omega_j) \subset \widetilde{M}$ over $\Omega_j \simeq (L \oplus L, 0)$, with $L^2 \simeq \mathcal{O}_C$, is the Grassmannian of Lagrangian subspaces of the 4-dimensional symplectic vector space $(V, \omega_V)$.

    The deformation theory of Higgs bundles gives the identification of $(V, \omega_V)$ with the space of Higgs bundles extensions of $(L,0)$ by itself, namely   \[\mathrm{Ext}^1_{\mathrm{Higgs}}(L,L)\simeq H^0(C, K_C) \oplus H^1(C, \mathcal{O}_C) \simeq H^1(C, \CC),\]
    endowed with the symplectic form given by cup product. For further details, we refer the interested reader to \cite[\S 3.2.2]{Felisetti2018}. We just observe that $H^0(C, K_C)$ parametrises deformations of $L$ with fixed underlying line bundle, while $H^1(C, \mathcal{O}_C)$ parametrises deformations of $L$ with fixed underlying Higgs field. Therefore, the rescaling action of Higgs fields yields the $\Gm$-action on $\mathrm{Ext}^1_{\mathrm{Higgs}}(L,L)$ defined by
    $\lambda \cdot (v,\bar{v})=(\lambda v, \bar{v})$, where $v \in H^0(C, K_C)$ and $\bar{v} \in H^1(C, \mathcal{O}_C)$. This in turn induces the $\Gm$-action on $\tomega_j$, whose fixed loci are described in the next \cref{prop:fixedlocusG}.
\begin{prop}[Fixed locus of $\tomega_j$] \label{prop:fixedlocusG}
The fixed locus of the $\Gm$-action on $\tomega_j$ is
\[
\Fix(\tomega_j)= t_j \sqcup s^+_j \sqcup T_j,
\]
where
\begin{enumerate}
\item the points $t_j$ and $T_j$ correspond to the Lagrangian subspaces $H^0(C, K_C)$ and $H^1(C, \mathcal{O}_C)$;
\item the curve $s^+_j$ parametrises Lagrangian subspaces generated by $v_1 \in H^0(C, K_C)$ and $v_2 \in H^1(C, \mathcal{O}_C)$, and it is isomorphic to $\PP^1$.
\end{enumerate}
In particular, $t_j$, $s^+_j$ and $T_j$ have weights $(1,1,1)$, $(-1,0,1)$ and $(-1,-1,-1)$ respectively. 
\end{prop}
\begin{proof}
 The Pl\"{u}cker polarization $H_j$ embeds $\tomega_j$ as a smooth quadric in the linear system $|H_j| = \PP(W) \simeq \mathbb{P}^4 \subset \mathbb{P}(\bigwedge^2 V)$. The $\Gm$-action on $\tomega_j$ induces an action on $W$ with weights $(0,1,1,1,2)$ in suitable coordinates $(x_0,\ldots,x_4)$. In these coordinates, $\tomega_j$ is defined by the equation $x_1^2+x_2x_3+x_0x_4=0$. 
 
 Since the Pl\"{u}cker embedding is $\Gm$-equivariant, the fixed loci of $\tomega_j$ are the intersections of $\tomega_j$ with the isotypic components of the $\Gm$-representation $W$, i.e.\
 \begin{align*}
     t_j & =[1:0:0:0:0],\\
     s_j^+& = [0:x_1:x_2:x_3:0] \cap \tomega_j = \{x_1^2+x_2x_3 =0\}\simeq \PP^1,\\
     T_j & = [0:0:0:0:1].
 \end{align*}
 
 Moreover, the tangent space
 \[
     T_{t_j}\tomega_j \simeq T_{t_j}\PP(W)/N_{\tomega_j/\PP(W), t_j} \simeq \Hom(t_j, W/\langle t_j, T_j\rangle)
\]
has weight $(1,1,1)$. Analogously, if $p = [0:0:1:0:0] \in s^{+}_j$, then
$
T_{p}\tomega_j \simeq \Hom(p, \langle t_j, \partial_{x_1}, T_j\rangle)
$
has weight $(-1,0,1)$, while $T_{T_j}\tomega_j \simeq \Hom(T_j, W/\langle t_j, T_j\rangle)$ has weight $(-1,-1,-1)$. 
\end{proof}

\begin{prop}[Fixed locus of $\widetilde{M}$] \label{prop:fixedlocusMtilde}
The fixed locus of the $\Gm$-action on $\widetilde{M}$ is
\[
\Fix(\widetilde{M})= \widetilde{\NBun} \sqcup \widetilde{S}^+ \sqcup \widetilde{\Theta} \sqcup \bigsqcup^{16}_{j \in 1} T_j,
\]
where
\begin{enumerate}
\item $\widetilde{\NBun} \coloneqq f^{-1}_*\NBun$ is the strict transform of $\NBun$, isomorphic to $\PP^3$;
\item $\widetilde{S}^+$ is a Kummer surface;
\item $\widetilde{\Theta} \coloneqq f^{-1}(\Theta)$;
\item $T_j$ are points lying on the Lagrangian Grassmannians $\tomega_j = f^{-1}(\Omega_j)$.
\end{enumerate}
\end{prop}
\begin{proof}
First, observe that $\Fix(\widetilde{M})$ lies over $\Fix(M)$, and so
\[\widetilde{N} \sqcup \widetilde{\Theta} = f^{-1}_*\Fix(M) \subset \Fix(\widetilde{M}) \subset f^{-1} (\Fix(M)).\]
The component of $\Fix(\widetilde{M})$ not contained in $f^{-1}_*\Fix(M)$ lies over $\Fix(M) \cap \Sigma \coloneqq S \subset \NBun$, which is isomorphic to $\Pic^0(C)/(\mathbb{Z}/2\mathbb{Z})$, i.e.\ the singular Kummer surface associated to $\Pic^0(C)$.

The fibre of $f$ over $p \in S \setminus \Omega$ is isomorphic to $\mathbb{P}^1$, and $\Gm$ acts with non-trivial weight on it by \cref{prop:weight}. Therefore, the $\PP^1$-bundle $f^{-1}(S \setminus \Omega)$ has two $\Gm$-fixed sections. We denote their closure by $\widetilde{S}^{-}$ and $\widetilde{S}^{+}$. Since the restriction of $f$ to $\widetilde{\NBun}$ is an isomorphism,  one of the two sections, say $\widetilde{S}^{-}$, lies in $\widetilde{\NBun}$. The same holds for one of the two fixed points in each $\tomega_j$, namely $t_j$ because of the weight considerations in \cref{prop:fixedlocusG} and  \cref{prop:weight}. 

The other section $\widetilde{S}^{+}$ must be the union of a copy of $S \setminus \Omega$ and the rational curve $s^{+}_j$, with $j=1, \ldots, 16$, thus isomorphic to the nonsingular Kummer surface associated to $\Pic^0(C)$. Indeed, by construction $\widetilde{S}^{+} \cap \tomega_j$ is a non-empty component of $\Fix(\tomega_j)$ different from a point; otherwise $\widetilde{S}^{+}$ would be singular, which is a contradiction since $\widetilde{S}^{+}$ is a fixed locus of a $\Gm$-action on a smooth manifold. Therefore, $\widetilde{S}^{+} \cap \tomega_j = s^{+}_j$ by \cref{prop:fixedlocusG}.

\end{proof}

\begin{prop}[Weights of $\widetilde{M}$]\label{prop:weight}
\begin{enumerate}
\item\label{item:wP3} $\widetilde{\NBun}$ has  weight $(0,0,0,1,1,1)$;
\item\label{item:wS+} $\widetilde{S}^+$ has weight $(-1,0,0,1,1,2)$;
\item\label{item:wTheta} $\widetilde{\Theta}_j$ and $\Theta_j$ have weight $(-1,-1,-1,2,2,2)$;
\item\label{item:wTj} $\widetilde{T}_j$ has weight $(-1,-1,-1,2,2,2)$;
\end{enumerate}
\end{prop}
\begin{proof}
Let $\widetilde{\omega}$ be the holomorphic symplectic form on the symplectic resolution $\widetilde{M}$ extending the canonical holomorphic symplectic form $\omega$ on the smooth locus of $M$.
As in \cite[Proposition 7.1]{Hitchin1987}, the $\Gm$-action rescales the holomorphic symplectic form $\widetilde{\omega}$ 
\[\lambda^*\widetilde{\omega} = \lambda \widetilde{\omega}.\] 

Let $p\in \mathrm{Fix}(\tm)$, and $W$ be a Lagrangian subspace of $T_p\tm$ with weights $(a,b,c)$. Then the isotropy condition yields an isomorphism 
$$W^*\simeq T_p\tm/W;$$
and the weights of the action on W become
$$\lambda (\lambda^{-a}, \lambda^{-b}, \lambda^{-c}) = (\lambda^{-a+1}, \lambda^{-b+1}, \lambda^{-c+1})$$
on $W^*$. As a result the torus action at a fixed point has weights $(a,b,c,-a+1,-b+1,-c+1)$.

In this way, the weights of the $\Gm$-action at $\widetilde{N}$, $\widetilde{S}^{+}$ and $\widetilde{T}_j$ follow immediately from the computations in \cref{prop:fixedlocusG}, by observing that $t_j \in \widetilde{N}$ and $s^{+}_j \in \widetilde{S}^{+}$. For the weights at $\widetilde{\Theta}_j$, instead, note that the locus of semistable Higgs bundles with underlying vector bundle $\theta^{-1}_j \oplus \theta_j$ is Lagrangian by definition of $\omega$ (cf \cite[Lemma 6.8]{Hitchin1987}), and has weight $(2,2,2)$ by \eqref{isotropicsliceTheta}.
\end{proof}

\begin{cor}\label{cor:tansverseintersection}
$\widetilde{\NBun}$ and $\tomega_j$ intersect transversely at the point $t_j = \widetilde{\NBun} \cap \tomega_j$.
\end{cor}
\begin{proof}
By \cref{prop:fixedlocusG}, the tangent space $T_{t_j}\tomega_j$ has weight one, while $T_{t_j}\widetilde{\NBun}$ has weight zero.
\end{proof}

\begin{cor}\label{cor:codimattr}
The attracting sets $\Attr(\NBun)$, $\Attr(\widetilde{S}^+)$, $\Attr(\widetilde{\Theta}_j)$ and $\Attr(\widetilde{T}_j)$ have codimension $0$,$1$,$3$,$3$ respectively.
\end{cor}
\begin{proof}
It is an immediate corollary of \eqref{eq:dimensiontangentspace} and \cref{prop:weight}.
\end{proof}

\subsubsection{Poincar\'{e} polynomials of $M$ and $\widetilde{M}$}

\begin{thm}[Cohomology of $M$ and $\widetilde{M}$]\label{thm:PoincarePoly}
The Poincar\'{e} polynomials of $M$ and $\widetilde{M}$ are
\begin{equation}\label{eq:Poincarepolynomial}
P_t(M)\coloneqq \sum_k (-1)^k\dim H^k(M)\, t^k=1+t^2+t^4+17 t^6    
\end{equation}
\begin{equation}\label{eq:Poincarepolynomialres}
P_t(\widetilde{M})\coloneqq \sum_k (-1)^k\dim H^k(\widetilde{M})\, t^k = 1+2t^2+23t^4+34t^6.
\end{equation}
\end{thm}
\begin{proof}
Since $\Attr(\NBun)$ is an open subset of $M$, we have $H^*(\Attr(\NBun), u^{!}\underline{\QQ}_{M}) = H^*(\Attr(\NBun))=H^*(\PP^3)$. The spectral sequence \eqref{eq:spseq1} gives
\begin{align*}
P_t(M)& =P_t(\NBun) + P_t(\Theta)t^6\\
&= 1+t^2+t^4+17 t^6.
\end{align*}
See \cite[Theorem 1.5]{DaskalopoulosWentworth10} for an alternative proof.

Similarly, by \cref{local-to-globalspseq}, \cref{prop:fixedlocusMtilde} and  \cref{cor:codimattr}, we obtain
\begin{align*}
P_t(\widetilde{M})& =P_t(\widetilde{\NBun}) + P_t(\widetilde{S}^+)t^2 + P_t(\widetilde{\Theta})t^6 + \sum^{16}_{j=1}P_t(T_j)t^6\\
& = (1+t^2+t^4+t^6)+(1+22 t^2 +t^4)t^2 + 16t^6 + 16t^6\\
& =1+2t^2+23t^4+34t^6.
\end{align*}
\end{proof}

\subsection{Moduli space of equivariant Higgs bundles $M_{\iota}$}\label{sec:moduliMiota}

\subsubsection{Equivariant Higgs bundle and the forgetful map $q$}\label{sec:orbibundlemoduli}
Recall that $\iota: C \to C$ is the hyperelliptic involution of the curve $C$ of genus 2.
\begin{defn}
A ($\iota$-)\textbf{equivariant Higgs bundle} over $C$ is a triple $(E, h, \phi)$ such that:
\begin{enumerate}
    \item $E$ is an $\iota$-invariant vector bundle, i.e.\ $\iota^*E \simeq E$;
    \item $h\colon E \to \iota^*E$ is a lift of the $\iota$-action on $E$ such that $\iota^*h \circ h =\id_E$;
    \item $\phi \in \Hom(E, E \otimes K_C)$ is an $\iota$-invariant Higgs field, i.e.\ a $\mathcal{O}_C$-linear morphism which makes the following diagram commutative:
 \vspace{0.2 cm}  
\begin{center}
\begin{tikzpicture}
\node (A) at (0,1.2) {$E$};
\node (B) at (2,1.2) {$E \otimes K_C$};
\node (C) at (0,0) {$\iota^*E$};
\node (D) at (2,0) {$\iota^*E\otimes K_C$.};
\draw[->,thick] (A) -- (C) node [midway,left] {$h$};
\draw[->,thick] (B) -- (D) node [midway,right] {$h \otimes \operatorname{id}_{K_C}$};
\draw[->,thick] (A) -- (B)node [midway,above] {$\phi$};
\draw[->,thick] (C) -- (D) node [midway,above] {$\iota^*\phi$};
\end{tikzpicture}
\end{center}
\end{enumerate}
A morphism between two equivariant Higgs bundles $(E_1, h_1, \phi_1)$ and $(E_2, h_2, \phi_2)$ is a homomorphism of vector bundles $\psi \in \Hom(E_1, E_2)$ such that the following diagrams commute:
 \vspace{0.2 cm}  
\begin{center}
\begin{tikzpicture}
\node (A) at (0,1.2) {$E_1$};
\node (B) at (2,1.2) {$\iota^*E_1$};
\node (C) at (0,0) {$E_2$};
\node (D) at (2,0) {$\iota^*E_2$.};
\node (E) at (4,1.2) {$E_1$};
\node (F) at (6,1.2) {$E_1 \otimes K_C$};
\node (G) at (4,0) {$E_2$};
\node (H) at (6,0) {$E_2 \otimes K_C$.};
\draw[->,thick] (A) -- (C) node [midway,left] {$\psi$};
\draw[->,thick] (B) -- (D) node [midway,right] {$\psi$};
\draw[->,thick] (A) -- (B)node [midway,above] {$h_1$};
\draw[->,thick] (C) -- (D) node [midway,above] {$h_2$};
\draw[->,thick] (E) -- (G) node [midway,left] {$\psi$};
\draw[->,thick] (F) -- (H) node [midway,right] {$\psi \otimes \id_{K_C}$};
\draw[->,thick] (E) -- (F)node [midway,above] {$\phi_1$};
\draw[->,thick] (G) -- (H) node [midway,above] {$\phi_2$};
\end{tikzpicture}
\end{center}
\end{defn}

The slope of a vector bundle $E$ over a curve $C$ is defined by $\mu(E)\coloneqq \deg(E)/\rank(E)$.
\begin{defn}
An equivariant Higgs bundle $(E, h, \phi)$ is \textbf{semistable} or \textbf{stable}, if for any proper equivariant Higgs subbundle $F \subset E$, the inequality $\mu(F)\leq \mu(E)$ holds, respectively $\mu(F)< \mu(E)$.
\end{defn}
Let $W = \{w_1, \ldots, w_6\}$ be the set of all Weierstrass points, i.e.\ the fixed points of $\iota$. For every $w \in W$, $h_w\colon  E_w \to E_w$ is an involution of the fibre $E_w$. 
\begin{defn}
The normal quasi-projective variety $M_{\iota}$ (respectively $M^{s}_{\iota}$) is the coarse \textbf{moduli space of semistable} (respectively stable) \textbf{equivariant Higgs bundle} $(E, h, \phi)$ of rank $2$ over $C$ with trivial determinant and $\Tr(h_{w})=0$ for all $w \in W$.
\end{defn} 
The existence of $M_{\iota}$ and $M^{s}_{\iota}$ follows from the work of Seshadri \cite{Seshadri70} and Nitsure \cite{Nitsure91}. In Section \ref{sec:constructionMi} we review the construction. Here we first describe $M_{\iota}$ as a quasi-\'{e}tale cover of $M$. This cover appears also in \cite[\S 6.3]{HeuLoray19} and references therein.

\begin{defn}[Quasi-\'{e}tale morphism]
A morphism $f\colon  X \to Y$ between normal varieties is \textbf{quasi-\'{e}tale} if $f$ is quasi-finite, surjective and \'{e}tale in codimension one, i.e.\ there exists a closed, subset $Z \subseteq X$ of codimension $\codim Z \geq 2$ such that $f|_{X \setminus Z}\colon  X \setminus Z \to Y$ is \'{e}tale.
\end{defn}
\begin{rmk}\label{rmk:purity}
By the purity of branch locus, a quasi-\'{e}tale morphism induces an \'{e}tale cover of the smooth locus of the codomain.
\end{rmk}

\begin{prop}\label{prop:forgetfulmap}
The forgetful map
\begin{align*}
    q \colon  M_{\iota} &\to M\\
     (E, h, \phi) &\mapsto (E, \phi)
\end{align*}
is well-defined, quasi-\'{e}tale of degree two, and branched along the singular locus $\Sigma$ of $M_{\iota}$.
\end{prop}
\begin{proof}
 The forgetful map $q$ is well-defined, because an equivariant Higgs bundle $(E, h, \phi)$ is semistable if and only if the Higgs bundle $(E, \phi)$ is semistable in the usual sense (the same proof of \cite[Lemma 2.7]{Biswas05} applies). 
 The map $q$ is also surjective: any semistable Higgs bundles $(E, \phi)$ admits a lift of the $\iota$-action on $E$ conjugating $\phi$ and $\iota^*\phi$ by \cite[Chapter 6, p.74, and Theorem 2.1]{HeuLoray19}.
 
We show now that $q$ is quasi-\'{e}tale. To this end, we closely follow the proof of Theorem 2.1 in \cite{Kumar00}. Given two equivariant Higgs bundles $(E, h_1, \phi)$ and $(E, h_2, \phi)$, there exists an automorphism $A \in \Aut(E)$ such that $h_2=h_1 \circ A$ and $\phi=A^{-1}\phi A$. 

If $(E, \phi)$ is stable, then the only automorphisms which fix the Higgs field are scalars. Then $h_2=\pm h_1$, and so there are only two non-equivalent equivariant Higgs bundles $(E, h_1, \phi)$ and $(E, -h_1, \phi)$ over $(E, \phi)$. Hence, $q$ is generically $2:1$.

If $(E, \phi)$ is strictly semistable, i.e.\ $(E, \phi) \in \Sigma$, then $E \simeq L \oplus L^{-1}$ with $L \in \Pic^0(C)$, and any two lifts are equivalent. Hence, $q$ is quasi-finite and branched along $\Sigma$.
\end{proof}

\subsubsection{Non-abelian Hodge correspondence}
Let $C \to \PP^1$ be the quotient of $C$ via the hyperelliptic involution, and $\underline{W}$ be the critical divisors on $\PP^1$, i.e.\ the projection of the Weierstrass points.

The moduli space $M_{\iota}$ is isomorphic to the moduli space of parabolic Higgs bundle of rank $2$ on $\PP^1$ with parabolic weight $1/2$ at all points of $\underline{W}$ and parabolic degree zero; see \cite[Theorem 3.5]{BIM13}.

The topological space underlying $M_{\iota}$ parametrises also representations of the orbifold fundamental group
\[
\pi^{\mathrm{orb}}_1(C/\iota) \simeq \langle \gamma_1, \ldots, \gamma_6\, | \, \gamma^2_1 = \ldots = \gamma^2_6 =1 \text{ and }\gamma_1 \ldots \gamma_6 =1 \rangle.
\]
\begin{thm}[Non-abelian Hodge correspondence]\label{thm:nonabelianHodgecorrespondenceforMi}
There exists a commutative square
\begin{center}
\begin{tikzpicture}
\node (A) at (0,1.2) {$M_{\iota}$};
\node (C) at (0,0) {$M$};
\node (B) at (5,1.2) {${M}_B(2,\Sl_2, \iota) \coloneqq \Hom(\pi^{\mathrm{orb}}_1(C/\iota), \Sl_2)\sslash \PGl_2$};
\node (D) at (5,0) {${M}_B(2, \Sl_2) = \Hom(\pi_1(C), \Sl_2)\sslash \PGl_2$.};
\draw[->,thick] (A) -- (C) node [midway,left] {$q$};
\draw[->,thick] (B) -- (D) node [midway,right] {$q^{\mathrm{top}}$};
\draw[->,thick] (A) -- (B)node [midway,above] {$\Psi_{\iota}$};
\draw[->,thick] (C) -- (D) node [midway,above] {$\Psi$};
\end{tikzpicture}
\end{center}
where the horizontal arrows are real analytic isomorphisms, and the vertical arrows are quasi-\'{e}tale covers.
\end{thm}
\begin{proof}
Identify $M_{\iota}$ with a moduli space of parabolic Higgs bundles as above. 
The correspondences $\Psi$ and $\Psi_{\iota}$ have been constructed by Hitchin \cite{Hitchin1987} and Simpson \cite{Simpson1990} respectively. By construction, the square commutes. 
\end{proof}

\subsubsection{Singularities of $M_{\iota}$}
\begin{notation}
We fix the following notation:
\begin{itemize}
        \item $\mathrm{Bun}^{ss}(C/\iota)$ is the moduli space of semistable $\iota$-equivariant vector bundles $(E, h)$. It is the inverse image of the moduli space of semistable vector bundles $\NBun$ via $q$;
        \item the inverse images of $\Omega$ via $q$ consists of the 16 points $\Omega_{\iota}$;
        \item the inverse images of $\Theta$ via $q$ consists of the 32 points $\Theta_{\iota}$.
\end{itemize}
\end{notation}

\begin{prop}[Singularities of $M_{\iota}$]\label{prop:singMiota}
$\quad$
 
\begin{enumerate}
    \item\label{item:singlocus} $\Omega_\iota$ is the singular locus of $M_{\iota}$.
\item\label{item:smoothlocus} The smooth locus of $M_{\iota}$, denoted $M_{\iota}^{\mathrm{sm}}$, is the moduli space of stable equivariant Higgs bundles $M^{s}_{\iota}$.
\end{enumerate}
\end{prop}
\begin{proof}
The local isomorphism type of the singularities of $M_{\iota}$ coincides with the model described in \cite[Lemma 3.1]{MRS18}. This yields the first statement. 
For the second statement, it is enough to show that
\[M_{\iota}^{\mathrm{sm}} = q^{-1}(M \setminus \Sigma) \cup q^{-1}(\Sigma \setminus \Omega) \subseteq M^{s}_{\iota}.\]
Any Higgs bundle $(E, \phi) \in M \setminus \Sigma$ is stable, and so the equivariant Higgs bundles in $q^{-1}(M \setminus \Sigma)$ are stable too. If $(E, \phi) \in \Sigma \setminus \Omega$ with $E \simeq L \oplus L^{-1}$, then the only line sub-bundles of $E$ are $L$ and $L^{-1}$, but since they are not $\iota$-invariant, $q^{-1}(\Sigma \setminus \Omega) \subseteq M^{s}_{\iota}$.
\end{proof}

\subsubsection{Construction of $M_{\iota}$}\label{sec:constructionMi}
The moduli space $M_{\iota}$ is constructed in the following way. All the ingredients have already appeared in \cite{Seshadri70, Nitsure91, Hausel98}.

Let $(E,h,\phi)$ be a stable equivariant Higgs bundle of rank $2$ over $C$ with trivial determinant and $\Tr(h_{w})=0$ for all $w \in W$. Fix an equivariant ample line bundle $\mathcal{O}_C(1)$ on $C$. Choose an integer $m \in \ZZ$ such that $H^1(C, E(m))=0$ and $E(m)$ is globally generated. 

The quot scheme $Q$ parametrizes all quotient sheaves of $H^0(C, E(m)) \otimes \mathcal{O}_C$ with the Hilbert polynomial of $E(m)$. Let $H^0(C, E(m)) \otimes  p_C^*\mathcal{O}_{C} \to \mathcal{E}_Q \otimes p_C^*\mathcal{O}_C(m)$ be the universal quotient bundle on $Q \times C$, with the natural projection $p_C \colon Q \times C \to C$. Let $R \subset Q$ be the subset of all $q \in Q$ for which $\mathcal{E}_q$ is locally free and the map $H^0(C, E(m)) \to H^0(C, \mathcal{E}_q(m))$ is an isomorphism.

By \cite[Proposition 3.6]{Nitsure91}, there exists a locally universal family of semistable Higgs bundles $\mathcal{E}_{ss} \xrightarrow{\Phi_{ss}} \mathcal{E}_{ss} \otimes p_C^*K_C$ on $F_{ss} \times C$, where $F_{ss}$ is an open subset of a linear $R$-scheme $F \to R$ together with a family of Higgs bundles $\mathcal{E}_F \xrightarrow{\Phi_F} \mathcal{E}_F \otimes p_C^*K_C$.

The involution $\iota^*$ on $H^0(C, E(m))$ induces a natural lift $j_0$ of the $\iota$-action on the trivial bundle $C \times H^0(C, E(m))$, and so an $\iota$-action on $F_{ss}$ with fixed locus $\mathrm{Fix}_{\iota}(F_{ss})$. In particular, $j_0$ descends to a lift $h_q$ of the $\iota$-action on $\mathcal{E}_q$, for any $q \in \mathrm{Fix}_{\iota}(F_{ss})$. Call $F_{ss, \iota}$ the connected component of $\mathrm{Fix}_{\iota}(F_{ss})$ consisting of the equivariant Higgs bundles $(\mathcal{E}_q,h_q, \Phi_q)$ with $\Tr(h_{q, w})=0$; see \cite[Chapter II, Proposition 6.(iv)]{Seshadri70} and \cite[Chapter II, Proposition 5 and Remark 2]{Seshadri70}. 

Let $H$ be the group of automorphism of the trivial bundle which commutes with $j_0$, and $PH\coloneqq H/\Gm$ the quotient of $H$ modulo scalar matrices. The moduli spaces $M_{\iota}$ and $M_{\iota}^s$ are the quotients $F_{ss, \iota}\sslash PH$ and $F_{s, \iota}/PH$ respectively, where $F_{s, \iota}$ is the subset of stable equivariant Higgs bundles in $F_{ss, \iota}$.

\subsubsection{Universal bundles}\label{sec:Universalbundle}
We show the existence of a universal bundle on $M^{\text{sm}}_{\iota} \times C$ (cf \cite[\S 5]{Hausel98}).


\begin{defn}
Let $Z$ be a subset of $M^{\mathrm{sm}}_{\iota}$. A \textbf{universal Higgs bundle} on $Z \times C$ is a rank two Higgs bundle $(\mathbb{E}, \Phi)$ such  
$(\mathbb{E}, \Phi)|_{\{(E, h, \phi)\}\times C}\simeq (E, \phi)$ for all $(E, h, \phi)\in Z$. 
\end{defn}

\begin{rmk}\label{rmk:tensorization}
Let $(\mathbb{E}_1, \Phi_1)$ and $(\mathbb{E}_2, \Phi_2)$ be universal Higgs bundles on $Z \times C$. Then there exists a line bundle $\mathcal{L} \in \Pic(Z)$ such that $(\mathbb{E}_1, \Phi_1) \simeq (\mathbb{E}_2 \otimes p_C^*\mathcal{L}, \Phi_2)$, with $p_C\colon Z \times C \to C$ the natural projection. In particular, $\mathbb{P}(\mathbb{E}_1)\simeq \PP(\mathbb{E}_2)$ is canonical. See \cite[4.2]{HauselThaddeus04}.
\end{rmk}

We adopt the notation of \cref{sec:constructionMi}. In addition, define $F^{\circ}_{\iota}$ the open subset of $F_{s, \iota}$ parametrizing stable equivariant Higgs bundle whose underlying vector bundle is either stable or isomorphic to $L \oplus \iota^*L$ with $L \in \Pic^0(C)$ with $L^2 \not\simeq \mathcal{O}_C$. 

The quotient $M^{\circ}_{\iota} \coloneqq F^{\circ}_{\iota}/PH$ is the attracting set of $\mathrm{Bun}^{ss}(C/\iota) \setminus \Omega_{\iota}$. Thus, according to \cref{Hitchin:classificationvb},
the complement $M_{\iota} \setminus M^\circ_{\iota}$ parametrises stable equivariant Higgs bundles whose underlying vector bundle is unstable or a non-trivial extension of $L$ by $L$ with $L^2 \simeq \mathcal{O}_C$, and so it has codimension 2 by \cite[Example 3.13 (iv) and (v)]{Hitchin1987}; see also \cite[Lemma 3.4]{KiemYoo08}. In particular, $F_{s, \iota} \setminus F^{\circ}_{\iota}$ has codimension 2.

%
\begin{prop}
A universal Higgs bundle on $M^{\text{sm}}_{\iota} \times C$ does exist.
\end{prop}
\begin{proof}

Let $\mathcal{E}$ be the restriction of the universal Higgs bundle $\mathcal{E}_F$ to $F_{s, \iota} \times C$, and denote by $p_F\colon F_{s, \iota} \times C\to F_{s, \iota}$ and $p_C\colon F_{s, \iota} \times C\to C$ the two projections. 

The natural lift of the $H$-action is such that the subgroup of scalar matrices acts by homotheties. Suppose that there exists an $H$-equivariant line bundle $\lambda(\mathcal{E})$ over $F_{s,\iota}$ with the same property, i.e.\ that the centre of $H$ acts by homotheties. Then, the centre of $H$ acts trivially on $\mathcal{E} \otimes p_F^*\lambda(\mathcal{E})^{-1}$. By Kempf's descent lemma \cite[Theorem 2.3]{DN89}, the $PH$-equivariant bundle $\mathcal{E} \otimes p_F^*\lambda(\mathcal{E})^{-1}$ descends to a vector bundle on $M^{\mathrm{sm}}_{\iota} \times C$, and since the section $\Phi$ is invariant, it also descends. 

Here is how to construct $\lambda(\mathcal{E})$. For any $(E, h, \phi) \in F^{\circ}_{\iota}$, $h$ acts on $H^{0}(C, E \otimes K_C)$, and induces a splitting
\[
H^{0}(C, E \otimes K_C)= H^0(C, E \otimes K_C)^{+} \oplus H^0(C, E \otimes K_C)^{-}
\]
into one-dimensional eigenspaces (relative to eigenvalues $\pm 1$ respectively); see \cite[Proposition 4.1]{HeuLoray19}. The lift $j_0$ induces an involution on $p_{F, *}(\mathcal{E} \otimes p^*_C K_C)$. Hence, set 
\[\lambda(\mathcal{E})^{\circ}\coloneqq p_{F, *}(\mathcal{E} \otimes p^*_C K_C)^{+}\]
as the $j_0$-invariant subsheaf of $p_{F, *}(\mathcal{E} \otimes p^*_C K_C)$. 
By semicontinuity, $\lambda(\mathcal{E})^{\circ}$ is a line bundle on $F^{\circ}_{ \iota}$ with fibre $H^0(C, E \otimes K_C)^{+}$. The multiplication by a scalar in $E$ induces multiplication in $H^0(C, E \otimes K_C)^{+}$ too, and so in $\lambda(\mathcal{E})^{\circ}$. Now let $i_{F^{\circ}_{\iota}}: F^{\circ}_{\iota} \hookrightarrow F_{\iota}$ be the natural inclusion, and define
\[\lambda(\mathcal{E})= i_{F^{\circ}_{\iota}, *} \lambda(\mathcal{E})^{\circ}.\]
Since $F_{s, \iota}$ is smooth and $F_{s, \iota} \setminus F^{\circ}_{\iota}$ has codimension 2, $\lambda(\mathcal{E})$ is a line bundle on $F_{s, \iota}$ with the right $H$-linearisation.
\end{proof}

\subsubsection{Nilpotent cone}\label{sec:nilpotent}
In this section we describe the components of the nilpotent cone of $M$, i.e.\ the zero fibre of the Hitchin fibration $\chi\colon M \to H^0(C, K_C^{\otimes 2})$. 

We show that $\chi^{-1}(0)$ has 17 irreducible components, one of them being the moduli space of semistable vector bundle $\NBun$. By \cite[Main Theorem, \S 3]{NR69} there is no universal Higgs bundle over any Zariski open set of $\NBun$. On the other hand, we construct a universal bundle on the normalization of the other components; see \cref{prop:universalbundleRj} and \cref{lem:equivalenceP1bundles}. 
\begin{prop}\label{prop:nilpotentconeM}
The nilpotent cone of $M$ is a compact union of 3-dimensional manifolds:
\[
\chi^{-1}(0)= \NBun \sqcup \bigsqcup^{16}_{j=1} \NBun_j,
\]
where $\NBun_j$ is isomorphic to the vector space $\operatorname{Ext}^1(\theta_j, \theta^{-1}_j)$, where $\theta_j$ runs over the 16 theta-characteristics $\theta_j^2=K_C$.
\end{prop}
\begin{proof} We adapt the proof of \cite[Proposition 19]{Thaddeus89}; see also \cite[\S 2]{PalPauly2018}. Since $\NBun \subset M$ is the locus of semistable Higgs bundles with trivial Higgs field, we see that $\NBun \subset \chi^{-1}(0)$. However, there are also stable Higgs bundles $(E, \phi) \in \chi^{-1}(0)$ with $\phi \neq 0$. 

Under this assumption, $\phi$ has generically rank one: denote by $A$ the line bundle $\mathrm{Im}\phi \subset E \otimes K_C$. Then $E$ sits in the following diagram
\begin{center}
\begin{tikzpicture}
\node (A) at (0,1.2) {$0$};
\node (B) at (2,1.2) {$A^{-1}$};
\node (C) at (4,1.2) {$E$};
\node (D) at (6,1.2) {$A$};
\node (E) at (8,1.2) {$0$}; 
\node (F) at (0,0) {$0$};
\node (G) at (2,0) {$A\otimes K_C$};
\node (H) at (4,0) {$E \otimes K_C$};
\node (I) at (6,0) {$A^{-1} \otimes K_C$};
\node (L) at (8,0) {$0.$}; 
\draw[->,thick] (A)--(B);
\draw[->,thick] (B)--(C);
\draw[->,thick] (C)--(D);
\draw[->,thick] (D)--(E);
\draw[->,thick] (F)--(G);
\draw[->,thick] (G)--(H);
\draw[->,thick] (H)--(I);
\draw[->,thick] (I)--(L);
\draw[->,thick] (C)--(H)node [midway,left] {$\phi$};
\draw[->,thick] (D)--(H);
\draw[dashed,->,thick] (D) -- (I)node [midway,right] {$u$};
\end{tikzpicture}
\end{center}
Since $\Tr(\phi)=0$, the composition $A \to E \otimes K_C \to A \otimes K_C$ is zero, and the inclusion $A \to E \otimes K_C$ factors through $u\colon A \to A^{-1} \otimes K_C$. The stability of $E$ implies $- \deg A < \deg E =0$, and since $u \in H^0(C, K_C \otimes A^{\otimes(-2)})$ is non-zero, we conclude that $A$ is a theta-characteristic.

Therefore, the Higgs bundle $(E, \phi)$ is determined by the triple $(\theta_j, v, u)$ given by
\begin{itemize}
\item the theta-characteristic $\theta_j$,
\item the extension class $v \in \mathrm{Ext}^1(\theta_j, \theta^{-1}_j)$ giving the exact sequence $\theta^{-1}_j \to E\to \theta_j$,
\item the non-zero scalar $u \in H^0(C, \Hom(\theta_j, \theta^{-1}_j \otimes K_C)) \simeq H^0(C, \mathcal{O}_C)$,
\end{itemize}
modulo the $\Gm$-action 
\[
c \cdot (\theta_j, v, u)=(\theta_j, cv, cu).\]
The equivalence class $(\theta_j, v, u)$ under rescaling is denoted $[\theta_j, v, u]$, and we identify the Higgs bundle $(E, \phi) \in \NBun_j$ with $[\theta_j, v, u]$. In particular, the irreducible components of $\chi^{-1}(0)$ different from $N$ are
\[
\NBun_j \coloneqq \PP(\mathrm{Ext}^1(\theta_j, \theta^{-1}_j) \oplus H^0(C, \mathcal{O}_C))\setminus \{u=0\} \simeq \mathrm{Ext}^1(\theta_j, \theta^{-1}_j).
\]
\end{proof}

\begin{proof}[Alternative proof] The nilpotent cone on $M$ is the union of the repelling sets of all the fixed loci
\[
\chi^{-1}(0)= \mathrm{Repell}(\NBun) \cup \mathrm{Repell}(\Theta) = \NBun \cup \bigcup^{16}_{j=1}\mathrm{Repell}(\Theta_j).
\]
By \cref{BBdecomposition}.(\ref{limitmap}), $\mathrm{Repell}(\Theta_j)$ is isomorphic to a 3-dimensional vector space. However, we rely on the previous proof for a modular interpretation of $\mathrm{Repell}(\Theta_j)$. 
\end{proof}
Let $\Rcomp_j$ be the total space of the projective bundle 
$\PP(\mathrm{Ext}^1(\theta_j, \theta^{-1}_j) \oplus H^0(C, \mathcal{O}_C))$ with hyperplane bundle $\mathcal{O}_{\Rcomp_j}(1)$.
As we observed above, there is a natural decomposition 
$
\Rcomp_j = \NBun_j \cup \PP(\mathrm{Ext}^1(\theta_j, \theta^{-1}_j)).
$
The inclusion $\NBun_j \hookrightarrow \chi^{-1}(0)$ extends to a bijective and algebraic morphism
\begin{align*}
r_j \colon \Rcomp_j & \hookrightarrow \chi^{-1}(0)\\
[v:u] & \mapsto [\theta_j, v, u]=(E, \phi),
\end{align*}
whose image is the closure $\overline{\NBun}_j$ of $\NBun_j$ in $\chi^{-1}(0)$; see Theorem \ref{BBdecomposition}.(\ref{limitmap}) and also \cite[Proposition 24]{Thaddeus89}.

\begin{prop}\label{prop:universalbundleRj}
There exists a universal bundle $\mathbb{E}_{\Rcomp_j}$ on $\Rcomp_j \times C$ which sits in the following exact sequence
\[
0 \to p_{\Rcomp_j}^* \mathcal{O}_{\Rcomp_j}(1) \otimes p_C^*\theta^{-1}_j \to \mathbb{E}_{\Rcomp_j} \to p_C^*\theta_j \to 0,
\]
where $p_{\Rcomp_j}\colon \Rcomp_j \times C \to \Rcomp_j$ and $p_C\colon \Rcomp_j \times C \to C$ are the natural projections.
\end{prop}
\begin{proof}
 Mutatis mutandis, the same argument of \cite[p.22]{Thaddeus89} works.
\end{proof}

Consider now the quasi-\'{e}tale cover $q\colon M_{\iota} \to M$. Since $\NBun_j$ is simply connected, $q^{-1}(\NBun_j)$ breaks into two irreducible components, say $\NBun^{+}_j$ and $\NBun^{-}_j$. In particular, $q$ restricts to an isomorphism between $M^{\circ}_{\iota} \cap \NBun^{+}_j$ (equivalently $M^{\circ}_{\iota} \cap \NBun^{-}_j$) and $\overline{\NBun}_j \setminus( \Omega \cup \Theta_j)$. If we set $\Rcomp^{\circ}_j \coloneqq \Rcomp_j \setminus \big(r^{-1}_j(\Omega) \cup r^{-1}_j(\Theta_j)\big)$, then the product map \begin{equation}\label{eq:Rj}
\mathfrak{r}_j \coloneqq (r_j \circ q^{-1}, \id)\colon\Rcomp^{\circ}_j \times C \to (M^{\circ}_{\iota} \cap \NBun^{+}_j) \times C
\end{equation}
is an algebraic bijection. 
Let $\mathbb{E}$ be the universal bundle on $M^{\circ}_{\iota} \times C$. 
\begin{lem}\label{lem:equivalenceP1bundles}
The $\PP^1$-bundles $\PP(\mathfrak{r}^*_j\mathbb{E})$ and $\PP(\mathbb{E}_{\Rcomp_j})$ on $\Rcomp^{\circ}_j \times C$ are isomorphic.
\end{lem}
\begin{proof} The vector bundles $\mathfrak{r}^*_j\mathbb{E}$ and $\mathbb{E}_{\Rcomp_j}$ are both universal on $\Rcomp^{\circ}_j \times C$. The result follows from \cref{rmk:tensorization}.
\end{proof}

\subsubsection{Quasi-\'{e}tale covers of $\MDolSl$}\label{sec:quasi-etalecover}
In this section we show that the quasi-\'{e}tale cover $\iota$ is a special feature of the moduli space $M={M}_{\mathrm{Dol}}(C,\Sl_2)$, which is not shared by any other space $\MDolSl$, $g \geq 2$. 

\begin{prop}\label{lem:fundamentalgroupsmoothlocus}
The smooth locus $\MDolSlsm$ of $\MDolSl$ is simply-connected for $g \geq 2$ and $(g,n)\neq (2,2)$. In particular, 
\[\pi_1 (M^{\mathrm{sm}})=\Z/2\Z.\]
\end{prop}
\begin{proof} $\MDolSlsm$ contains a Zariski open subset which can be identified with the cotangent bundle of the moduli space $\NSl(X,n)$ of stable vector bundles of rank $r$ and trivial determinant over $X$.
Therefore, the fundamental group of $\MDolSlsm$ is a quotient of $\pi_1(\NSl(X,n))$, which is trivial by \cite[Theorem 3.2.(i)]{DaskalopoulosUhlenbeck1995}, for $g \geq 2$ and $(g,n)\neq (2,2)$.

Consider now $M$. The forgetful map $q$ induces the following exact sequence in homotopy
\begin{equation}\label{eq:homotopyexactsequence}
1 \to \pi_1(M^{\mathrm{sm}}_{\iota}) \to \pi_1(M^{\mathrm{sm}}) \to \Z/2\Z \to 1.
\end{equation}
As before, $M^{\mathrm{sm}}_{\iota}$ contains a Zariski open subset isomorphic to the cotangent bundle of the moduli space $\NSli$ of stable $\iota$-equivariant bundles of rank $2$ over $C$ with trivial determinant and $\Tr(h_{w})=0$ for all $w \in W$; see for instance \cite[Chapter 6, p.73]{HeuLoray19}. Thus, we obtain that
$\pi_1(M^{\mathrm{sm}}_{\iota})$ is a quotient of $\pi_1(\NSli)$.

The space $\NSli$ is the smooth locus of the double cover $\NSliss$ of $\PP^3$ branched along a singular Kummer quartic. The singular locus of $\NSliss$ consists of 16 ordinary double points, which are known to admit a small resolution, i.e.\ the exceptional locus has codimension $\geq 2$. This implies that $\pi_1(\NSli)$ coincides with the fundamental group of a (small) resolution of $\NSliss$. Further, $\NSliss$ is rational by \cite[Theorem 2.2]{Kumar00} or \cite[Theorem 1.3]{CPC19}; see also \cite[\S 5.4.2 and \S 5.5]{HeuLoray19}, where a small resolution of $\NSliss$ is denoted $\operatorname{Bun}^{ss}_{]\frac{1}{4}, \frac{1}{2}[}(C/i)$. Since the fundamental group is a birational invariant of smooth proper varieties, we observe that $\operatorname{Bun}^{ss}_{]\frac{1}{4}, \frac{1}{2}[}(C/i)$ is simply-connected, since the projective space is so.

To summarize, we have shown that
\[1= \pi_1(\PP^3)\simeq \pi_1(\operatorname{Bun}^{ss}_{]\frac{1}{4}, \frac{1}{2}[}(C/\iota)) \simeq \pi_1(\NSli) \twoheadrightarrow \pi_1(M^{\mathrm{sm}}_{\iota}).\]
By the exact sequence (\ref{eq:homotopyexactsequence}), we conclude that $\pi_1(M^{\mathrm{sm}}) \simeq \Z/2\Z$.
\end{proof}

 The following corollary is an immediate consequence of Remark \ref{rmk:purity} and Proposition \ref{lem:fundamentalgroupsmoothlocus}.

\begin{cor}
There are no non-trivial quasi-\'{e}tale cover of $\MDolSl$ for $g \geq 2$ and $(g,n) \neq (2,2)$. The forgetful map $q$ is the only non-trivial quasi-\'{e}tale cover of $M$.
\end{cor}
\section{P=W conjectures for $M$}\label{sec:P=W2}

In this section we
reduce the proof of the P=W conjecture for $M$ and $\widetilde{M}$ to P=W phenomena for the summands of the decomposition theorem for $f\colon \widetilde{M} \to M$; see \cref{P=W and PI=WI} and \cref{cor:splitting}. The exchange of the perverse and weight filtrations for the summands supported on a subvariety strictly contained in $M$ is proved in \cref{P=Wforsingularloci}. Therefore, the ultimate goal of this section is to reduce the proof of the P=W conjecture for $M$ and $\widetilde{M}$ to the PI=WI conjecture.

We first show that the PI=WI conjecture for $M$ implies the P=W conjecture for $M$. Actually, this first statement does not require the decomposition theorem. 

\begin{thm}\label{P=W and PI=WI}
If the PI=WI conjecture for $M$ holds, then the P=W conjecture for $M$ holds.
\end{thm}
\begin{proof}
The fixed locus of the $\Gm$-action on $M$ can be identified with the (disjoint) union of connected components of the fixed locus of the $\Gm$-action on $\widetilde{M}$; see \cref{prop:fixedlocusM} and \cref{prop:fixedlocusMtilde}. This induces an injective morphism between the local-to-global spectral sequences \eqref{eq:spseq1} for $M$ and $\widetilde{M}$. Therefore, $f^*\colon H^*(M) \to H^*(\widetilde{M})$ is an injective map, and so is the natural map $H^*(M) \to IH^*(M)$, since $f^*\colon H^*(M) \to H^*(\widetilde{M})$ factors as $H^*(M) \to IH^*(M) \to H^*(\widetilde{M})$. The statement now follows from the fact that the injective map $H^*(M) \to IH^*(M)$ preserves the perverse and weight filtrations.
\end{proof}

With a slight abuse of notation, we denote by $f$ both the symplectic resolutions $f_{\mathrm{Dol}}(C,\Sl_2) \colon \widetilde{M} \to M$ and $f_{\mathrm{B}}(C,\Sl_2) \colon \widetilde{M}_B \to M_B$. By \cite[Lemma 2.11]{Kaledin06}, any symplectic resolution is semismall. Therefore, the decomposition theorem (\cref{DecThm}) provides canonical isomorphisms:

\begin{equation}\label{eq:dtdol}
    \mathrm{R}f_*\underline{\QQ}_{\tm}[6]\simeq IC_{M}\oplus\underline{\QQ}_{\Sigma}[4](-1)\oplus \underline{\QQ}_{\Omega}(-3);
\end{equation}
\begin{equation}\label{eq:dtb}
    \mathrm{R}f_*\underline{\QQ}_{\widetilde{M}_B}[6]\simeq IC_{M_B}\oplus\underline{\QQ}_{\Sigma_B}[4](-1)\oplus \underline{\QQ}_{\Omega_B}(-3).
\end{equation}
Thus, in cohomology 
we have:
\begin{align}
    H^*(\widetilde{M})&\simeq IH^*(M)\oplus H^{*-2}(\Sigma)(-1)\oplus H^{*-6}(\Omega)(-3);\\ \label{eq:decompositioncohMtilde}
    H^*(\widetilde{M}_B)&\simeq IH^*(M_B)\oplus H^{*-2}(\Sigma_B)(-1)\oplus H^{*-6}(\Omega_B)(-3).
\end{align}
These decompositions split the perverse and weight filtration, as shown in the following lemmas.
\begin{lem}\label{lem:splittingp}
$$P_k H^*(\widetilde{M})=P_k IH^*(M)\oplus P_{k-1}H^{*-2}(\Sigma)\oplus P_{k-3}H^{*-6}(\Omega),$$
where $P_k H^*(\widetilde{M})$, $P_k IH^*(M)$, $P_{k}H^{*}(\Sigma)$ and $ P_{k}H^{*}(\Omega)$ denote the pieces of the perverse filtration associated to the maps $\chi \circ f$, $\chi$, $\chi|_{\Sigma}$ and $\chi|_{\Omega}$ respectively.

\end{lem}
\begin{proof}
Apply $\chi_*$ to the splitting \eqref{eq:dtdol} and notice that perverse truncation functors $^{\mathfrak{p}}\tau_{\leq i}$ are exact.
\color{black}
\end{proof}
\begin{lem}\label{lem:splittingw}
$$W_{2k} H^*(\widetilde{M}_B)=W_{2k}IH^*(M_B)\oplus W_{2k-2}H^{*-2}(\Sigma_B)\oplus W_{2k-6}H^{*-6}(\Omega_B)$$
\end{lem}
\begin{proof}
As the decomposition theorem is an isomorphism of mixed Hodge structures, we have 
$$W_{2k} H^*(\widetilde{M}_B)=W_{2k}IH^*(M_B)\oplus W_{2k}H^{*-2}(\Sigma_B)(-1)\oplus W_{2k}H^{*-6}(\Omega_B)(-3).$$
Recalling that Tate shifts $(-k)$ increase weights of $2k$, the result follows by including them in the grading of the weight filtration. 
\end{proof}
\begin{thm}\label{cor:splitting}
The $P=W$ conjecture for $\tm$ is equivalent to the following two statements:
\begin{enumerate}[1.]
\item PI=WI conjecture for $M$;
\item P=W conjecture for $\Sigma$ and $\Omega$, i.e.\ 
\[ P_{k}H^{*}(\Sigma)=\Psi|_\Sigma^* W_{2k}H^*(\Sigma_B), \qquad  P_{k}H^{*}(\Omega)=\Psi|_\Omega^*W_{2k}H^*(\Omega_B).\]
\end{enumerate}
\end{thm}
\begin{proof}
Let $\Psi\colon M \to M_B$ be the non-abelian Hodge correspondence, and $\widetilde{\Psi}\colon \widetilde{M} \to \widetilde{M}_B$ be the diffeomorphism lifting $\Psi$ in the sense of \cref{LiftNAHC}. By the commutativity of the diagram (\ref{Psicommsquarecoh}), and since the map $\Psi$ preserves the stratifications (\ref{WhitneyM}) and (\ref{WhitneyMB}), the map $\widetilde{\Psi}^*\colon H^*(\widetilde{M}_B) \to H^*(\tm)$ splits on the summands of the decomposition theorem. More precisely, $\widetilde{\Psi}^*$ is given by the product map 
\begin{align}
 (\Psi^*, \Psi|_{\Sigma}^{*-2}, \Psi|^{*-6}_{\Omega})\colon IH^* & (M_B) \oplus  H^{*-2}(\Sigma_B)(-1)\oplus H^{*-6}(\Omega_B)(-3) \label{splitnonabelian}\\
 & \to IH^*(M)\oplus H^{*-2}(\Sigma)(-1)\oplus H^{*-6}(\Omega)(-3).   \nonumber
\end{align}
The statement then follows by \cref{lem:splittingp} and \cref{lem:splittingw}.
\end{proof}
\begin{rmk}
The product map (\ref{splitnonabelian}) suggests that it is possible to define the isomorphism  in cohomology $\widetilde{\Psi}^*$, without constructing the diffeomorphism $\widetilde{\Psi}$. This is indeed the approach of \cite{deCataldoHauselMigliorini2013}. However, the virtue of \cref{LiftNAHC} is to establish that the isomorphism between the cohomology rings of $\widetilde{M}$ and $\widetilde{M}_B$ which realises the exchange of perverse and weight filtration has a geometric origin.
\end{rmk}

\begin{thm}[P=W for singular loci] \label{P=Wforsingularloci}
The P=W conjecture for $\Sigma$ and $\Omega$ holds.
\end{thm}
\begin{proof}
Since $\Omega$ is a collection of points, the perverse and the weight filtrations are all concentrated in degree zero, and so the P=W conjecture for $\Omega$ trivially holds.

We show now that the P=W conjecture for $\Sigma$ holds. To this end, note that 
the map $\chi|_{\Sigma}$ factors as follows:
\[\chi|_{\Sigma}\colon \Sigma \simeq (\Pic^0(C) \times H^0(K_C))/(\Z/2\Z) \twoheadrightarrow H^0(K_C)/(\Z/2\Z) \subset H^0(K_C^{\otimes 2}). \]
Equivalently, $\chi|_{\Sigma}$ can be identified with the quotient of the projection $\Pic^0(C) \times H^0(C, K_C) \to H^0(C, K_C)$ via the involution $(L, s) \mapsto (L^{-1}, -s)$. Therefore, the general fiber $\chi|_{\Sigma}^{-1}(s)$, with $s \in H^0(K_C)/(\Z/2\Z)$, is isomorphic to $\Pic^0(C)$. The zero fibre $\chi|_{\Sigma}^{-1}(0)$ instead is isomorphic to the singular Kummer surface associated to $\Pic^0(C)$, denoted by $S$ as in the proof of \cref{prop:fixedlocusMtilde}. Since $\Sigma$ is attracted by $S$ via the flow of the $\Gm$-action, $\Sigma$ retracts on $S$. In particular, we obtain that
\[
H^*(\Sigma)\simeq H^*(S) \simeq H^*(\chi|_{\Sigma}^{-1}(s))^{\Z/2\Z},
\]
and the restriction $H^*(\Sigma) \to H^*(\chi|_{\Sigma}^{-1}(s))$ is injective. Hence, by \cref{thm:kercharacterizationp}, we conclude that $H^d(\Sigma)$ has top perversity $d$.

On the Betti side, $\Sigma_B$ is isomorphic to $(\CC^*)^{4}/(\ZZ/2\ZZ)$ (cf (\ref{descriptionSigmaB})). This means that
\[H^*(\Sigma_B) \simeq H^*((\CC^*)^{4})^{\ZZ/2\ZZ} \subset H^*((\CC^*)^{4}).\]
In particular, $H^{d}(\Sigma_B)$ has only even cohomology of top weight $2d$, since $H^d((\CC^*)^{4})$ does. Since both the perverse and the weight filtrations are supported in top degree, the P=W conjecture for $\Sigma$ holds. 
\end{proof}

Having proved the second item in \cref{cor:splitting}, 
Section \ref{sec:PI=WI} 
will be devoted to the proof of the PI=WI conjecture.

\section{PI=WI conjecture for $M$}\label{sec:PI=WI}
\subsection{Action of the 2-torsion of the Jacobian}\label{sec:The action of the 2-torsion of the Jacobian}

The action of $\Gamma = \Pic^0(C)[2]$ induces the splitting 
\begin{equation}\label{eq:varinvariant}
IH^*(M) = IH^*(M)^{\Gamma} \oplus IH_{\mathrm{var}}^*(M),   
\end{equation}
where $IH^*(M)^{\Gamma}$ is fixed by the action of $\Gamma$, and $IH_{\mathrm{var}}^*(M)$ is the {variant} part, i.e.\ the unique $\Gamma$-invariant complement of $IH^*(M)^{\Gamma}$ in $IH^*(M)$. Note that the decomposition (\ref{eq:varinvariant}) induces a splitting of the perverse filtration. This follows from the exactness of the perverse truncation functors $^{\mathfrak{p}}\tau_{\leq i}$ applied to the character decomposition $\chi_*\underline{\QQ}_M \simeq \chi_*\underline{\QQ}^{\Gamma}_M \oplus \chi_*\underline{\QQ}_{M, \mathrm{var}}$.

In a similar way there exists an isomorphism of mixed Hodge structures
\[
IH^*(M_B) = IH^*(M_B)^{\Gamma} \oplus IH_{\mathrm{var}}^*(M_B).
\]
This implies the following theorem.

\begin{thm}\label{Thm:PWinv and var}
The PI=WI conjecture for $M$ is equivalent to the following two statements:
\begin{enumerate}
    \item \emph{(PI = WI conjecture for the invariant intersection cohomology)}
    \begin{equation}\label{eq:invariantPIWI}
    P_kIH^*(M)^{\Gamma}= \Psi^* W_{2k}IH^*(M_B)^{\Gamma} \qquad k \geq 0.  
    \end{equation}
    \item \emph{(PI = WI conjecture for the variant intersection cohomology)}
    \begin{equation}\label{eq:variantPIWI}
    P_kIH^*(M)_{\mathrm{var}}= \Psi^*  W_{2k}IH^*(M_B)_{\mathrm{var}} \qquad k \geq 0.    
    \end{equation}
\end{enumerate}
\end{thm}

We continue 
with the computation of the intersection Poincar\'{e} polynomial of $M$ and the intersection E-polynomial of $M_B$.
\color{black}

\begin{prop}\label{prop:Poincarepolynomialintersection}
The intersection Poincar\'{e} polynomials are
\begin{align*}
IP_t(M) & \coloneqq \sum_k \dim IH^k(M)\, t^k = 1+t^2+17t^4+17t^6 \\  
IP_t(M)^{\Gamma}& \coloneqq \sum_k \dim IH^k(M)^{\Gamma} \, t^k  = 1+t^2+2t^4+2t^6\\
IP_{t,{\mathrm{var}}}(M) & \coloneqq \sum_k \dim IH^k_{\mathrm{var}}(M) \, t^k  = 15t^4+15t^6.
\end{align*}
\end{prop}
\begin{proof}

By \eqref{eq:Poincarepolynomialres} and \eqref{eq:decompositioncohMtilde} we have
\begin{align*}
IP_t(M) & =P_t(\widetilde{M})-P_t(\Sigma)t^2-P_t(\Omega)t^6\\
& = (1+2t^2+23t^4+34t^6)-(1+6t^2+1)t^2 - 16t^6\\
& = 1+t^2+17t^4+17t^6;
\end{align*}
see also \cite[Theorem 6.1]{Felisetti2018}. 

Since the differentials of the local-to-global spectral sequence \eqref{eq:spseq2} are $\Gamma$-equivariant, we obtain 
\[
P_{t, \mathrm{var}}(\widetilde{M}) =P_{t, \mathrm{var}}(\widetilde{\NBun}) + P_{t, \mathrm{var}}(\widetilde{S}^+)t^2 + P_{t, \mathrm{var}}(\widetilde{\Theta})t^6 + P_{t, \mathrm{var}}(\bigcup^{16}_{j=1}T_j)t^6,
\]
in the notation of \cref{thm:PoincarePoly}. The group $\Gamma$ acts trivially on $H^*(\widetilde{N})$ and $H^*(\Sigma) \simeq H^*(S) \subset H^*(\bar{S}^+)$, and as the regular representation on the 16-dimensional vector spaces
\[\bigoplus^{16}_{j=1}\QQ [s^+_j] \subset H^*(\bar{S}^+), \quad  H^0(\widetilde{\Theta}), \quad \bigoplus^{16}_{j=1}\QQ [T_j], \quad H^0(\Omega).\]
Again by \eqref{eq:decompositioncohMtilde}, we get
\begin{align*}
IP_{t, \mathrm{var}}(M) = & P_{t, \mathrm{var}}(\widetilde{M}) -P_{t, \mathrm{var}}(\Sigma)t^2-P_{t, \mathrm{var}}(\Omega)t^6\\
 =  & (\dim(\bigoplus^{16}_{j=1}\QQ [s^+_j])-1)t^4 + (\dim H^0(\bigcup^{16}_{j=1} T_j)-1)t^6 \\
& + (\dim H^0(\widetilde{\Theta})-1)t^6 - (\dim H^0(\Omega))-1)t^6\\ 
= & 15t^4+15t^6.
\end{align*}
Finally, $IP_t(M)^{\Gamma} = IP_t(M) - IP_{t, \mathrm{var}}(M)= 1+t^2+2t^4+2t^6$.
\end{proof}

\begin{prop}\label{prop:Epolynomial}
The intersection E-polynomial of $M_B$ is
\begin{align*}
IE(M_B) & \coloneqq \sum_{p,q, d}(-1)^d \dim ( \Gr^W_{p+q} IH^d_{c}(M_B, \CC))^{p,q} u^pv^q\\
& = \sum_{k, d} \dim \Gr^W_{2k}IH^d(M_B) q^k = 1+17q^2+17q^4+q^6
\end{align*}
with $q=uv$. 
In particular, 
$\dim \Gr^W_{2k+1}IH^d(M_B)=0$ for all $k, d \in \mathbb{N}$.
\end{prop}
\begin{proof}

The analogue of \cref{lem:splittingw} for compactly supported cohomology yields
\begin{equation}\label{eq:ieb}
IE(M_B)= E(\widetilde{M}_B)- E(\Sigma_B) q - E(\Omega_B) q^3. 
\end{equation}
In order to compute $E(\widetilde{M}_B)$, consider the stratification of $\widetilde{M}_B$ 
$$\widetilde{M}_B= M_B^{\mathrm{sm}}\sqcup \widetilde{\Sigma}_B\setminus \widetilde{\Omega}_B\sqcup \widetilde{\Omega}_B,$$
where $\widetilde{\Sigma}_B\setminus \widetilde{\Omega}_B \coloneqq f^{-1}({\Sigma}_B\setminus {\Omega}_B)$ and $\widetilde{\Omega}_B \coloneqq f^{-1}(\Omega_B)$. 
It is proved in \cite[\S 8.2.3]{LogaresMunozNewstead2013} that the E-polynomial of $M_B$ is $E(M_B)=1+q^2+17q^4+q^6$.
This implies that 
\begin{align}
    E(M_B^{\mathrm{sm}})& = E(M_B) - E(\Sigma_B) \label{EpolyMB}\\
    & = (1+q^2+17q^4+q^6) - (1+6q^2+q^4) =-5q^2+16q^4+q^6, \nonumber 
\end{align}
where the second equality follows from the fact that the weight filtration on $H_{c}^*(\Sigma_B)$ is concentrated in top degree; see \cref{P=Wforsingularloci}.
Since  $\widetilde{\Sigma}_B\setminus \widetilde{\Omega}_B$ is a 
$\PP^1$-bundle over $\Sigma_B\setminus \Omega_B$, we obtain 
\begin{align}
    E(\widetilde{\Sigma}_B\setminus \widetilde{\Omega}_B) & =E(\PP^1) \cdot E(\Sigma_B\setminus \Omega_B) =(q+1) (1 + 6q^2+q^4 -16). \label{EpolySigma}
    &  
\end{align}
Observe that $\widetilde{\Omega}_B$ is the disjoint union of 16 smooth quadric 3-folds $\tomega_{B,j}$, so that
\begin{equation}\label{EpolOmegaB}
    E(\widetilde{\Omega}_B)=\sum_{j=1}^{16} E(\tomega_{B,j})= 16(1+q+q^2+q^3). 
\end{equation}
Adding up the E-polynomials (\ref{EpolyMB}), (\ref{EpolySigma}) and (\ref{EpolOmegaB}), we get \begin{equation}\label{EpoltildeM}
    E(\widetilde{M}_B)=1+q+17q^2+22q^3+17q^4+q^5+q^6.
\end{equation}
Finally, from \eqref{EpoltildeM} and \eqref{eq:ieb} we obtain
\begin{equation}\label{IEM_B}
IE(M_B)=1+17q^2+17q^4+q^6.
\end{equation}

By the vanishing of the odd intersection cohomology (cf \cref{prop:Poincarepolynomialintersection}), every non-trivial $(\Gr^W_{p+q} IH^d_{c}(M_B, \CC))^{p,q}$ will contribute with non-negative coefficient to $IE(M_B)$. Therefore, there is no cancellation and by \eqref{IEM_B} any non-trivial $(\Gr^W_{p+q} IH^d_{c}(M_B, \CC))^{p,q}$ has type $(p,p)$, i.e.\ the mixed Hodge structure on $IH^d_{c}(M_B, \CC)$ is of Hodge-Tate type.
In symbols, we write
\begin{align*}
    IE(M_B) & = \sum_{p,q, d} \dim ( \Gr^W_{p+q} IH^d_{c}(M_B, \CC))^{p,q} u^pv^q\\
    & = \sum_{k, d} \dim (\Gr^W_{2k} IH^d_{c}(M_B, \CC))^{k,k} q^k = \sum_{k, d} \dim \Gr^W_{2k} IH^d(M_B, \CC) q^k,
\end{align*}
where the last equality follows from Poincar\'{e} duality and the fact that the polynomial in \eqref{IEM_B} is palindromic.
\end{proof}

\begin{prop}\label{prop:Epolynomialinvvar}
The intersection E-polynomials are
\begin{align*}
IE(M_B)^{\Gamma} & \coloneqq \sum_{k, d} \dim \Gr^W_{2k}IH^d(M_B)^{\Gamma} q^k = 1+2q^2+2q^4+q^6\\
IE_{\mathrm{var}}(M_B) & \coloneqq \sum_{k, d} \dim \Gr^W_{2k}IH^d_{\mathrm{var}}(M_B)q^k = 15q^2+15q^4.
\end{align*}
\end{prop}

\begin{proof}
The solution of the linear system
\[
\begin{cases}
    \dim\Gr^W_{k}IH^d(M_B) \leq \dim\Gr^W_{k} H^{d}(\widetilde{M}_B)=0  \text{ for }k<d&  \text{\cite[Prop 4.20]{PetersSteenbrink2008}}\\
    \dim \Gr^W_{2k+1}IH^d(M_B)=0  &  \text{\cref{prop:Epolynomial}}\\
    \sum_{k, d} \dim \Gr^W_{2k}IH^d(M_B) \, q^{k} = 1+17q^2+17q^4+q^{6}  &  \text{\cref{prop:Epolynomial}}\\
    \sum_k \dim IH^k(M_B)\, t^k = 1+t^2+17t^4+17t^6  &  \text{\cref{prop:Poincarepolynomialintersection}}
\end{cases}
\]
is given by
\begin{equation}\label{gradingw2d}
    \dim \Gr^W_{4d} IH^{2d}(M_B) = 1 \text{ for }d=0,1,2,3,
\end{equation} 
\begin{equation}\label{gradingw4}
 \dim \Gr^W_{4} IH^4(M_B) =  \dim \Gr^W_{8} IH^6(M_B) = 16.\end{equation}
The terms in this list are all the non-zero graded pieces of the mixed Hodge structure on $IH^*(M_B)$.

Note that the top graded pieces $\Gr^W_{2d} IH^d(M_B)$ are generated by $\alpha^d$, where $\alpha$ is a ($\Gamma$-invariant) generator of $IH^2(M_B)$. The class $\alpha$ corresponds via the non-abelian Hodge correspondence to the first Chern class of a $\chi$-ample (or $\chi$-antiample) divisor on $M$. In particular, $\alpha^2$ and $\alpha^3$ are non-zero and $\Gamma$-invariant.
This implies that 
\begin{align*}
    IH_{\mathrm{var}}^{4}(M_B) \subset W_4 IH^{4}(M_B) \simeq \Gr^{W}_4 IH^{4}(M_B)\\
    IH_{\mathrm{var}}^{6}(M_B) \subset W_8 IH^{6}(M_B)) \simeq \Gr^{W}_8 IH^{6}(M_B).
\end{align*} Together with \cref{prop:Poincarepolynomialintersection} and \cref{prop:Epolynomial}, we conclude that
\begin{align*}
    IE_{\mathrm{var}}(M_B) & = \dim \Gr^{W}_4 IH^{4}_{\mathrm{var}}(M_B) q^2 + \dim \Gr^{W}_8 IH^{6}_{\mathrm{var}}(M_B) q^4\\
    & = \dim IH^{4}_{\mathrm{var}}(M_B) q^2 + \dim IH^{6}_{\mathrm{var}}(M_B) q^4 = 15 q^2+15q^4\\
    IE(M_B)^{\Gamma} & = IE(M_B)-IE_{\mathrm{var}}(M_B)\\
    & = (1+17q^2+17q^4+q^6) - (15 q^2+15q^4) = 1+2q^2+2q^4+q^6.
\end{align*}
\end{proof}

As a result, an analogue of \cite[Corollary 4.5.1]{HauselRodriguez-Villegas2008} holds for $M_B$.

\begin{cor}\label{vanishinginterform}
The intersection form on $H^6_c(M_B) = IH^6_c(M_B)$ is trivial. Equivalently, the forgetful map $H^6_c(M_B) \to H^6(M_B)$ is zero.
\end{cor}

\begin{proof}
By \eqref{gradingw2d}, \eqref{gradingw4} and Poincar\'{e} duality,  the weight filtrations on $IH^6(M_B)$ and $IH^6_c(M_B)$ are concentrated in degree $[8,12]$ and $[0,4]$. Since the forgetful map is a morphism of mixed Hodge structures, it has to vanish. 
\end{proof}

\begin{rmk}[Failure of curious hard Lefschetz] \label{failurecurioushardLef}
By \eqref{eq:Poincarepolynomial} and the proof of \cref{prop:Epolynomialinvvar}, we have
\begin{equation}\label{EpolMB}
 \sum_{k, d} \dim \Gr^W_{2k}H^d(M_B) q^k = 1+q^2+17q^4+q^6.   
\end{equation}
The fact that the polynomial \eqref{EpolMB} is not palindromic implies that curious hard Lefschetz \eqref{cHL} fails for $H^*(M_B)$. Analogously, one can show that relative hard Lefschetz fails for $H^*(M)$. 
\end{rmk}

\subsection{The variant part of $IH^{*}(M)$}\label{sec:The variant part}
The goal of this section is to show that the PI = WI conjecture for the variant part of $IH^{*}(M)$ holds. As we will explain in the proof of \cref{PI=WIvariant}, it is enough to prove it in degree 4 and 6.

\begin{prop}\label{prop:perv4var}
$IH_{\mathrm{var}}^4(M) \subset P_2 IH^4(M)$.
\end{prop}
\begin{proof}
The argument of \cite[\S 4.4]{deCataldoHauselMigliorini2012} and \cite[Proposition 1.4]{deCataldoMaulikShen2020} works with few changes.

The endoscopic locus $\mathcal{A}_e \subset H^0(C, K_C^{\otimes 2})$ is the subset of sections $s' \in H^0(C, K_C^{\otimes 2})$ such that the Prym variety associated to the corresponding spectral curve $C_{s'}$ is not connected (cf \cite[\S 4.4]{deCataldoHauselMigliorini2012}). It is the union of 15 lines, obtained as images of the squaring map 
\[i_L\colon H^0(C, K_C \otimes L) \to H^0(C, K_C^{\otimes 2}) \qquad i_L(a)=a \otimes a,\]
where $L \in \Gamma \setminus \{0\}$. In particular, a general affine line $\Lambda^1$ in $H^0(C, K_C^{\otimes 2})$ does not intersect $\mathcal{A}_e$. It is important to remark that $\Gamma$ acts trivially on $H^*(\chi^{-1}(s))$ for any $s \in \Lambda^1$: the proof in \cite[\S 4.4]{deCataldoHauselMigliorini2012} is independent of the choice of the degree of the Higgs bundles, and so it holds also in the untwisted case. 
This implies
\[H^*(\chi^{-1}(\Lambda^1))^{\Gamma} = H^*(\chi^{-1}(\Lambda^1))= IH^*(\chi^{-1}(\Lambda^1)),\]
where the last equality follows from the fact that $\chi^{-1}(\Lambda^1)$ has quotient singularities.
We conclude by \cref{thm:kercharacterizationp} that
\begin{align*}
IH_{\mathrm{var}}^4(M) & \subset \ker \{IH^4(M) \to IH^4(\chi^{-1}(\Lambda^1))= IH^4(\chi^{-1}(\Lambda^1))^{\Gamma}\}\\
& = P_2 IH^4(M),
\end{align*}
because the restriction map is $\Gamma$-equivariant.
\end{proof}

\begin{thm}\label{PI=WIvariant}
The PI = WI conjecture for the variant intersection cohomology of $M$ (\ref{eq:variantPIWI}) holds.
\end{thm}
\begin{proof}
The variant Poincar\'{e} polynomial in \cref{prop:Poincarepolynomialintersection} shows that $IH^*_{\mathrm{var}}(M)$ is concentrated in degree 4 and 6. 

By relative hard Lefschetz, we can write
\[
\Gr^P_0 IH^4(M)\simeq \Gr^P_{6} IH^{10}(M) \qquad 
\Gr^P_1 IH^4(M)\simeq \Gr^P_{5} IH^{8}(M),
\]
which both vanish by \cref{prop:Poincarepolynomialintersection}. 
Together with Proposition \ref{prop:perv4var} and the proof of Proposition \ref{prop:Epolynomialinvvar}, this implies
\[ P_2 IH_{\mathrm{var}}^4(M) = IH_{\mathrm{var}}^4(M) = \Psi^* IH_{\mathrm{var}}^4(M_B) = \Psi^*W_4 IH_{\mathrm{var}}^4(M_B).\]
This proves the PI=WI conjecture for the variant part in degree 4.

Again by relative hard Lefschetz, there exists a $\chi$-ample $\alpha \in H^{2}(M)$ such that the cup product $\cup \alpha$
induces the isomorphism 
\begin{align*}
    \cup \alpha\colon IH_{\mathrm{var}}^4(M) \simeq \Gr^P_2 IH_{\mathrm{var}}^4(M)   & \to \Gr^P_4 IH^6_{\mathrm{var}}(M).
\end{align*}
By \cref{prop:Poincarepolynomialintersection} we obtain that
\[15 = \dim IH_{\mathrm{var}}^4(M) = \dim \Gr^P_4 IH^6_{\mathrm{var}}(M) \leq \dim IH_{\mathrm{var}}^6(M)=15.\]
This implies that the cup product 
\[\cup \alpha\colon IH^4_{\mathrm{var}}(M)  \to IH^6_{\mathrm{var}}(M)\] 
is an isomorphism, which preserves the perverse and weight filtrations; see \cite[Lemma 1.4.4]{deCataldoHauselMigliorini2012}. Therefore, the PI=WI conjecture for the variant part holds in degree 6 as well.
\end{proof}

\subsection{A tautological class}\label{sec:a tautological class}
We show now that $IH^4(M)^{\Gamma}$ is generated by the square of the relatively ample class $\alpha$, and a class of perversity $2$ and weight $4$. As usual, we adopt the notation of the previous sections, and in particular of \cref{sec:moduliMiota}.

Consider the forgetful map $q\colon M_{\iota} \to M$. The action of $\Gamma$ on $M$ lifts to $M_{\iota}$, and together with the deck transformation of $q$, we obtain a group of symmetries of order $32$, denoted $\Gamma_{\iota}$.

\begin{prop}
$IH^4(M)^{\Gamma} = H^4(M^{\mathrm{sm}}_{\iota})^{\Gamma_{\iota}}$.
\end{prop}
\begin{proof}
Since $M_{\iota}$ has isolated singularities by \cref{prop:singMiota}, we have that
$
IH^4(M_{\iota}) = H^4(M^{\mathrm{sm}}_{\iota});
$
see \cite[\S 1.7]{GM81} or \cite[Lemma 1]{Durfee95}. The proof of \cite[Proposition 3]{GottscheSoergel93} implies that \[
IH^4(M)^{\Gamma} = IH^4(M_{\iota})^{\Gamma_{\iota}}= H^4(M^{\mathrm{sm}}_{\iota})^{\Gamma_{\iota}}.\]\end{proof}

Fix $c \in C$ a base point. Recall that $\mathbb{E}$ is a universal bundle on $M^{\mathrm{sm}}_{\iota} \times C$; see \cref{sec:Universalbundle}. 

\begin{defn}
$\Runi$ is the total space of the projective bundle $\mathbb{P}(\mathbb{E}|_{M^{\mathrm{sm}}_{\iota} \times \{c\}})$. Its associated principle $\mathrm{PGL}_2$-bundle parametrises equivariant Higgs bundles $(E, h, \phi)$ together with a frame for the fibre $E_c$, up to rescaling.
\end{defn}


The second Chern class of a $\PP^1$-bundle is the pull-back of a generator of $H^4(\operatorname{BPGL}_2) \simeq \QQ$ via the classifying map. In particular, if the $\PP^1$-bundle is a projectivization of the rank-two vector bundle $E$, then
\[
c_2(\PP(E))=c^2_1(E)-4c_2(E).
\]
\begin{prop}\label{prop:IH4}
The second Chern class $c_2(\Runi)$ of the projective bundle $\Runi$ and the square of the $\chi$-ample class $\alpha$ generate $IH^4(M)^{\Gamma}$
\[IH^4(M)^{\Gamma}= \QQ \,\alpha^2 \oplus \, \QQ \, c_2(\Runi).\]
\end{prop}

\begin{proof}
The proposition is a consequence of the following facts:
\begin{enumerate}
    \item $c_2(\Runi) \in H^4(M^{\mathrm{sm}}_{\iota})^{\Gamma_{\iota}} = IH^4(M)^{\Gamma}$, since the $\Gamma_{\iota}$-action lifts to $\Runi$.
    \item $\alpha^2 \in H^4(M)^{\Gamma} \subset IH^4(M)^{\Gamma}$.
    \item $c_2(\Runi) \neq 0$ by \cref{lem:c2neqzero}.
    \item $\alpha^2$ and $c_2(\Runi)$ are linearly independent, because $\alpha^2$ has top perversity by \cref{thm:kercharacterizationp}, while $c_2(\Runi) \in  P_2 IH^4(M)$; see \cref{lem:c2perversity}.
    \item $\dim IH^4(M)^{\Gamma} =2$ by \cref{prop:Poincarepolynomialintersection}.
\end{enumerate}
\end{proof}
We now prove the lemmas used in the proof above.

\begin{lem}
\label{lem:c2neqzero}$c_2(\Runi) \neq 0$.
\end{lem}
\begin{proof}
Let $\mathfrak{r}_j\colon\Rcomp^{\circ}_j \times C \to (M^{\circ}_{\iota} \cap \NBun^{+}_j) \times C$ be the algebraic bijection defined in \ref{sec:nilpotent}.(\ref{eq:Rj}). \cref{lem:equivalenceP1bundles} and \cref{prop:universalbundleRj} give
\begin{align*}
\mathfrak{r}^*_j c_2(\Runi) & = c_2(\PP(\mathbb{E}|_{\Rcomp^{\circ}_j \times \{c\}})) \\
& = (c_1(p_{\Rcomp_j}^* \mathcal{O}_{\Rcomp_j}(1)\otimes p_C^*\theta^{-1}_j) - c_1(p_C^*\theta_j))^2|_{\Rcomp^{\circ}_j \times \{c\}}=  c_1(\mathcal{O}_{\Rcomp^{\circ}_j}(1))^2.
\end{align*}
In particular, $0 \neq c_1(\mathcal{O}_{\Rcomp^{\circ}_j}(1))^2 \in H^4(\Rcomp^{\circ}_j) \simeq H^4(\PP^3)$.
\end{proof}

\begin{lem}\label{lem:c2Rperversity3}
$c_2(\Runi) \in  P_3 IH^4(M)$.
\end{lem}
\begin{proof} 
Fix $s$ a generic point in $H^0(C, K_C^{\otimes 2})$, and let $p_s \colon C_s \to C$ the corresponding spectral curve, i.e.\ the double cover of $C$ ramified along the zeroes of $s$; see for instance \cite[\S 3]{BNR89}. We denote the product map $p_s \times \id\colon C_s \times \Pic^0(C_s)\to C \times \Pic^0(C_s)$ simply by $p_s$. A universal bundle on $\chi^{-1}(s) \times C \simeq C \times \Pic^0(C_s)$ do exist, and it is isomorphic to $p_{s, *} \mathcal{P}$, where $\mathcal{P}$ is the Poincar\'{e} line bundle over $C_s \times \Pic^0(C_s)$. 

The abelian variety $(\chi \circ q)^{-1}(s)$ parametrises line bundles of $C_s$ decorated with a lift of the hyperelliptic involution $\iota\colon C \to C$. This implies that the restriction of $\mathbb{E}$ to $(\chi \circ q)^{-1}(s) \times C$ is isomorphic to $q^* (p_{s}, \id)_* \mathcal{P}$, up to tensorization by a line bundle in $\Pic((\chi \circ q)^{-1}(s))$.

As a result, we have that
\begin{align*}
   c_2(\Runi)|_{\chi^{-1}(s)} & = c^2_1(\mathbb{E}|_{(\chi \circ q)^{-1}(s) \times \{c\}})-4c_2(\mathbb{E}|_{(\chi \circ q)^{-1}(s) \times \{c\}})\\
   & = q^*\Big(c^2_1\big((p_{s,*} \mathcal{P})|_{\chi ^{-1}(s) \times \{c\}}\big) - 4c_2\big((p_{s,*} \mathcal{P})|_{\chi ^{-1}(s) \times \{c\}}\big)\Big)=0,
\end{align*}
where the last equality follows from \cite[\S 4]{Thaddeus89} or \cite[Eq. (5.1.10) and (5.1.11)]{deCataldoHauselMigliorini2012}.
This implies that $c_2(\Runi)$ does not have top perversity by \cref{thm:kercharacterizationp}. 
\end{proof}

\begin{lem}\label{lem:c2perversity}
$c_2(\Runi) \in  P_2 IH^4(M)$.
\end{lem}
\begin{proof}
By Lemma \ref{lem:c2Rperversity3}, it is enough to show that the projection $[c_2(\Runi)]$ in the graded piece $\Gr^P_3IH^4(M)^{\Gamma}$ vanishes. Suppose on the contrary that $[c_2(\Runi)]\neq 0$. 
Then \cref{prop:Poincarepolynomialintersection} would imply \[\dim \Gr^P_2 IH^4(M)^{\Gamma}\leq \dim IH^4(M)^{\Gamma} - \dim (\QQ \,\alpha^2 \oplus \, \QQ \, c_2(\Runi)) =2-2=0.\] 
By relative hard Lefschetz, $\Gr^P_4 IH^6(M)^{\Gamma}$ would be trivial. Analogously, $$\Gr^P_5 IH^6(M)^{\Gamma}\simeq \Gr^P_1 IH^2(M)^{\Gamma}=0.$$
Again by \cref{prop:Poincarepolynomialintersection} we would conclude that
\begin{align*}
    \dim \Gr^P_3IH^6(M)^{\Gamma} & = \dim IH^6(M)^{\Gamma} - \sum^6_{k=4}\dim \Gr^P_kIH^6(M)^{\Gamma} \\
    & = \dim IH^6(M)^{\Gamma} - \dim \QQ\alpha^3 = 2-1= 1.
\end{align*}
However, this is a contradiction by \cref{vanishinginterform}.
\end{proof}

\begin{lem}\label{lem:c2perversityII}
$P_3 IH^6(M)^{\Gamma}=0$.
\end{lem}

\begin{proof}
\cref{lem:splittingp} gives the splitting 
\[
P_3 H^6(\widetilde{M})^{\Gamma} = P_3 IH^6(\widetilde{M})^{\Gamma} \oplus P_2 H^4(\Sigma)^{\Gamma} \oplus  H^0(\Omega)^{\Gamma}.
\]
The P=W conjecture for $\Sigma$ gives $P_2 H^4(\Sigma)=0$.
Moreover, we have that
$H^0(\Omega)^{\Gamma}  \simeq \QQ [\Omega]$; see the proof of  \cref{prop:Poincarepolynomialintersection}. Therefore, we get
\[
P_3 H^6(\widetilde{M})^{\Gamma} = P_3 IH^6(\widetilde{M})^{\Gamma} \oplus  \QQ [\Omega].
\]
Up to a different numbering convention, \cite[Theorem 2.1.10]{deCataldoMigliorini05} says that the dimension of $P_3 H^6(\widetilde{M})^{\Gamma}$ is not greater than the rank of the intersection form on $H^6(\widetilde{M})^{\Gamma}$. Therefore, \cref{vanishinginterform} 
implies
\[
\dim P_3 H^6(\widetilde{M})^{\Gamma}\leq 1.
\]
We conclude that $\dim P_3 IH^6(M)^{\Gamma}=0$.
\end{proof}

We conclude the section by showing that the class $c_2(\Runi)$ has weight 4.
\begin{lem}\label{lem:weightc2R}
$c_2(\Runi) \in  W_4 IH^4({M}_B)$.
\end{lem}
\begin{proof}
The principal $\mathrm{PGL}_2$-bundle $\mathcal{S} \to M^{\mathrm{sm}}_{B, \iota} \coloneqq \Psi_\iota(M^{\mathrm{sm}}_{\iota})$ is the restriction of the quotient $\Hom(\pi^{\mathrm{orb}}_1(C/\iota), \Sl_2) \to {M}_B(2, \Sl_2,\iota)$. It parametrises $\iota$-equivariant local systems $E'$ on $C$ together with a frame for the fibre $E'_c$ over $c \in C$, i.e.\ the base point of $\pi^{\mathrm{orb}}_1(C/\iota) = \pi^{\mathrm{orb}}_1(C/\iota, c)$, up to rescaling. 

By construction, the non-abelian Hodge correspondence $\Psi_{\iota}\colon M^{\mathrm{sm}}_{\iota} \to M^{\mathrm{sm}}_{B,  \iota}$ (\cref{thm:nonabelianHodgecorrespondenceforMi}) extends to a diffeomorphism between the principle $\mathrm{PGL}_2$-bundle associate to $\Runi$ and $\mathcal{S}$. This implies that
$c_2(\Runi) = (\Psi_{\iota}^{-1})^*c_2(\mathcal{S})$, and $c_2(\mathcal{S})$ has weight $4$ by \cite[Theorem 9.1.1, Proposition 9.1.2]{Deligne74}.
\end{proof}

\subsection{The invariant part of $IH^{*}(M)$}\label{sec:invIH}

\begin{thm}\label{thm:P=Winvar}
The invariant $PI=WI$ conjecture (\ref{eq:invariantPIWI}) holds for $M$.
\end{thm}
\begin{proof}
The statement is obvious in degree 0 and 2, because $IH^0(M)^{\Gamma}$ and $IH^2(M)^{\Gamma}$ have dimension one by \cref{prop:Poincarepolynomialintersection} and \cref{prop:Epolynomialinvvar}.

Now we have 
\begin{align*}
    P_2 IH^4(M)^{\Gamma} & \simeq W_4 IH^4(M_B)^{\Gamma}\\
    \Gr^P_3 IH^4(M)^{\Gamma} & \simeq \Gr^W_6 IH^4(M_B)^{\Gamma} = 0 \\
    P_4 IH^4(M)^{\Gamma} = IH^4(M)^{\Gamma} & \simeq IH^4(M_B)^{\Gamma} = W_8 IH^4(M_B)^{\Gamma}
\end{align*}
due to \cref{prop:IH4}, \cref{lem:weightc2R} and \cref{prop:Epolynomialinvvar}. This proves the invariant PI=WI conjecture in degree 4.

By relative hard Lefschetz, the cup product with the $\chi$-ample $\alpha \in H^{2}(M)$
induces the isomorphisms
\begin{align*}
    \cup \alpha\colon \Gr^P_2 IH^4(M)^{\Gamma}   & \to \Gr^P_4 IH^6(M)^{\Gamma}\\
    \cup \alpha\colon \Gr^P_4 IH^4(M) \simeq \QQ [\alpha^2]   & \to \Gr^P_6 IH^6(M) \simeq \QQ [\alpha^3].
\end{align*}
Note that $\Gr^P_2 IH^4(M)^{\Gamma}$ and $\Gr^P_4 IH^4(M)$ are the only non-trivial pieces of the perverse filtration on $IH^4(M)^{\Gamma}$, and by \cref{prop:Poincarepolynomialintersection} we have that $\dim IH_{\mathrm{var}}^4(M)^{\Gamma} = \dim IH_{\mathrm{var}}^6(M)^{\Gamma}$.
This implies that the cup product
\[
\cup \alpha\colon IH^4(M)^{\Gamma} \to IH^6(M)^{\Gamma}
\]
is an isomorphism which preserves both perverse and weight filtration (cf \cite[Lemma 1.4.4]{deCataldoHauselMigliorini2012}). Therefore, the invariant PI=WI conjecture holds in degree 6, as well.
\end{proof}
\newpage
\appendix
\section{Degenerations of hyperk\"ahler varieties}\label{appendix:deg}

In this appendix we describe degenerations of compact hyperk\"ahler manifolds to  (non-compact) symplectic resolutions of Dolbeault moduli spaces. Instances of these constructions can be found in \cite{DEL97}, \cite{deCataldoMaulikShen2019}, \cite{deCataldoMaulikShen2020}. Here a degeneration is a flat (not necessarily proper) morphism of normal algebraic varieties, typically over a curve.

The compact hyperk\"ahler manifolds appearing in these degenerations are Mukai moduli spaces of sheaves on a K3 surface or an abelian surface $S$. Given an effective Mukai vector\footnote{i.e.\ there exists a coherent sheaf $\mathcal{F}$ on $S$ such that $v = (rk(\mathcal{F}), c_1(\mathcal{F}), \chi(\mathcal{F})-\epsilon(S)rk(\mathcal{F}))$, with $\epsilon(S)\coloneqq 1$ if $S$ is K3, and  $0$ if $S$ is abelian.} $v \in H^{*}_{\mathrm{alg}}(S, \ZZ)$, we denote by $\Mukai$ the moduli space of Gieseker semistable sheaves on $S$ with Mukai vector $v$ 
for a sufficiently general polarization $H$ (which we will typically omit in the notation); see \cite[\S 1]{Simpson1994I}.  Further, if $S$ is an abelian variety with dual $\hat{S}$, and $\dim \Mukai\geq 6$, 
then the Albanese morphism 
$\alb_S\colon \Mukai \to \hat{S} \times S$
is isotrivial, and we set
$\Kumm \coloneqq \alb_S^{-1}(0_S, \mathcal{O}_S)$. By \cite{PeregoRapagnetta18}, the moduli space $\Mukai$ of sheaves on the K3 surface $S$ and the moduli space $\Kumm$ of sheaves on the abelian surface $S$ are irreducible holomorphic symplectic varieties, in brief IHSv.

\subsection{Deformation to the normal cone: $\Gl_n$ case}
Let $j \colon X \hookrightarrow S$ be the embedding of a smooth projective curve\footnote{In \cite{DEL97} $X$ is a very ample divisor, but the assumption can be dropped.} of genus $g$ into a K3 surface $S$. The degeneration to the normal cone of $j \colon X \hookrightarrow S$ is the family
$$ \mathcal{S}=\left(\mathrm{Bl}_{X\times 0}S\times \Aff^1\right)\setminus \left( S\times 0\right)\rightarrow \Aff^1.$$
The central fibre $\mathcal{S}_0$ is isomorphic to $T^*X$, while the restriction to $\Aff^1 \setminus \{0\}$ is a trivial fibration
$S\times (\Aff^1 \setminus \{0\})\rightarrow \Aff^1 \setminus \{0\}$.

For all $t\in \Aff^1$, let $\beta_t=n[X]\in H_2(\mathcal{S}_t,\ZZ)$ with $n>0$. Take a relative compactification $\mathcal{S} \subset \overline{\mathcal{S}}$ over $\Aff^1$. Then $$\mathcal{M} \to \Aff^1$$ is the coarse relative moduli space of one-dimensional Gieseker semistable sheaves $\mathcal{F}$ whose support is proper and contained in  $\mathcal{S}_t \subseteq \overline{\mathcal{S}}_t$ with $\chi(\mathcal{F})=n(1-g)$ and $[\mathrm{Supp}\mathcal{F}]=\beta_t$; see \cite[Theorem 1.21]{Simpson1994I}. 
The central fibre recovers the Dolbeault moduli space  
$$\mathcal{M}_0 \simeq \MDolGl.$$
Indeed, the moduli space of Higgs bundles on $X$ of rank $n$ and degree $0$ can be realized as the moduli space of one-dimensional Gieseker-semistable sheaves $\mathcal{F}$ on $T^*X$ with $\chi(\mathcal{F})=n(1-g)$ and $[\mathrm{Supp}\mathcal{F}]=\beta_0$, via the BNR-correspondence \cite{BNR89}. The general fibre is isomorphic to 
\[\mathcal{M}_t \simeq M(S,v)\]
with Mukai vector $v=(0,nX, n(g-1))$.

\begin{exa}[genus one:  $K3^{[n]}$] If $g=1$, then the degeneration $\mathcal{M}\rightarrow \Aff^1$ is the relative $n$-fold symmetric product of $\mathcal{S}$. The relative Hilbert-Chow morphism $\widetilde{\mathcal{M}}\rightarrow \mathcal{M}$ is a desingularization of $\mathcal{M}$. The composition $\widetilde{\mathcal{M}}\rightarrow \mathcal{M} \to \Aff^1$ is a family whose general fibre is the compact hyperk\"{a}hler manifold $S^{[n]}$ and whose central fibre is $(T^*X)^{[n]}$, i.e.\ the symplectic resolution of $\MDolGl \simeq (T^*X)^{(n)}$.
\end{exa}

\begin{exa}[genus two and rank two: O'Grady 10]\label{ex:OGrady10} If $(g,n)=(2,2)$, then the blow-up $\widetilde{\mathcal{M}}_t$ of the singular locus of $\mathcal{M}_t \simeq M(S,v)$ is a smooth compact hyperk\"{a}hler manifold deformation equivalent to OG10; see for instance \cite{PeregoRapagnetta13}. Analogously, the blow-up $\widetilde{\mathcal{M}}_0$ of the singular locus of $\mathcal{M}_0 \simeq M_{\mathrm{Dol}}(X, \mathrm{GL}_2)$ gives the symplectic resolution of $\mathcal{M}_0$. Note that the proof of \cite[Proposition 2.16]{PeregoRapagnetta13} shows that the degeneration $\mathcal{M}\rightarrow T$ is locally analytically trivial. Therefore, the blow-up $\widetilde{\mathcal{M}}$ of the singular locus of $\mathcal{M}$ is a smooth family over $\Aff^1$ whose general member is deformation equivalent to OG10 and whose central fibre is the symplectic resolution of $M_{\mathrm{Dol}}(X, \mathrm{GL}_2)$.
\end{exa}

\begin{rmk}\label{rmk:DEL}
Taking schematic supports via Fitting ideals defines a Lagrangian morphism $M(S, v) \to |nX|$, called Mukai system. It is classically known that the Mukai system degenerates to the Hitchin fibration, see \cite{DEL97}.
\end{rmk}

\begin{rmk}
If the Mukai vector $(0, X, g-1)$ is primitive (e.g.\ if $\Pic(S) = \ZZ X$), then the second author observed in \cite[Remark 2.5]{Mauri2021} that the degeneration $\mathcal{M}\rightarrow \Aff^1$ is locally analytically trivial. Therefore, the functorial resolution $\mathcal{R}(\mathcal{M}) \to \mathcal{M}$ of $\mathcal{M}$ gives a simultaneous resolution of $\mathcal{M}_{t}$ for any $t \in \Aff^1$; see for instance \cite[Lemma 4.2]{Graf}.
\end{rmk}

\subsection{Deformation to the normal cone: $\Sl_n$ case}
Suppose now that $X$ is a smooth projective curve embedded in an abelian surface $S$. To avoid confusion, we relabel $S$ by $A$.\footnote{In this section we denote by $A$ an abelian surface, and not a curve of genus one as in the rest of the paper.} As in the previous section, there exists a degeneration $\mathcal{M} \to \Aff^1$ from the moduli space $M(A,v)$ to $\MDolGl$. In this case, however, $M(A,v)$ is no longer an IHSv because of the Albanese morphism 
$\alb_S\colon M(A,v) \to \hat{A} \times A $. 

In genus one and two it is possible to slice $\mathcal{M}$ to obtain a degeneration of the IHSv $\KummA$ to $\MDolSl$.
\begin{exa}[genus one: $K^{[n]}(A)$]\label{genusonekummer}
If $g=1$, then the family
$$ \mathcal{A}=\left(\mathrm{Bl}_{C\times 0}A\times \Aff^1\right)\setminus \left( A\times 0\right)\rightarrow \Aff^1$$
is a group scheme, and the degeneration $\mathcal{M} \to \Aff^1$ is the relative n-fold symmetric product $\mathcal{A}^{(n)} \to \Aff^1$ whose general fibre is $A^{(n)}$, and whose central fibre is $(T^*C)^{(n)}$. 

Consider now the relative addition map $a_n\colon  \mathcal{A}^{(n)} \to \mathcal{A}$, given by $a_n(x_1, \ldots, x_n)= \sum_{i=1}^{n}x_i$. The inverse image of the identity section of $\mathcal{A} \to \Aff^1$ under the addition map is a degeneration
\[\mathcal{K} \to \Aff^1\]
whose general fibre is the singular generalised Kummer variety $\KummA \simeq K^{(n)}(A)$ and whose central fibre is $\MDolSl \simeq K^{(n)}(T^*C)$. The inverse image of the identity section of $\mathcal{A} \to \Aff^1$ under the composition $\mathcal{A}^{[n]} \to \mathcal{A}^{(n)} \to \mathcal{A}$ is a degeneration
\[\widetilde{\mathcal{K}} \to  \mathcal{K} \to \Aff^1\]
whose general fibre is the generalised Kummer manifold $K^{[n]}(A)$ and whose central fibre is the symplectic resolution of $\MDolSl$. 
\end{exa}
\begin{exa}[genus two]\label{ex:genus2}
If $g=2$, the Albanese map \cite{Yoshioka2001}
\[\alb_S\colon M(A,v) \to \hat{A} \times A\]
degenerates to  the map
\[
    \alb\colon  \MDolGl \to  \Pic^0(X) \times H^0(X, K_X) \simeq \hat{A} \times \Aff^{g},
\]
defined in \eqref{eq:alb}; see \cite[\S 4]{deCataldoMaulikShen2020}. 
Taking fibres over the identity, one obtains a family $\mathcal{K}\rightarrow \Aff^1$ such that the central fibre is $\MDolSl$ and the general fibre is the IHSv  $\KummA$.
\end{exa}
\begin{exa}[genus two and rank two: O'Grady 6] 
The symplectic resolution $f_A \colon \widetilde{K}(A,v) \to K(A,v)$, with $v=(0,2X,2)$ and $g=2$, is a compact hyperk\"{a}her manifold of OG6 type. Let $\widetilde{\mathcal{K}}$ be the blow-up of the singular locus of the variety $\mathcal{K}$ obtained in \cref{ex:genus2}, with $(g,n)=(2,2)$. Then $\mathcal{K} \to \Aff^1$ is a degeneration of $K(A,v)$ to the Dolbeault moduli space $M$ in \S \ref{sec:preliminary}. Further, as in \cref{ex:OGrady10}, $\widetilde{\mathcal{K}}$ is a smooth family over $\Aff^1$ whose general member is the compact hyperk\"{a}her manifold $\widetilde{K}(A,v)$ of OG6 type and whose central fibre is the symplectic resolution $\widetilde{M}$ of $M$.

We observe that the cohomology of $\widetilde{K}(A,v)$ governs the cohomology of $\widetilde{M}$ in the following sense.

\begin{prop}\label{prop:surjspeci}
The specialisation morphism \emph{\cite[(86)]{deCataldoMaulikShen2019}}
\[\spmap \colon H^*(\widetilde{K}(A,v)) \to H^*(\widetilde{M}) \]
is a surjection.
\end{prop}
\begin{proof}The following facts hold:
\begin{itemize}
    \item The Mukai system $\chi_A \colon \widetilde{K}(A,v) \to |2X|$ specialises to the Hitchin fibration $\chi \circ f \colon \widetilde{M} \to H^0(X,K^{\otimes 2}_X)$. In particular,  a $\chi_A$-ample line bundle on $\widetilde{K}(A,v)$ specialises to a generator of $H^2(\widetilde{M})$. 
    \item The fibre $\chi_A^{-1}(2X)$ consists of 34 irreducible components which specialise to the irreducible components of the nilpotent cone of $\widetilde{M}$, which generate $H^6(\widetilde{M})$; see \cite[Proposition 3.0.3]{Wu2021}. 
    \item Denote by $\Sigma_A$ and $\Sigma$ the singular locus of ${K}(A,v)$ and $M$, isomorphic to $(A \times \hat{A})/\pm 1$ and $(\Aff^2 \times \hat{A})/\pm 1$ respectively. 
    As in \eqref{eq:decompositioncohMtilde}, $H^{*-2}(\Sigma_A)$ is a direct summand of $H^*(K(A,v))$. By definition of $\spmap$ in \cite[(86)]{deCataldoMaulikShen2019}, the restriction of the specialisation map to $H^{*-2}(\Sigma_A)$ is the pullback
    \begin{align*}
        (\Aff^2 \times \hat{A})/\pm 1 \hookrightarrow \PP(T^*\hat{A} &  \oplus  \mathcal{O}_{\hat{A}})/\pm 1 \hookrightarrow (\mathrm{Bl}_{0 \times \hat{A} \times 0}A \times \hat{A} \times \Aff^1)/\pm 1 \\
        & \to (A \times \hat{A} \times \Aff^1)/\pm 1 \to (A \times \hat{A})/\pm 1.
    \end{align*}
    So given the inclusion $j \colon (0 \times \hat{A})/\pm 1 \hookrightarrow (A \times \hat{A})/\pm 1$, we have 
    \[\spmap(\gamma)=j^*\gamma \in H^*( \hat{A}/\pm 1)\simeq H^*( (\Aff^2 \times \hat{A})/\pm 1)\]
    for $\gamma \in H^*((A \times \hat{A})/\pm 1)$, which is a surjection.
\end{itemize}
We conclude that 
\[\mathrm{Im}(\spmap) \supset H^2(\widetilde{M}) \oplus H^6(\widetilde{M}) \oplus H^{*-2}(\Sigma_A).\]
By the description of $H^*(\widetilde{M})$ (cf \cref{fig:my_label}) and relative Hard Lefschetz, this suffices to show that $\mathrm{Im}(\spmap)$ equals the whole $H^*(\widetilde{M})$.
\end{proof}
\end{exa}

\begin{rmk}\label{rmk:smoothvssingular} Recall that for any odd number $d$ the twisted Dolbeault moduli space $M^{\mathrm{tw}}(X, \mathrm{SL}_2, d)$ parametrises semistable $\mathrm{SL}_2$-Higgs bundles of degree $d$ on the curve $X$. It is curious that the analogue of \cref{prop:surjspeci} fails for $M^{\mathrm{tw}}(X, \mathrm{SL}_2,d)$ and $g=2$: there is no degeneration of compact hyperk\"{a}hler manifolds to $M^{\mathrm{tw}}(X, \mathrm{SL}_2, d)$ such that the specialization map $\spmap$ is surjective; see \cite[Proposition 4.3]{deCataldoMaulikShen2020}. 
\end{rmk}

\begin{exa}[genus $>2$] There is no degeneration from $K(A,v)$ with Mukai vector $v=(0,nX, n(g-1))$ with $g>2$ to $\MDolSl$ for dimensional reason. However, $K(A,v)$ and $\MDolSl$ have the same type of singularities: they are stably isosingular in the sense of \cite[Definition 2.6 and Theorem 2.11]{Mauri2021}. Therefore, it is natural to ask the following. \vspace{0.2 cm}
\begin{quote}
\textsc{Question.} \emph{ Does there exist a degeneration of compact symplectic varieties endowed with a Lagrangian fibration in Prym varieties to the Hitchin fibration \[\chi(X,\Sl_n)\colon  \MDolSl \to \bigoplus^{n}_{i=2} H^0(X, K_X^{\otimes i})\] for $g>2$? }
\end{quote}
\vspace{0.2 cm}
Note that the question is answered positively in \cite{SawonShen2021} if we replace the special linear group $\Sl_n$ with the symplectic group $\mathrm{Sp}_n$. 
\end{exa}

\bibliographystyle{plain}
\bibliography{construction}

\begin{thebibliography}{10}

\bibitem{AV21}
E.~Amerik and M.~Verbitsky.
\newblock Contraction centers in families of hyperk\"{a}hler manifolds.
\newblock {\em Selecta Math. (N.S.)}, 27(4):Paper No. 60, 26 pp., 2021.

\bibitem{Beauville84}
A.~Beauville.
\newblock Vari\'{e}t\'{e}s {K}\"{a}hleriennes dont la premi\`ere classe de
  {C}hern est nulle.
\newblock {\em J. Differential Geom.}, 18(4):755--782 (1984), 1983.

\bibitem{BNR89}
A.~Beauville, M.S. Narasimhan, and S.~Ramanan.
\newblock Spectral curves and the generalised theta divisor.
\newblock {\em J. Reine Angew. Math.}, 398:169--179, 1989.

\bibitem{BeilinsonBernsteinDeligne1981}
A.A. Beilinson, J.N. Bernstein, and P.~Deligne.
\newblock Faisceaux pervers.
\newblock In {\em Analysis and topology on singular spaces, {I}
  ({L}uminy,1981)}, volume 100 of {\em Ast\'{e}risque}, pages 5--171. Soc.
  Math. France, Paris, 1982.

\bibitem{BellamySchedler2019}
G.~{Bellamy} and T.~{Schedler}.
\newblock {Symplectic resolutions of character varieties}.
\newblock {\em arXiv:1909.12545}, 2019.

\bibitem{BB73}
A.~Białynicki-Birula.
\newblock Some theorems on actions of algebraic groups.
\newblock {\em Ann. of Math. (2)}, 98:480--497, 1973.

\bibitem{BMT2011}
E.~Bierstone, P.D. Milman, and M.~Temkin.
\newblock {$\Bbb Q$}-universal desingularization.
\newblock {\em Asian J. Math.}, 15(2):229--249, 2011.

\bibitem{Biswas05}
I.~Biswas.
\newblock Orbifold principal bundles on an elliptic fibration and parabolic
  principal bundles on a {R}iemann surface. {II}.
\newblock {\em Collect. Math.}, 56(3):235--252, 2005.

\bibitem{BIM13}
I.~Biswas, S.~Majumder, and M.L. Wong.
\newblock Parabolic {H}iggs bundles and {$\Gamma$}-{H}iggs bundles.
\newblock {\em J. Aust. Math. Soc.}, 95(3):315--328, 2013.

\bibitem{CPC19}
I.~Cheltsov, V.~Przyjalkowski, and C.~Shramov.
\newblock Which quartic double solids are rational?
\newblock {\em J. Algebraic Geom.}, 28(2):201--243, 2019.

\bibitem{CHS2020}
S.M. {Chiarello}, T.~{Hausel}, and A.~{Szenes}.
\newblock {An Enumerative Approach to $P=W$}.
\newblock {\em arXiv:2002.08929}, 2020.

\bibitem{DaskalopoulosUhlenbeck1995}
G.D. Daskalopoulos and K.K. Uhlenbeck.
\newblock An application of transversality to the topology of the moduli space
  of stable bundles.
\newblock {\em Topology}, 34(1):203--215, 1995.

\bibitem{DaskalopoulosWentworth10}
G.D. Daskalopoulos, R.A. Wentworth, and G.~Wilkin.
\newblock Cohomology of {${\rm SL}(2,\Bbb C)$} character varieties of surface
  groups and the action of the {T}orelli group.
\newblock {\em Asian J. Math.}, 14(3):359--383, 2010.

\bibitem{deCataldo2020}
M.A. de~Cataldo.
\newblock {Projective Compactification of Dolbeault Moduli Spaces}.
\newblock {\em Int. Math. Res. Not.}, 2020.

\bibitem{deCataldoHauselMigliorini2012}
M.A. {de}~{Cataldo}, T.~Hausel, and L.~Migliorini.
\newblock Topology of {H}itchin systems and {H}odge theory of character
  varieties: the case {$A_1$}.
\newblock {\em Ann. of Math. (2)}, 175(3):1329--1407, 2012.

\bibitem{deCataldoHauselMigliorini2013}
M.A. {de}~{Cataldo}, T.~Hausel, and L.~Migliorini.
\newblock Exchange between perverse and weight filtration for the {H}ilbert
  schemes of points of two surfaces.
\newblock {\em J. Singul.}, 7:23--38, 2013.

\bibitem{deCataldoMaulik2018}
M.A. de~Cataldo and D.~Maulik.
\newblock The perverse filtration for the {H}itchin fibration is locally
  constant.
\newblock {\em Pure Appl. Math. Q.}, 16(5):1441--1464, 2020.

\bibitem{deCataldoMaulikShen2020}
M.A. {de Cataldo}, D.~{Maulik}, and J.~{Shen}.
\newblock {On the P=W conjecture for $\mathrm{SL}_n$}.
\newblock {\em arXiv:2002.03336}, 2020.

\bibitem{deCataldoMaulikShen2019}
M.A. de~Cataldo, D.~Maulik, and J.~{Shen}.
\newblock Hitchin fibrations, abelian surfaces, and the {P=W} conjecture.
\newblock {\em Jour. Amer. Math. Soc.}, pages 1--47, 2021.

\bibitem{deCataldoMigliorini05}
M.A. de~Cataldo and L.~Migliorini.
\newblock The {H}odge theory of algebraic maps.
\newblock {\em Ann. Sci. \'{E}cole Norm. Sup. (4)}, 38(5):693--750, 2005.

\bibitem{deCataldoMigliorini2010}
M.A. de~Cataldo and L.~Migliorini.
\newblock The perverse filtration and the {L}efschetz hyperplane theorem.
\newblock {\em Ann. of Math. (2)}, 171(3):2089--2113, 2010.

\bibitem{Deligne74}
P.~Deligne.
\newblock Th\'{e}orie de {H}odge. {III}.
\newblock {\em Inst. Hautes \'{E}tudes Sci. Publ. Math.}, (44):5--77, 1974.

\bibitem{DEL97}
R.~Donagi, L.~Ein, and R.~Lazarsfeld.
\newblock Nilpotent cones and sheaves on {$K3$} surfaces.
\newblock In {\em Birational algebraic geometry ({B}altimore, {MD}, 1996)},
  volume 207 of {\em Contemp. Math.}, pages 51--61. Amer. Math. Soc.,
  Providence, RI, 1997.

\bibitem{DN89}
J.M. Drezet and M.S. Narasimhan.
\newblock Groupe de {P}icard des vari\'{e}t\'{e}s de modules de fibr\'{e}s
  semi-stables sur les courbes alg\'{e}briques.
\newblock {\em Invent. Math.}, 97(1):53--94, 1989.

\bibitem{Durfee95}
A.H. Durfee.
\newblock Intersection homology {B}etti numbers.
\newblock {\em Proc. Amer. Math. Soc.}, 123(4):989--993, 1995.

\bibitem{Felisetti2018}
C.~{Felisetti}.
\newblock {Intersection cohomology of the moduli space of Higgs bundles on a
  genus 2 curve}.
\newblock {\em J. Inst. Math. Jussieu}, pages 1--50, 2021.

\bibitem{FelisettiShenYin2021}
C.~Felisetti, J.~Shen, and Q.~Yin.
\newblock On intersection cohomology and {L}agrangian fibrations of irreducible
  symplectic varieties.
\newblock {\em Trans. Amer. Math. Soc.}, 375(4):2987--3001, 2022.

\bibitem{FGPN14}
E.~Franco, O.~Garcia-Prada, and P.E. Newstead.
\newblock Higgs bundles over elliptic curves.
\newblock {\em Illinois J. Math.}, 58(1):43--96, 2014.

\bibitem{FuNamikawa04}
B.~Fu and Y.~Namikawa.
\newblock Uniqueness of crepant resolutions and symplectic singularities.
\newblock {\em Ann. Inst. Fourier (Grenoble)}, 54(1):1--19, 2004.

\bibitem{GM80}
M.~Goresky and R.~MacPherson.
\newblock Intersection homology theory.
\newblock {\em Topology}, 19(2):135--162, 1980.

\bibitem{GM81}
M.~Goresky and R.~MacPherson.
\newblock Morse theory and intersection homology theory.
\newblock In {\em Analysis and topology on singular spaces, {II}, {III}
  ({L}uminy, 1981)}, volume 101 of {\em Ast\'{e}risque}, pages 135--192. Soc.
  Math. France, Paris, 1983.

\bibitem{GottscheSoergel93}
L.~G\"{o}ttsche and W.~Soergel.
\newblock Perverse sheaves and the cohomology of {H}ilbert schemes of smooth
  algebraic surfaces.
\newblock {\em Math. Ann.}, 296(2):235--245, 1993.

\bibitem{Graf}
P.~Graf and M.~Schwald.
\newblock On the {K}odaira problem for uniruled {K}\"{a}hler spaces.
\newblock {\em Ark. Mat.}, 58(2):267--284, 2020.

\bibitem{GKK10}
D.~Greb, S.~Kebekus, and S.~Kov\'{a}cs.
\newblock Extension theorems for differential forms and {B}ogomolov-{S}ommese
  vanishing on log canonical varieties.
\newblock {\em Compos. Math.}, 146(1):193--219, 2010.

\bibitem{Groechenig14}
M.~Groechenig.
\newblock Hilbert schemes as moduli of {H}iggs bundles and local systems.
\newblock {\em Int. Math. Res. Not.}, (23):6523--6575, 2014.

\bibitem{Grothendieck_1957}
A.~Grothendieck.
\newblock Sur quelques points d'alg\`ebre homologique.
\newblock {\em Tohoku Math. J. (2)}, 9:119--221, 1957.

\bibitem{Harder2019}
A.~{Harder}.
\newblock {Torus fibers and the weight filtration}.
\newblock {\em arXiv:1908.05110}, 2019.

\bibitem{HLSY2019}
A.~{Harder}, L.~{Zhiyuan}, J.~{Shen}, and Q.~{Yin}.
\newblock {P=W for Lagrangian fibrations and degenerations of hyper-K{\"a}hler
  manifolds}.
\newblock {\em Forum Math. Sigma}, 9:e50, 2021.

\bibitem{Hausel98}
T.~Hausel.
\newblock Vanishing of intersection numbers on the moduli space of {H}iggs
  bundles.
\newblock {\em Adv. Theor. Math. Phys.}, 2(5):1011--1040, 1998.

\bibitem{HauselLetellierRodriguez-Villegas2011}
T.~Hausel, E.~Letellier, and F.~Rodriguez-Villegas.
\newblock Arithmetic harmonic analysis on character and quiver varieties.
\newblock {\em Duke Math. J.}, 160(2):323--400, 2011.

\bibitem{HauselRodriguez-Villegas2008}
T.~Hausel and F.~Rodriguez-Villegas.
\newblock Mixed {H}odge polynomials of character varieties.
\newblock {\em Invent. Math.}, 174(3):555--624, 2008.
\newblock With an appendix by Nicholas M. Katz.

\bibitem{HauselVillegas15}
T.~Hausel and F.~Rodriguez-Villegas.
\newblock Cohomology of large semiprojective hyperk\"{a}hler varieties.
\newblock {\em Ast\'{e}risque}, (370):113--156, 2015.

\bibitem{HauselThaddeus03}
T.~Hausel and M.~Thaddeus.
\newblock Mirror symmetry, {L}anglands duality, and the {H}itchin system.
\newblock {\em Invent. Math.}, 153(1):197--229, 2003.

\bibitem{HauselThaddeus04}
T.~Hausel and M.~Thaddeus.
\newblock Generators for the cohomology ring of the moduli space of rank 2
  {H}iggs bundles.
\newblock {\em Proc. London Math. Soc. (3)}, 88(3):632--658, 2004.

\bibitem{HeuLoray19}
V.~Heu and F.~Loray.
\newblock Flat rank two vector bundles on genus two curves.
\newblock {\em Mem. Amer. Math. Soc.}, 259(1247):v+103, 2019.

\bibitem{Hitchin1987}
N.J. Hitchin.
\newblock The self-duality equations on a {R}iemann surface.
\newblock {\em Proc. London Math. Soc. (3)}, 55(1):59--126, 1987.

\bibitem{Kaledin06}
D.~Kaledin.
\newblock Symplectic singularities from the {P}oisson point of view.
\newblock {\em J. Reine Angew. Math.}, 600:135--156, 2006.

\bibitem{HKP2019}
L.~Katzarkov, V.~V. Przyjalkowski, and A.~Harder.
\newblock {${\rm P}={\rm W}$} {P}henomena.
\newblock {\em Mat. Zametki}, 108(1):33--46, 2020.

\bibitem{KiemYoo08}
Y.-H. Kiem and S.-B. Yoo.
\newblock The stringy {$E$}-function of the moduli space of {H}iggs bundles
  with trivial determinant.
\newblock {\em Math. Nachr.}, 281(6):817--838, 2008.

\bibitem{KW2006}
F.~Kirwan and J.~Woolf.
\newblock {\em An introduction to intersection homology theory}.
\newblock Chapman \& Hall/CRC, Boca Raton, FL, second edition, 2006.

\bibitem{Kollar2007}
J.~Koll\'{a}r.
\newblock {\em Lectures on resolution of singularities}, volume 166 of {\em
  Annals of Mathematics Studies}.
\newblock Princeton University Press, Princeton, NJ, 2007.

\bibitem{K13}
J.~Koll\'{a}r.
\newblock {\em Singularities of the minimal model program}, volume 200 of {\em
  Cambridge Tracts in Mathematics}.
\newblock Cambridge University Press, Cambridge, 2013.

\bibitem{KollarMori1998}
J.~Koll{\'a}r and S.~Mori.
\newblock {\em Birational geometry of algebraic varieties}, volume 134 of {\em
  Cambridge Tracts in Mathematics}.
\newblock Cambridge University Press, Cambridge, 1998.

\bibitem{Kumar00}
C.~Kumar.
\newblock Invariant vector bundles of rank 2 on hyperelliptic curves.
\newblock {\em Michigan Math. J.}, 47(3):575--584, 2000.

\bibitem{Lazarsfeld2004}
R.~Lazarsfeld.
\newblock {\em Positivity in algebraic geometry. {I}}, volume~48 of {\em
  Ergebnisse der Mathematik und ihrer Grenzgebiete. 3. Folge. A Series of
  Modern Surveys in Mathematics}.
\newblock Springer-Verlag, Berlin, 2004.

\bibitem{LehnSorger2006}
M.~Lehn and C.~Sorger.
\newblock La singularit\'{e} de {O}'{G}rady.
\newblock {\em J. Algebraic Geom.}, 15(4):753--770, 2006.

\bibitem{LogaresMunozNewstead2013}
M.~Logares, V.~Mu\~{n}oz, and P.~Newstead.
\newblock Hodge polynomials of {${\rm SL}(2,\Bbb{C})$}-character varieties for
  curves of small genus.
\newblock {\em Rev. Mat. Complut.}, 26(2):635--703, 2013.

\bibitem{MaulikOkunkov19}
D.~Maulik and A.~Okounkov.
\newblock Quantum groups and quantum cohomology.
\newblock {\em Ast\'{e}risque}, (408):ix+209, 2019.

\bibitem{Mauri2021}
M.~Mauri.
\newblock Intersection cohomology of rank 2 character varieties of surface
  groups.
\newblock {\em J. Inst. Math. Jussieu}, page 1–40, 2021.

\bibitem{MMS}
M.~Mauri, E.~Mazzon, and M.~Stevenson.
\newblock On the geometric {P}={W} conjecture.
\newblock {\em Selecta Math. (N.S.)}, 28(3):Paper No. 65, 2022.

\bibitem{Mellit2019}
A.~{Mellit}.
\newblock Cell decompositions of character varieties.
\newblock {\em arXiv:1905.10685}, 2019.

\bibitem{MRS18}
G.~Mongardi, A.~Rapagnetta, and G.~Sacc\`a.
\newblock The {H}odge diamond of {O}'{G}rady's six-dimensional example.
\newblock {\em Compos. Math.}, 154(5):984--1013, 2018.

\bibitem{MumfordFogartyKirwan1994}
D.~Mumford, J.~Fogarty, and F.~Kirwan.
\newblock {\em Geometric invariant theory}, volume~34 of {\em Ergebnisse der
  Mathematik und ihrer Grenzgebiete (2)}.
\newblock Springer-Verlag, Berlin, third edition, 1994.

\bibitem{NarasimhanRamanan69}
M.S. Narasimhan and S.~Ramanan.
\newblock Moduli of vector bundles on a compact {R}iemann surface.
\newblock {\em Ann. of Math. (2)}, 89:14--51, 1969.

\bibitem{NR69}
M.S. Narasimhan and S.~Ramanan.
\newblock Vector bundles on curves.
\newblock In {\em Algebraic {G}eometry ({I}nternat. {C}olloq., {T}ata {I}nst.
  {F}und. {R}es., {B}ombay, 1968)}, pages 335--346. Oxford Univ. Press, London,
  1969.

\bibitem{NemethiSzabo2020}
A.~{N{\'e}methi} and S.~{Szab{\'o}}.
\newblock {The Geometric P=W Conjecture in the Painlevé Cases via Plumbing
  Calculus}.
\newblock {\em Int. Math. Res. Not.}, 09 2020.

\bibitem{Nitsure91}
N.~Nitsure.
\newblock Moduli space of semistable pairs on a curve.
\newblock {\em Proc. London Math. Soc. (3)}, 62(2):275--300, 1991.

\bibitem{PalPauly2018}
S.~{Pal} and C.~{Pauly}.
\newblock The wobbly divisors of the moduli space of rank-2 vector bundles.
\newblock {\em Adv. Geom.}, 21(4):473--482, 2021.

\bibitem{PeregoRapagnetta13}
A.~Perego and A.~Rapagnetta.
\newblock Deformation of the {O}'{G}rady moduli spaces.
\newblock {\em J. Reine Angew. Math.}, 678:1--34, 2013.

\bibitem{PeregoRapagnetta18}
A.~Perego and A.~Rapagnetta.
\newblock The moduli spaces of sheaves on {K3} surfaces are irreducible
  symplectic varieties.
\newblock {\em arXiv:1802.01182}, 2018.

\bibitem{PetersSteenbrink2008}
C.~Peters and J.~Steenbrink.
\newblock {\em Mixed {H}odge structures}, volume~52 of {\em Ergebnisse der
  Mathematik und ihrer Grenzgebiete. 3. Folge. A Series of Modern Surveys in
  Mathematics}.
\newblock Springer-Verlag, Berlin, 2008.

\bibitem{SawonShen2021}
J.~{Sawon} and C.~{Shen}.
\newblock {Deformations of compact Prym fibrations to Hitchin systems}.
\newblock {\em arXiv:2103.04274}, 2021.

\bibitem{Seshadri70}
C.S. Seshadri.
\newblock Moduli of {$\pi $}-vector bundles over an algebraic curve.
\newblock In {\em Questions on {A}lgebraic {V}arieties ({C}.{I}.{M}.{E}., {III}
  {C}iclo, {V}arenna, 1969)}, pages 139--260. Edizioni Cremonese, Rome, 1970.

\bibitem{ShenZhang2018}
J.~{Shen} and Z.~{Zhang}.
\newblock {Perverse filtrations, Hilbert schemes, and the P=W conjecture for
  parabolic Higgs bundles}.
\newblock {\em Alg. Geom.}, 8(4):465--489, 2021.

\bibitem{ShenYin2018}
Junliang Shen and Qizheng Yin.
\newblock Topology of {L}agrangian fibrations and {H}odge theory of
  hyper-{K}\"{a}hler manifolds.
\newblock {\em Duke Math. J.}, 171(1):209--241, 2022.

\bibitem{Simpson1990}
C.T. Simpson.
\newblock Harmonic bundles on noncompact curves.
\newblock {\em J. Amer. Math. Soc.}, 3(3):713--770, 1990.

\bibitem{Simpson92}
C.T. Simpson.
\newblock Higgs bundles and local systems.
\newblock {\em Inst. Hautes \'{E}tudes Sci. Publ. Math.}, (75):5--95, 1992.

\bibitem{Simpson1994I}
C.T. Simpson.
\newblock Moduli of representations of the fundamental group of a smooth
  projective variety. {I}.
\newblock {\em Inst. Hautes \'{E}tudes Sci. Publ. Math.}, (79):47--129, 1994.

\bibitem{Simpson1994}
C.T. Simpson.
\newblock Moduli of representations of the fundamental group of a smooth
  projective variety. {II}.
\newblock {\em Inst. Hautes \'Etudes Sci. Publ. Math.}, (80):5--79 (1995),
  1994.

\bibitem{Simpson1997}
C.T. Simpson.
\newblock The {H}odge filtration on nonabelian cohomology.
\newblock In {\em Algebraic geometry---{S}anta {C}ruz 1995}, volume~62 of {\em
  Proc. Sympos. Pure Math.}, pages 217--281. Amer. Math. Soc., Providence, RI,
  1997.

\bibitem{Szabo2018}
S.~{Szab{\'o}}.
\newblock {Perversity equals weight for Painlev{\'e} spaces}.
\newblock {\em Adv. Math.}, 383:Paper No. 107667, 45, 2021.

\bibitem{Szabo2019}
S.~{Szabó}.
\newblock {Simpson's geometric P=W conjecture in the Painlev{\'e} VI case via
  abelianization}.
\newblock {\em arXiv:1906.01856}, 2019.

\bibitem{Temkin2012}
M.~Temkin.
\newblock Functorial desingularization of quasi-excellent schemes in
  characteristic zero: the nonembedded case.
\newblock {\em Duke Math. J.}, 161(11):2207--2254, 2012.

\bibitem{Thaddeus89}
M.~Thaddeus.
\newblock Topology of the moduli space of stable bundles on a {R}iemann
  surface.
\newblock {\em Master's Thesis}, 1989.

\bibitem{Weber17}
A.~Weber.
\newblock Hirzebruch class and {B}ia\l ynicki-{B}irula decomposition.
\newblock {\em Transform. Groups}, 22(2):537--557, 2017.

\bibitem{Williamson20}
G.~Williamson.
\newblock Modular representations and reflection subgroups.
\newblock {\em Current Developments in Mathematics}, 2019.

\bibitem{Wu2021}
B.~{Wu}.
\newblock {Hodge Numbers of O'Grady 6 Via Ng{\^o} Strings}.
\newblock {\em arXiv:2105.08545}, 2021.

\bibitem{Yoshioka2001}
K.~Yoshioka.
\newblock Moduli spaces of stable sheaves on abelian surfaces.
\newblock {\em Math. Ann.}, 321(4):817--884, 2001.

\bibitem{ZhangZili2019}
Z.~{Zhang}.
\newblock {The P=W identity for cluster varieties}.
\newblock {\em Math. Res. Lett.}, 28(3):925--944, 2021.

\end{thebibliography}
\end{document}